\documentclass[11pt]{amsart}
\usepackage{amsmath}
\usepackage{amsfonts}
\usepackage{amssymb}
\usepackage{amsthm}
\usepackage{url}
\usepackage[all]{xy}
\usepackage{dsfont}
\usepackage{graphicx}
\usepackage{caption}
\usepackage{subcaption}
\usepackage{comment} 
\usepackage{stmaryrd}
\usepackage{hyperref}
\usepackage{todonotes}
\usepackage{color}
\usepackage{enumerate,tikz-cd}
\usepackage[backend=biber,style=alphabetic,url=false,doi=false,isbn=false,sorting=nyt,maxbibnames=9,giveninits=true]{biblatex}
\usepackage{MnSymbol}
\usepackage{comment}
\usepackage{mathrsfs}
\allowdisplaybreaks
\usepackage{geometry}
\numberwithin{equation}{section}

\newtheorem{proposition}{Proposition}[section]

\newtheorem{lemma}[proposition]{Lemma}
\newtheorem{theorem}[proposition]{Theorem}
\newtheorem{corollary}[proposition]{Corollary}

\theoremstyle{definition}
\newtheorem{remark}[proposition]{Remark}
\newtheorem{definition}[proposition]{Definition}
\newtheorem{example}[proposition]{Example}

\newtheorem{question}{Question}[section]

\DeclareMathOperator{\Proj}{Proj}

\DeclareMathOperator{\Supp}{Supp}
\DeclareMathOperator{\Sing}{Sing}

\DeclareMathOperator{\codim}{codim}

\DeclareMathOperator{\rk}{rk}

\DeclareMathOperator{\DF}{DF}

\DeclareMathOperator{\lct}{lct}
\DeclareMathOperator{\Ding}{Ding}

\DeclareMathOperator{\ord}{ord}

\DeclareMathOperator{\Val}{Val}

\newcommand{\R}{\mathbb{R}}
\newcommand{\C}{\mathbb{C}}
\newcommand{\A}{\mathbb{A}}
\newcommand{\Z}{\mathbb{Z}}

\newcommand{\Q}{\mathbb{Q}}
\newcommand{\G}{\mathbb{G}}

\newcommand{\pr}{\mathbb{P}}
\newcommand{\PP}{\mathbb{P}}
\newcommand{\QQ}{\mathbb{Q}}
\newcommand{\OO}{\mathcal{O}}

\renewcommand{\epsilon}{\varepsilon}

\newcommand{\M}{\mathcal{M}}

\newcommand{\F}{\mathcal{F}}

\renewcommand{\L}{\mathcal{L}}
\newcommand{\X}{\mathcal{X}}
\newcommand{\Y}{\mathcal{Y}}

\renewcommand{\phi}{\varphi}

\newcommand\vol{\mathrm{vol}}

\newcommand\cA{\mathcal{A}}

\newcommand\cF{\mathcal{F}}
\newcommand\cG{\mathcal{G}}

\newcommand\cO{\mathcal{O}}

\newcommand\fa{\mathfrak{a}}

\newcommand\NA{^{\mathrm{NA}}}

\newcommand\red{{\mathrm{red}}}

\title{K-stability of adjoint foliated structures}

\author[T. S. Papazachariou]{Theodoros Stylianos Papazachariou}
\address{Yau Mathematical Sciences Center, Jingzhai, Tsinghua University, Haidian District, Beijing, China.}
\email{tpapazachariou@mail.tsinghua.edu.cn}

\begin{document}

\begin{abstract}
We introduce a notion of K-stability for adjoint foliated structures via test configurations and the foliated Donaldson--Futaki invariant. We prove reduction to special test configurations for adjoint Fano foliated structures by showing that the mixed Donaldson-–Futaki invariant is non-increasing along the birational procedure. We also introduce a notion of Ding stability for adjoint Fano foliated structures which we show is equivalent to our notion of K-stability. We then introduce mixed $\alpha$, $\beta$ and $\delta$-invariants and use the reduction theorem to establish valuative criteria for the K-stability of adjoint Fano foliated structures. To conclude, as an application, we show that K-semistable adjoint Fano foliated structures with bounded volume form a bounded family.
\end{abstract}

\maketitle

\setcounter{tocdepth}{1}
{\hypersetup{hidelinks}
\tableofcontents}

\section{Introduction}
The theory of K-stability of Fano varieties has found remarkable success in recent years, both in determining the existence of K{\"a}hler-Einstein metrics \cite{CDS2013, LXZ22} but also in constructing a full, complete moduli theory for klt Fano varieties (see \cite{Xu2025} and citations therein for an excellent survey). Both the former, in the singular case, and especially the latter, have relied on recent key developments and results in birational geometry, and especially the Minimal Model Programme (MMP) for projective varieties \cite{BCHM10,HMX13,HMX14,Birkar21,Birkar19}.

In parallel, in recent years, the focus of many birational geometers has been the extension of these MMP techniques to (algebraically integrable) foliations, i.e. saturated subsheaves $T_\F\subset T_{X}$ closed under Lie bracket. In the specific case of adjoint foliated structures, i.e. triples $(X,\F,t)$ with ``canonical'' divisor $(1-t)K_{\X}+tK_{\F}$, recent results have shown that key results of the MMP, such as finite generation, existence and termination of flips, and even an analogue of the BAB conjecture, hold assuming mild singularity conditions \cite{CHLMSSX24, CHLMSX25}. The existence of these MMP techniques, as well as the massive developments in the theory of K-stability for Fano varieties, prompts the following, very natural question:

\begin{question}\label{quest: quest intro}
    Does there exist an analogue of K-stability for adjoint Fano foliation structures, or Fano foliations?
\end{question}

In this note, we answer Question \ref{quest: quest intro} affirmatively, by introducing a K-stability notion for adjoint foliated structures via test configurations. We also prove some key results in the case of adjoint Fano foliated structures such as reduction to special test configurations, existence of a valuative criterion via prime divisors, and boundedness. These results parallel the key results of \cite{LX14, Fujita16,BJ20, Jiang2020} for K-stability of Fano varieties.

More precisely, we introduce the notion of $\F$-compatible test configurations. For a foliated polarised variety $(X,\F,L)$ these are usual (normal) test configurations $(\X,\L)\rightarrow\A^1$ on $(X,L)$ with the added condition that there exists a $\G_m$-equivariant (algebraically integrable) foliation $T_{\F_\X}\subset T_{\X/\A^1}$, that restricts to the foliation $\F$ in the general fibre of the test configuration (see more details in Definition \ref{def: foliated t.c.} and Section \ref{sec: K-stability def}). This notion allows us to define a foliated Donaldson--Futaki invariant, intersection-theoretically
\[
\DF^{[t]}(\X,\F_{\X},\L)
:=
\frac{1}{V}\left(
\frac{n}{n+1}\,\mu(X,\F,L)\bar{\L}^{n+1}
+
K^{[t]}_{\bar{\X}/\mathbb P^1}\cdot \bar{\L}^n
\right).
\]

The positivity of this invariant allows us to define the notion of $t$-K-stability which is the key notion of this paper (see also Definition \ref{def: k-stability definition} for more details).

\begin{definition}\label{def: k-stability definition intro}
Let $(X,\F,L)$ be a polarised foliated variety.

\begin{enumerate}
\item $(X,\F,L)$ is \emph{$t$-K-semistable} if $\DF^{[t]}(\X,\F_{\X},\L)\ge0$ for all normal $\F$-compatible test configurations.
\item $(X,\F,L)$ is \emph{$t$-K-stable}
if the above inequality is strict for all non-trivial $\F$-compatible test configurations.
\item $(X,\F,L)$ is \emph{$t$-K-polystable} if it is $t$-K-semistable and \[
\DF^{[t]}(\X,\F_{\X},\L)= 0 \text{ if and only if }(\X,\F_{\X}, \L) \text{ is $\F$-compatible of product type.}
\] 
\end{enumerate}
\end{definition}

Our next step is to investigate some of the properties of this stability notion. We focus on adjoint Fano foliated structures, i.e. triples $(X,\F,t)$ where the divisor $-K_{X,\F}^{[t]}: = -tK_X-(1-t)K_\F$ is ample. The study of these structures in the context of our notion of $t$-K-semistability is natural for two reasons. First, by fixing $t$ and setting $L\sim_{\Q} -K_{X,\F}^{[t]}$ we land naturally in the regime of (polarised) adjoint Fano foliated structures. The second is that due to recent achievements \cite{CHLMSSX24,CHLMSX25}, we know that the MMP for such structures exists and terminates \cite[Theorem A]{CHLMSX25}, and that adjoint Fano foliated structures with mild singularity conditions are bounded \cite[Theorem B]{CHLMSX25}. As such, the expectation is that one can replicate techniques developed for the study of K-stability of Fano varieties depending on the MMP to the setting of $t$-K-stability of adjoint Fano foliated structures.

We demonstrate this by adapting the road map developed in \cite{LX14} to show that when studying $t$-K-stability of adjoint Fano foliated structures we can reduce to  $\F$-compatible special test configurations. In particular, we prove the following.

\begin{theorem}[See Theorem \ref{thm:reduction-special-mixed-DF}]
\label{thm:reduction-special-mixed-DF intro}
Let $(X,\F)$ be an adjoint Fano foliated structure, and $(\pi\colon \X\to \A^1,\L,\F_{\X})$ be a normal $\F$-compatible test configuration for $(X,\F,L)$.
Then there exists a normal $\F$-compatible special test configuration $(\X^{\mathrm{sp}},\F_{\X^{\mathrm{sp}}}, \L^{\mathrm{sp}})$ for $(X,\F,L)$, such that
\[
\,\DF^{[t]}(\X,\F_{\X},\L_{\X})\ge
\DF^{[t]}(\X^{\mathrm{sp}},\F_{\X^{\mathrm{sp}}}, \L^{\mathrm{sp}}).
\]
\end{theorem}

The method of proof for Theorem \ref{thm:reduction-special-mixed-DF intro} is as follows. Starting from a normal $\F$-compatible test configuration, after a finite base change (where $m$ is the degree of the extension) we construct a compactification over $\PP^1$ and produce a log canonical model $\X^{0}$ of the total space. We then run a $\G_m$-equivariant MMP with scaling for the adjoint divisor $K^{[t]}_{\X/\PP^1}$, using the techniques of \cite{CHLMSSX24, CHLMSX25}. The end-point of this MMP is an anti-adjoint model $\X^{ac}$ on which the polarisation is relatively
anti-canonical. From this model, we extract a special fibre of discrepancy zero and pass to the relative anti-adjoint model, producing the special $\F$-compatible test configuration $(\X^{\mathrm{sp}},\F_{\X^{\mathrm{sp}}}, \L^{\mathrm{sp}})$. For more details on this construction see the beginning of Section \ref{sec: special t.c.} for a schematic description. This proves the first part of the Theorem. For the latter part, we show that the mixed
Donaldson--Futaki invariant is non-increasing along each birational step performed in the MMP procedure (see Section \ref{sec: reduction to special t.c.} for the technical details). This implies that the resulting special test configuration has no larger mixed Donaldson--Futaki invariant than the original one, proving the second part of the Theorem.

Our next step is to define a $t$-Ding invariant and the notion of $t$-Ding stability for adjoint Fano foliated structures. Although we will not go into specifics about the definition here (see Section \ref{sec: def ding inv}), following the above procedure we also show that the $t$-Ding invariant decreases along the birational steps detailed above, giving us an identical theorem to Theorem \ref{thm:reduction-special-mixed-DF intro} for the $t$-Ding invariant (see Theorem \ref{thm:t-Ding-special-reduction}). Combining this with the earlier reduction theorem we obtain the following equivalence.

\begin{theorem}[See Corollary \ref{cor:t-K-vs-t-Ding}]\label{thm: intro ding vs K}
    Let $(X,\F,t)$ be an adjoint Fano foliated structure and set $L\sim_{\Q} -K^{[t]}_{X,\F}$. Then the following are equivalent 
\begin{enumerate}
\item
$(X,\F,t)$ is $t$-K-semistable (respectively uniformly $t$-K-stable);
\item
$(X,\F,t)$ is $t$-Ding-semistable (respectively uniformly $t$-Ding-stable).
\end{enumerate}
\end{theorem}

Using this result together with an adaptation of the argument in \cite[Theorem 4.1]{Fujita16} we are also able to obtain the following key Theorem, which provides a valuative criterion for $t$-K-stability of adjoint Fano foliated structures (with $L\sim_{\Q} -K_{X,\F}^{[t]}$).

\begin{theorem}[See Theorem \ref{thm:main-prime-divisor-criterion}]
\label{thm:dreamy-valuative-criterion intro 1}
$(X,\F,L)$ is $t$-K-semistable (respectively uniformly $t$-K-stable) if and only if $\beta^{[t]}_{X,\F,L}(v)\ge 0$ (respectively $>\epsilon\cdot j(v)$ for $0<\epsilon$) for every divisorial valuation $v$.
\end{theorem}

Here $\beta^{[t]}$ is the mixed $\beta$-invariant, 
\[
\beta^{[t]}_{X,\F,L}(v)
:=
A^{[t]}_{X,\F}(v)-S_L(v).
\]

The existence of a $\beta$-invariant for $t$-K-stability of adjoint Fano foliated structures prompts us to define mixed $\delta$ and $\alpha$-invariants, via the mixed log canonical threshold and basis type divisors, as well as a mixed normalised volume. We show (see Theorem \ref{thm:mixed-theorem-4-4}) that these can be described as follows
\[\delta^{[t]} := \inf_{v\in \Val^{*}_X}\frac{A^{[t]}_{X,\F}(v)}{S(v)},\quad \alpha^{[t]}:=
\inf_{v\in \Val_X^*}\frac{A^{[t]}_{X,\F}(v)}{T_L(v)}, \quad \widehat{\mathrm{vol}}^{[t]}_{X,\F}(v):=
A^{[t]}_{X,\F}(v)^n\,\mathrm{vol}(v)\]
following the approach in \cite{BJ20}. We are then able to show the following alternative valuative description of $t$-K-stability of adjoint Fano foliated structures.

\begin{theorem}[See Theorem \ref{thm:dreamy-valuative-criterion} and Corollary \ref{cor: k-ss to alpha}]\label{thm: delta alpha thm intro}
    The adjoint Fano foliated structure $(X,\F,t)$ is $t$-K-semistable (uniformly $t$-K-stable) if and only if $\delta^{[t]}(X,\F;L)\ge 1$ (respectively $>1$). If $(X,\F,t)$ is $t$-K-semistable then $\alpha^{[t]}(X,\F;L)\ge \frac{1}{n+1}$.
\end{theorem}

The boundedness of the $\alpha^{[t]}(X,\F;L)$ invariant from below for $t$-K-semistable adjoint Fano foliated structures, along with \cite[Theorem B]{CHLMSX25} allows us to prove that $t$-K-semistable adjoint Fano foliated structures are bounded.

\begin{theorem}[See Corollary \ref{cor: boundedness for K-semi}]\label{thm: boundedness for K-semi intro}
    Fix a positive integer $d$ and a real number $\delta > 0$. Then the set of $d$-dimensional $t$-K-semistable adjoint Fano foliated structures with $(-K_{X,\F}^{[t]})^d > \delta$ forms a bounded family.
\end{theorem}

Here, boundedness refers to the birational geometric definition that requires that after fixing the numerical data, the relevant adjoint Fano foliated structures are contained in finitely many algebraic families (see Definition \ref{def: boundedness}).

In order to prove this Theorem we follow the strategy in \cite{Jiang2020}, refined in \cite[\S 7.2]{Xu2025}. We first show that bounding the mixed $\alpha$-invariant from below shows that the adjoint Fano foliated structure $(X,\F,t)$ is $\epsilon$-lc for a suitably chosen $\epsilon$. This, combined with \cite[Proposition 9.1]{CHLMSX25} allows us to verify that all of the conditions in \cite[Theorem B]{CHLMSX25} are satisfied. Combined with the bound of the mixed $\alpha$-invariant for $t$-K-semistable adjoint Fano foliated structures we have in Theorem \ref{thm: delta alpha thm intro}, we obtain the desired boundedness statement. We note that in \cite[Theorem E]{CLSV26} the authors prove a similar statement. The difference, however, is that they although they prove the result in terms of bounding the $\alpha$ invariant, they do not associate a K-stability notion to their theorem. In addition, the method of proof is different, as it uses estimates on minimal log discrepancies, while we use estimates using divisorial valuations. We also note that \cite[Theorem E]{CLSV26} can also be deduced from Theorem \ref{thm: boundedness for K-semi intro}.

The existence of these ``nice'' properties for $t$-K-stability of adjoint Fano foliated structures (reduction to special test configurations, existence of MMP techniques, boundedness) prompts us to ask the following natural question regarding the existence of a moduli theory for such structures.

\begin{question}\label{conj: moduli stack intro}
    Does there exist a reduced Artin stack of finite type over $\C$ $\M^{K}_{n,V,t}$ whose $\C$-points parametrise $t$-K-semistable adjoint Fano foliation structures of dimension $n$ and volume $V$?  Does this stack admit a good moduli space $\bar{M}^{K}_{n,V,t}$, which is a reduced projective scheme of finite type over $\C$, whose $\C$-points parameterize $t$-K-polystable adjoint Fano foliation structures?
\end{question}

The road map in proving such a statement is now well established, with Theorem \ref{thm: boundedness for K-semi intro} being one of the key elements needed for the proof. The other is openness of the $t$-K-semistable locus, which seems out of reach with the current MMP methods. These two will provide the existence of the quotient moduli stack $\M^{K}_{n,V,t}$. The existence of the moduli space $\bar{M}^{K}_{n,V,t}$ will then follow by verifying the standard properties of $S$-completeness and $\Theta$-reductivity (see \cite{AlperHalpernLeistnerHeinloth2023}), for which the formulation in terms of test configurations seems essential. 

To conclude, in Section \ref{sec: examples} we use the valuative criteria described above, to prove some general results on how to detect $t$-K-semistability, and we use these to provide some examples of $t$-K-semistable, $t$-K-stable and $t$-K-unstable adjoint Fano foliated structures. In particular we obtain an example of a $t$-K-stable Fano foliated structure where both $X$ and $\F$ are Fano (see Example \ref{ex:cubic-fourfold-pencil-foliation}). Since the theory would benefit from more examples, a natural question is whether the Abban--Zhuang theory \cite{AZ22}, which is one of the most effective tools at hand to determine K-stability of Fano varieties in explicit examples, also extends to $t$-Kstability.
\begin{question}\label{quest: az method}
    Is there an analogue of the Abban--Zhuang theory for $t$-K-stability?
\end{question}
Since currently, we have a limited toolkit in order to explicitly verify $t$-K-stability in explicit examples a positive answer to Question \ref{quest: az method} is very desirable.

\subsection*{Structure of the paper}
In Section \ref{sec:preliminaries} we recall the basic notions of test configurations, as well as some background on valuations and birational geometric notions for varieties and adjoint foliated structures that will be used throughout the paper. In Section \ref{sec: K-stability def} we introduce the notion of $\F$-compatible test configurations and the central notion of this paper, $t$-K-semistability. In Section \ref{sec: valuative invariants} we introduce the mixed beta invariant and we prove that that the mixed Donaldson--Futaki invariant for special $\F$-compatible test configurations is given as the mixed beta invariant of the central fibre, which is one of the key components of the valuative criterion of Theorem \ref{thm:dreamy-valuative-criterion intro 1}. 

In Section \ref{sec: special t.c.} we apply the MMP techniques to show that starting from a normal $\F$-compatible test configuration we can construct a special $\F$-compatible test configuration, proving the first part of Theorem \ref{thm:reduction-special-mixed-DF intro}. Then, in Section \ref{sec: reduction to special t.c.} we show that the mixed Donaldson--Futaki invariant decreases along each step of this birational procedure, proving the second part of Theorem \ref{thm:reduction-special-mixed-DF intro}. 

We then define the notion of the mixed Ding invariant and $t$-Ding stability in Section \ref{sec: ding stability} and we show that the mixed Ding invariant also decreases along each birational step of Section \ref{sec: special t.c.}. This allows us to prove the equivalence between $t$-Ding stability and $t$-K-stability detailed in Theorem \ref{thm: intro ding vs K}. Furthermore, by studying the mixed Ding invariant along divisorial valuations, we are able to prove Theorem \ref{thm:dreamy-valuative-criterion intro 1}. 

Then, in Section \ref{sec: birational inv for t ks} we introduce the mixed alpha and delta invariants and prove the valuative description of $t$-K-semistability described in Theorem \ref{thm: delta alpha thm intro}. We use these to provide specific examples of $t$-K-stable, $t$-K-semistable and $t$-K-unstable adjoint Fano foliated structures in Section \ref{sec: examples}.

We conclude by using this valuative description along with deriving some bounds on mixed log discrepancies to prove Theorem \ref{thm: boundedness for K-semi intro} in Section \ref{sec: boundedness}.

\subsection*{Acknowledgments}
I would like to thank Caucher Birkar, Federico Bongiorno, Paolo Cascini, Ruadhai Dervan, Jihao Liu and Calum Spicer for many helpful conversations and valuable comments. I am supported by  Beijing Natural Science Foundation Project IS25037 and a Shuimu Scholar Programme Scholarship at Tsinghua University.

\section{Preliminaries}\label{sec:preliminaries}
We will first collect some results that will be useful throughout the paper. 

\subsection{Test configurations and K-stability}
We first recall the notion of test configurations. For further information we prompt the reader to excellent surveys in \cite{Xu2025} and \cite{calabi_book} and the references therein.

\begin{definition}\label{def: test configurations}
    Let $X$ be an $n$-dimensional normal variety, and $L$ be an ample line bundle on $X$. Then a \textup{normal test configuration} $(\X;\L)$ of $(X;L)$ consists of 
\begin{itemize}
    \item a normal projective variety $\X$ with a flat projective morphism $\pi:\X\rightarrow \A^1$;
    \item a line bundle $\L$ ample over $\A^1$;
    \item a $\G_m$-action on the polarised variety $(\X,\L)$ such that $\pi$ is $\G_m$-equivariant, where $\A^1$ is equipped with the standard $\G_m$-action;
    \item a $\G_m$-equivariant isomorphism between the restriction $(\X\setminus \X_0;\L|_{\X\setminus \X_0})$ and $(X;L)\times (\A^1\setminus\{0\})$;
\end{itemize}
A test configuration is called
\begin{itemize}
    \item a \textup{product test configuration} if 
    \begin{equation}
    (\X;\L)\simeq (X\times \A^1;p_1^{*}L\otimes \OO_{\X}(l\X_0))
            \label{eq:product-TC}
    \end{equation}
    for some $l\in \Z$;
    \item a \textup{trivial test configuration} if it is a product test configuration and the isomorphism \eqref{eq:product-TC} is $\G_m$-equivariant, where the $\mathbb G_m$-action on $X$ is trivial;
    \item a \textup{special test configuration} if $\X_0$ is irreducible, reduced and  $(\X, \X_0)$ is purely log terminal (plt). 
\end{itemize}

The \textup{Donaldson Futaki invariant} of a normal test configuration $(\X,\L)/\A^1$ is $$\DF(\X,\L):=\frac{1}{L^n}\left(\frac{n}{n+1}\cdot\mu(X,L)\cdot{\bar{\L}^{n+1}}+{\bar{\L}^n.K_{\bar{\X}/\PP^1}}\right),$$ where $(\bar{\X};\bar{\L})$ is the natural compactification of $(\X;\L)$ over $\PP^1$. Here $\mu(X,L) = \frac{-K_X\cdot L^{n-1}}{L^n}$ is the \emph{slope} of the test configuration. 
\end{definition}

We also recall the definition of the non-Archimedean J-functional for a test configuration.

\begin{definition}
Let $(\X,\L)$ be a normal test configuration for $(X,L)$.
The non-Archimedean $J$-functional is defined by
\[
J^{NA}(\X,\L)
:=
\frac{1}{V}
\left(
\L\cdot (\mu^*L)^n
-
\frac{1}{n+1}\L^{n+1}
\right),
\]
where $\mu:\X\to X\times\mathbb P^1$ is the natural birational map.
\end{definition}

The following is the usual definition of K-stability via test configurations. 

\begin{definition}
Let $(X,L)$ be a polarised variety.

\begin{enumerate}
\item $(X,L)$ is \emph{-K-semistable}
if
\[
\DF^{[t]}(\X,\L)\ge0
\]
for all normal test configurations.

\item $(X,L)$ is \emph{K-stable}
if
\[
\DF^{[t]}(\X,\L)>0
\]
for all non-trivial test configurations.

\item $(X,L)$ is \emph{uniformly K-stable}
if there exists $\delta>0$ such that
\[
\DF^{[t]}(\X,\L)
\ge
\delta\,J^{NA}(\X,\L)
\]
for all normal test configurations.
\item $(X,L)$ is \emph{K-polystable} if it is K-semistable and \[
\DF^{[t]}(\X,\L)= 0 \text{ if and only if }(\X, \L) \text{ is of product type.}
\] 
\end{enumerate}
\end{definition}

\subsection{Valuations and divisorial valuations}

We briefly recall some standard terminology; see for instance
\cite[Section~2.1]{Kollar13},
\cite[Chapter~9]{Laz04},
and \cite[Section~3]{JM12}.

Let $X$ be an irreducible normal variety with function field $K(X)$. A \emph{valuation} on $X$ is a valuation
\[
v \colon K(X)^{*}\to \mathbb{R}
\]
which is trivial on the ground field and has a center on $X$. If $f\in K(X)^{*}$, we write $v(f)$ for the value of $f$ under $v$.

A valuation $v$ is called \emph{divisorial} if there exist a birational morphism $\mu \colon Y \to X$ from a normal variety $Y$, a prime divisor $E\subset Y$, and a positive real number $c\in \mathbb{R}_{>0}$ such that $v=c\,\ord_E$. When $c=1$, we often identify the divisor $E$ with the valuation $\ord_E$. We say that $E$ is a divisor \emph{over} $X$.

If $D$ is a $\Q$-Cartier $\Q$-divisor on $X$, then for a divisorial valuation $v=c\,\ord_E$ we define
\[
v(D):=c\,\ord_E(\mu^*D),
\]
where $\mu\colon Y\to X$ is a model on which $E$ appears. Likewise, if $\fa\subset \cO_X$ is a coherent ideal sheaf, we set
\[
v(\fa):=\min\{\,v(f)\mid f\in \fa_{c_X(v)}\,\},
\]
where $c_X(v)$ denotes the center of $v$ on $X$.

We denote by $\Val^{*}_X$ the space of all non-trivial valuations on $X$ and by $\Val^{\mathrm{div}, *}_X\subset \Val^*_X$ the space of all non-trivial divisorial valuations on $X$. 

For $v\in \Val_X$ and $m\in \R_{\ge 0}$, we define the \emph{valuation ideal}
\[
\fa_m(v):=\{\,f\in \cO_{X,c_X(v)} \mid v(f)\ge m\,\}.
\]
Assume that $x:=c_X(v)$ is a closed point. The \emph{volume} of $v$ is
\[
\vol_{X,x}(v)
:=
\limsup_{m\to +\infty}
\frac{\ell\bigl(\cO_{X,x}/\fa_m(v)\bigr)}{m^n/n!},
\]
where $n=\dim X$ and $\ell(-)$ denotes the length of an $\cO_{X,x}$-module (c.f.
\cite[Section~3]{ELS03} and \cite[Section~2]{Li18}).

\subsection{Log discrepancy and log canonical thresholds}\label{sec: lg disc prelim}

We briefly recall the standard definitions; see for instance
\cite[Definition 2.25]{KM98},
\cite[Section~2.3]{Kollar13},
and \cite[Section~3.1]{Laz04}.

Let $X$ be a normal variety over an algebraically closed field of characteristic $0$, and let $\Delta\geq 0$ be an effective $\Q$-divisor such that $K_X+\Delta$ is $\Q$-Cartier. For a prime divisor $E$ over $X$, namely a prime divisor on some birational model $\mu \colon Y \to X$, the \emph{log discrepancy} of $E$ with respect to the pair $(X,\Delta)$ is
\[
A_{X,\Delta}(E)
\,=\,
1+\ord_E\!\bigl(K_Y-\mu^*(K_X+\Delta)\bigr);
\]
equivalently, if
\[
K_Y=\mu^*(K_X+\Delta)+\sum_E a(E,X,\Delta)\,E,
\]
then
\[
A_{X,\Delta}(E)=1+a(E,X,\Delta).
\]

\begin{definition}\label{def: MMP sings}
    We say that the pair $(X,\Delta)$ is
\begin{itemize}
    \item \emph{terminal} if $A_{X,\Delta}(E)>1$ for every exceptional prime divisor $E$ over $X$;
    \item \emph{canonical} if $A_{X,\Delta}(E)\geq 1$ for every exceptional prime divisor $E$ over $X$;
    \item \emph{klt} if $A_{X,\Delta}(E)>0$ for every prime divisor $E$ over $X$;
    \item \emph{log canonical} if $A_{X,\Delta}(E)\geq 0$ for every prime divisor $E$ over $X$;
    \item  \emph{$\epsilon$-log canonical} if $A_{X,\Delta}(E)\geq \epsilon$ for every prime divisor $E$ over $X$;
    \item \emph{$\epsilon$-klt} if $A_{X,\Delta}(E)> \epsilon$ for every prime divisor $E$ over $X$.
\end{itemize}
When $\Delta=0$, we will say that $X$ has terminal, canonical, klt, or log canonical singularities respectively.
\end{definition}

Let $D\geq 0$ be an effective $\Q$-Cartier $\Q$-divisor on $X$. The \emph{log canonical threshold} of $D$ with respect to $(X,\Delta)$ is
\[
\lct(X,\Delta;D)
:=
\sup \bigl\{\, c\in \R_{\geq 0}\mid (X,\Delta+cD)\text{ is log canonical}\,\bigr\},
\]
see \cite[Definition 9.3.1]{Laz04} or \cite[Section~2.4]{Kollar13}. Equivalently,
\[
\lct(X,\Delta;D)
=
\inf_E \frac{A_{X,\Delta}(E)}{\ord_E(D)} = \inf_{v}\frac{A_{X,\Delta}(v)}{v(D)},
\]
where the infimum is taken over all prime divisors $E$ over $X$ with $\ord_E(D)>0$, or over all valuations $v$ with $v(D)>0$ (c.f. \cite[Proposition 9.5.13]{Laz04}, \cite{JM12}). When $\Delta=0$, we simply write $A_X(E)$ and $\lct(X;D)$. 

For later use, we also recall the definition of the $S$-invariant. Let $L$ be a big $\Q$-Cartier $\Q$-divisor on $X$, and let $E$ be a prime divisor over $X$. The \emph{pseudo-effective threshold} of $E$ with respect to $L$ is
\[
T_L(E):=\sup\{\, t\in \R_{\ge 0}\mid \vol(\mu^*L-tE)>0\,\},
\]
where $\mu\colon Y\to X$ is a birational model on which $E$ appears as a divisor. The \emph{$S$-invariant} of $E$ with respect to $L$ is
\[
S_L(E)
:=
\frac{1}{\vol(L)}
\int_0^{T_L(E)} \vol(\mu^*L-tE)\,dt.
\]
This quantity is independent of the choice of model $Y$; see for instance \cite[Section~2.1]{BoJ20} and \cite[Section~4]{Fujita16}.

When $L=-(K_X+\Delta)$ is big, we write
\[
S_{X,\Delta}(E):=S_{-(K_X+\Delta)}(E).
\]
In particular, if $(X,\Delta)$ is log Fano, then
\[
S_{X,\Delta}(E)
=
\frac{1}{\vol\bigl(-(K_X+\Delta)\bigr)}
\int_0^{T_{X,\Delta}(E)}
\vol\bigl(-\mu^*(K_X+\Delta)-tE\bigr)\,dt.
\]
The \emph{$\beta$-invariant} of a prime divisor $E$ over $X$ is then 
$$\beta_{(X,\Delta)}(E) = A_{(X,\Delta)}(E) - S_{L}(E).$$

Note that these notions extend naturally for arbitrary valuations (see for instance \cite{BJ20} or \cite{Xu2025}) although we will not explicitly define them. Recall also the following definition.

\begin{definition}\label{def:j-invariant}
Let $v$ be a divisorial valuation over $X$. We define
\[
j_L(v):=T_L(v)-S_L(v).
\]
If $v=\ord_E$ for a prime divisor $E$ over $X$, we also write $j_L(E)$.
\end{definition}

Now let $(X,\Delta)$ be a klt pair. For $v\in \Val_{X,x}$, the \emph{normalised volume} of $v$ is
\[
\widehat{\vol}_{(X,\Delta),x}(v)
:=
A_{X,\Delta}(v)^n \cdot \vol_{X,x}(v),
\]
where $A_{X,\Delta}(v)$ is the log discrepancy of $v$. The normalised volume of the singularity $x\in (X,\Delta)$ is
\[
\widehat{\vol}(x,X,\Delta)
:=
\inf_{v\in \Val_{X,x}} \widehat{\vol}_{(X,\Delta),x}(v).
\]
We refer the reader to \cite{Li18,BLX22} for further background.

\subsection{Foliations and adjoint foliated structures}

In this subsection we recall the basic definitions concerning foliations and adjoint foliated structures. We follow the standard conventions in the birational geometry of foliations; for general background on foliations we refer for instance to \cite{Brunella15,AD13, ACSS21, CS25}, and for adjoint foliated structures to \cite{CHLMSX25,CHLMSSX24}.

\begin{definition}\label{def: foliation}
Let $X$ be a normal variety. A \emph{foliation} on $X$ is a saturated coherent subsheaf $T_\F\subseteq T_X$ which is closed under the Lie bracket. Its rank is the generic rank of $T_\F$, and its codimension is $\codim(\F):=\dim X-\rk({\F})$. Equivalently, one may regard $\F$ as an integrable distribution on $X$.
\end{definition}

Since $T_{\F}$ is saturated in the reflexive sheaf $T_X$, it is reflexive. In particular, $\det(T_{\F})$ is a rank one reflexive sheaf, hence corresponds to a Weil divisor class on $X$.

\begin{definition}
Let $X$ be a normal variety and let $\F$ be a foliation on $X$. We say that $\F$ is \emph{algebraically integrable} if the leaf of $\F$ through a very general point of $X$ is Zariski open in an algebraic subvariety of $X$. Equivalently, the closure of the leaf through a very general point is an algebraic subvariety.
\end{definition}

Equivalently, $\F$ is algebraically integrable if its general leaves are algebraic. In this case, after possibly replacing $X$ by a birational model, one may often view $\F$ as induced by a rational map
\[
\begin{tikzcd}
f\colon X \arrow[r, dashed] & Y
\end{tikzcd}
\]
whose general fibres are tangent to $\F$. Throughout when we refer to foliations, we will always assume they are algebraically integrable, unless stated otherwise.

\begin{definition}\label{def: canonical class of foliation}
Let $X$ be a normal variety and let $\F$ be a foliation on $X$. The \emph{canonical class} of $\F$ is the Weil divisor class $K_{\F}$ determined by
\[
\OO_X(-K_{\F})\simeq \det(T_{\F}),
\]
or equivalently
\[
\OO_X(K_{\F})\simeq \det(T_{\F})^{\vee}.
\]
When $X$ is smooth, if $N_{\F}^{\vee}:=(T_X/T_{\F})^{\vee}$ denotes the conormal sheaf of $\F$, then
\[
K_X = K_{\F} + \det(N_{\F}^{\vee}).
\]
\end{definition}

We will usually also impose as an extra condition that $K_{\F}$ is $\Q$-Cartier.

\begin{definition}
A \emph{foliated pair} consists of a normal variety $X$ together with a foliation $\F$ such that $K_{\F}$ is $\QQ$-Cartier. Furthermore, we call the triple $(X,\F,L)$ a \emph{foliated polarised variety} if $(X,\F)$ is a foliated pair, and $L$ is an ample divisor on $X$.
\end{definition}

We now recall the main object of interest in this paper.

\begin{definition}\label{def: adjoint foliated structure}
An \emph{adjoint foliated structure} is a triple $(X,\F,t)$, where $X$ is a normal variety, $\F$ is a foliation on $X$, and $t\in[0,1]$. To such a triple one associates the $\QQ$-divisor
\[
K_{X,\F}^{[t]}:= tK_{\F}+(1-t)K_X.
\]
We say that $(X,\F,t)$ is $\QQ$-Gorenstein if $K_{X,\F}^{[t]}$ is $\QQ$-Cartier.
\end{definition}

\begin{definition}
    We say that the adjoint foliated structure $(X,\F,t)$ is an \emph{adjoint Fano foliated structure} if $-K_{X,\F}^{[t]}$ is ample.
\end{definition}

Thus adjoint foliated structures interpolate between the birational geometry of the ambient variety and that of the foliation: when $t=0$ one recovers the usual canonical class $K_X$, while when $t=1$ one recovers the foliated canonical class $K_{\F}$. This point of view is central in the recent development of the minimal model program and boundedness theory for foliated structures.

\begin{remark}\label{rem: adjoint fol structures in our notation}
In the recent literature, adjoint foliated structures are treated as the natural analogue of log pairs in the foliated setting, with the divisor $tK_{\F}+(1-t)K_X$ playing the role of the adjoint canonical class. We should note that in \cite{CHLMSSX24, CHLMSX25} these structures can be defined over some $U$ and with additional boundary divisor data and are usually denoted by $(X,\F,B,\mathbf{M},t)/U$. But, for our purposes, $B = \textbf{M}=0$, and $U$ will either be trivial, or implicit (usually a smooth curve as in Section \ref{sec: special t.c.}) so will usually be omitted.
\end{remark}

\subsection{Foliated log discrepancy}

Let $\mu:Y\to X$ be a birational morphism extracting a prime divisor $E$; denote by $\cF_Y$ the induced foliation on $Y$. We say that $E$ is \emph{$\cF$-invariant} if the foliation is tangent to $E$ generically (equivalently, $\cF_Y\subset T_Y(-\log E)$ generically); otherwise $E$ is \emph{transverse}. The foliated canonical class satisfies
$$
K_{\cF_Y}\ =\ \mu^*K_{\cF}+\sum_E a(E,\cF)\,E.
$$
We define
$$
\varepsilon(E)\ :=\
\begin{cases}
1, & E \text{ transverse},\\
0, & E \text{ $\cF$-invariant},
\end{cases}
\qquad
{\ A_{X,\cF}(E)\ :=\ \varepsilon(E)+a(E,\cF)\ }.
$$
For a divisorial valuation $v=c\,\mathrm{ord}_E$, we set $A_{X,\cF}(v):=c\,A_{X,\cF}(E)$. The log discrepancy extends naturally to quasi-monomial valuations (and valuations in general) by continuity. We also obtain similar definitions for foliated log pairs.

For an adjoint foliated structure $(X,\F,t)$ we define the \emph{mixed log discrepancy},
\[
A^{[t]}_{X,\F}(E)
=
t\bigl(a(E,\F)+\varepsilon(E)\bigr)+(1-t)\bigl(a(E,X)+1\bigr).
\]
This obtains a natural extension when we consider adjoint foliated structures $(X,\F,D,t)$ with a divisor $D$. In fact, the following is a useful observation.

\begin{lemma}
\label{lem:mixed-discrepancy-subtract-D}
Let $(X,\cF)$ be a foliated variety with $K_X$ and $K_{\cF}$ $\QQ$-Cartier, and let
$D\geq 0$ be an effective $\QQ$-divisor such that $K_X+D$ and $K_{\cF}+D$ are
$\QQ$-Cartier. Let $E$ be a prime divisor over $X$. Then
\[
A^{[t]}_{X,\cF;D}(E):=tA_{X,D}(E)+(1-t)A_{X,\cF;D}(E)
= A^{[t]}_{X,\cF}(E)-\ord_E(D).
\]
\end{lemma}

\begin{proof}
The usual discrepancy identity
\[
A_{X,D}(E)=A_X(E)-\ord_E(D)
\]
is standard.

For the foliated discrepancy, let $\mu:Y\to X$ be a birational morphism
extracting $E$, and let $\cF_Y$ be the induced foliation on $Y$. By definition,
\[
K_{\cF_Y}=\mu^*K_{\cF}+\sum_F a(F,\cF)\,F.
\]
Since
\[
\mu^*(K_{\cF}+D)=\mu^*K_{\cF}+\mu^*D,
\]
we have
\[
K_{\cF_Y}-\mu^*(K_{\cF}+D)
=
\bigl(K_{\cF_Y}-\mu^*K_{\cF}\bigr)-\mu^*D.
\]
Taking the coefficient of $E$ gives
\[
\ord_E\!\bigl(K_{\cF_Y}-\mu^*(K_{\cF}+D)\bigr)
=
a(E,\cF)-\ord_E(D).
\]
Moreover, $\varepsilon(E)$ is unchanged, since it depends only on whether $E$ is
$\cF$-invariant or transverse, and adding the divisor $D$ does not alter the
foliation. Therefore
\[
A_{X,\cF;D}(E)
=
\varepsilon(E)+a(E,\cF)-\ord_E(D)
=
A_{X,\cF}(E)-\ord_E(D).
\]

Finally,
\[
\begin{aligned}
A^{[t]}_{X,\cF;D}(E)
&=(1-t)A_{X,D}(E)+tA_{X,\cF;D}(E)\\
&=(1-t)\bigl(A_X(E)-\ord_E(D)\bigr)
 +t\bigl(A_{X,\cF}(E)-\ord_E(D)\bigr)\\
&=A^{[t]}_{X,\cF}(E)-\ord_E(D),
\end{aligned}
\]
as required.
\end{proof}

\begin{definition}
Let $(X,\F,t)$ be an adjoint foliated structure with $t\in[0,1]$.  We say that $(X,\F,t)$ is
\begin{enumerate}
    \item \emph{log canonical} if $A^{[t]}_{X,\F}(E)\ge 0$ for every prime divisor $E$ over $X$;
\item \emph{klt} if $A^{[t]}_{X,\F}(E)> 0$ for every prime divisor $E$ over $X$;
\item \emph{$\epsilon$-log canonical} if $A^{[t]}_{X,\F}(E)\ge \epsilon$ for every prime divisor $E$ over $X$;
\item \emph{$\epsilon$-klt} if $A^{[t]}_{X,\F}(E)> \epsilon$ for every prime divisor $E$ over $X$.
\end{enumerate}
\end{definition}

\begin{remark}
\label{rem:comparison-Cas25-singularities}
Our definition of $\epsilon$-lc (resp. lc, klt) for an adjoint foliated structure $(X,\F,t)$ is identical to the definition in \cite{CHLMSX25}. More precisely, in \cite{CHLMSX25} singularities are defined for adjoint foliated structures of the form $\cA=(X,\F,B,M,t)$, via the condition
\[
a(E,A)+t\,\epsilon_{\F}(E)+(1-t)\ge \epsilon
\]
for all divisorial valuations $E$ over $X$, where $a(E,A)$ is the adjoint discrepancy.

In the special case $B=M=0$, which as in Remark \ref{rem: adjoint fol structures in our notation} this paper focuses on, we have
\[
a(E,A)=t\,a(E,\F)+(1-t)\,a(E,X),
\]
and therefore
\[
a(E,A)+t\,\epsilon_{\F}(E)+(1-t)
=
t\bigl(a(E,\F)+\epsilon_{\F}(E)\bigr)
+(1-t)\bigl(a(E,X)+1\bigr).
\]
This coincides with our definition
\[
A^{[t]}_{X,\F}(E)
=
t\bigl(a(E,\F)+\epsilon_{\F}(E)\bigr)
+(1-t)\bigl(a(E,X)+1\bigr).
\]

In particular, the notions of $\epsilon$-lc, lc, and klt used in this paper agree with those of \cite{CHLMSX25} in the boundary-free case.
\end{remark}

We will now define a mixed log canonical threshold for adjoint foliated structures.

\begin{definition}
\label{def:mixed-lct}
Let $D\ge 0$ be an effective $\QQ$-divisor on $X$. We define the \emph{mixed log canonical threshold} of $D$ by
\[
\lct^{[t]}(X,\F;D)
:=
\inf_{v\in \Val_X^*}\frac{A^{[t]}_{X,\F}(v)}{v(D)},
\]
where $\Val_X^*$ denotes the set of nontrivial valuations of finite mixed
discrepancy. Equivalently,
\[
\lct^{[t]}(X,\F;D)
=
\inf_{v \in \Val_X^{\mathrm{div}, *}}
\frac{A^{[t]}_{X,\F}(v)}{v(D)}.
\]
\end{definition}

\begin{remark}
The second equality follows by the same argument as for the usual valuation formula for log canonical thresholds: the discrepancy functional is replaced by $A^{[t]}_{X,\F}$, while the proof depends only on lower semicontinuity,
homogeneity, and approximation by divisorial valuations \cite{JM12}. Equivalently, 
\[\lct(X,\F;D):=\sup\{s\geq 0\mid (X,\F,sD)\text{ is lc}\}.\]
\end{remark}

We end this section with a definition of boundedness in the setting of foliations.

\begin{definition}\label{def: boundedness}
We say that a set $\mathcal P$ of projective foliated pairs $(X,\mathcal F)$ is \emph{bounded} if there exist finitely many flat projective morphisms $f_i \colon X_i \to T_i$ of normal varieties with normal fibers and foliations $\cG_i$ of rank $r_i$ on $X_i$ with $i=1,\dots,N$, such that
\begin{enumerate}
    \item for any closed point $t\in T_i$, if $(X_i)_t$ denotes the fiber of $f_i$ over $t$, then $T_{X_i/T_i}\big|_{(X_i)_t} \simeq T_{(X_i)_t}$;
    \item $T_{\mathcal G_i}\subset T_{X_i/T_i}$ and for any closed point $t\in T_i$, we have that $T_{\mathcal G_i}\big|_{(X_i)_t} \subset T_{(X_i)_t}$ defines a foliation $(\mathcal F_i)_t$ of rank $r_i$ on $(X_i)_t$; and
    \item for any $(X,\mathcal F)\in \mathcal P$, there exist $i=1,\dots,N$, a closed point $t\in T_i$, and an isomorphism $\phi\colon X \to (X_i)_t$ such that $\phi_*\mathcal F \simeq (\mathcal F_i)_t$.
\end{enumerate}
\end{definition}

\section{Foliated test configurations and K-stability for foliations}\label{sec: K-stability def}

\subsection{\texorpdfstring{$\F$-compatible test configurations}{F-compatible test configurations}}

In this section we introduce the notion of foliated test configurations which will allow us to define K-stability for foliations.

Let $X$ be a normal projective variety of dimension $n$ over $\mathbb{C}$ and
$L$ an ample $\mathbb{Q}$-line bundle on $X$.
Let $\mathcal{F}\subset T_X$ be a saturated integrable subsheaf
(i.e.\ an algebraic foliation) such that the canonical divisor
$K_{\mathcal{F}}$ is $\mathbb{Q}$-Cartier. For $t\in[0,1]$ we define the \emph{mixed canonical divisor}
\[
K^{[t]}_{X,\F}:=(1-t)K_X+tK_{\F}.
\]

\begin{definition}\label{def: foliated t.c.}
A \emph{foliated test configuration} for the polarised foliated variety
$(X,\F,L)$ consists of a triple
\[
(\pi:\X\to \mathbb A^1,\ \F_{\X},\ \L)
\]
such that:

\begin{enumerate}
\item $(\X,\L)$ is a normal test configuration for $(X,L)$,
i.e.\ $\pi:\X\to\mathbb A^1$ is a flat projective morphism endowed
with a $\mathbb G_m$-action lifting the standard action on $\mathbb A^1$,
$\L$ is a relatively ample $\mathbb Q$-line bundle, and
\[
(\X,\L)|_{\pi^{-1}(\mathbb A^1\setminus\{0\})}
\simeq (X,L)\times (\mathbb A^1\setminus\{0\})
\]
$\G_m$-equivariantly;

\item $\F_{\X}\subset T_{\X/\mathbb A^1}$ is a
$\mathbb G_m$-equivariant saturated integrable subsheaf;

\item over $\mathbb A^1\setminus\{0\}$ the foliation coincides with the
product foliation:
$$(\X,\F_{\X})|_{\pi^{-1}(\mathbb A^1\setminus\{0\})}
\simeq (X,\F)\times (\mathbb A^1\setminus\{0\});$$

\item $\F_{\X}$ is algebraically integrable.
\end{enumerate}

Such a test configuration will be called \emph{$\F$-compatible}.
\end{definition}

\begin{remark}\label{rem: alg integrable condition}
    Note, that we are imposing condition (4) in order to be able to apply the MMP and base-point-freeness results for adjoint foliated structures from \cite{CHLMSSX24,CHLMSX25} in the Sections that are to follow. It is important to note that we do not claim here that algebraic integrability is preserved under arbitrary deformations; rather, it is imposed as part of the definition of the class of test configurations considered in this paper. It is natural to ask whether condition (4) can be removed in the future. 
\end{remark}

\subsection{The mixed Donaldson--Futaki invariant}

We will now define the \emph{$t$-foliated Donaldson--Futaki invariant} via intersection numbers following the description of Odaka--Wang \cite{Odaka11,Odaka12, Wang2012}. In order to do this, given an arbitrary $\F$-compatible test configuration $\pi:(\X,\F_{\X},\L)\rightarrow \A^1$, we take $\bar{\pi}:(\bar{\X},\F_{\bar{\X}},\bar{\L})\rightarrow \PP^1$ for the natural compactification and we set $V:=L^n$ and 
\[\mu(X,\F,L) = \frac{-K^{[t]}_{X,\F}\cdot L^{n-1}}{L^n}\]
for the \emph{foliated slope}. Furthermore, we define the \emph{relative mixed canonical divisor} by
\[
K^{[t]}_{\bar{\X}/\PP^1}
:=
(1-t)K_{\bar{\X}/\PP^1}+tK_{\F_{\bar{\X}}}.
\]
Here the intersection numbers are interpreted via the intersection formula, viewing $K^{[t]}_{\X,\F_{\X}/\A^1}$ as an $\R$-Weil divisor intersected linearly with powers of the $\Q$-Cartier polarization $\L$.

\begin{definition}\label{def:DFwt-corrected}
Let $(\X,\ \F_{\X},\ \L)$
be a normal $\F$-compatible test configuration for $(X,\F,L)$.
We fix $t\in (0,1)$. The \emph{$t$-foliated Donaldson--Futaki invariant} is defined by
\[
\DF^{[t]}(\X,\F_{\X},\L)
:=
\frac{1}{V}\left(
\frac{n}{n+1}\,\mu(X,\F,L)\bar{\L}^{n+1}
+
K^{[t]}_{\bar{\X}/\mathbb P^1}\cdot \bar{\L}^n
\right).
\]
\end{definition}

We will also call this invariant the ``mixed Donaldson--Futaki'' invariant for brevity.

\begin{remark}\label{rem: sanity checks}
If $\F=T_X$, then $K_{\F}=K_X$ and
\[
\DF^{[t]}(\X,\F_{\X},\L)
=
\frac{1}{V}\left(
\frac{n}{n+1}\mu(X,L)\bar{\L}^{n+1}
+
K_{\bar{\X}/\mathbb P^1}\cdot \bar{\L}^n
\right),
\]

which coincides with the usual Donaldson--Futaki invariant. Similarly, if $t=0$ we recover the usual notion for K-stability, noting that we do not have a foliation contribution, hence we reduce to the usual test configurations of Definition \ref{def: test configurations}. Similarly, if $\F = 0$, then $K_{\F} = 0$ and we recover the usual K-stability notions for $X$.
\end{remark}

Note that the notions of product, trivial and special test configurations of Definition \ref{def: test configurations} extend to those of $\F$-compatible product, trivial and special test configurations in the natural way. Here we introduce the notion of $t$-weakly special $\F$-compatible test configurations.

\begin{definition}\label{def:t-weakly-special}
A normal ample $\F$-compatible test configuration $(\X,\F_{\X},\L)$ for $(X,\F,L)$ is called \emph{$t$-weakly special} if
\begin{enumerate}
\item
\[
\L \sim_{\Q,\A^1} -K^{[t]}_{\X/\A^1,\F_{\X}};
\]
\item the pair $(\X,\F_{\X};\X_0)$ is mixed log canonical,
i.e.
\[
A^{[t]}_{\X,\F_{\X}}(w)-(1-t)w(\X_0)\ge 0
\]
for every divisorial valuation $w$ over $\X$.
\end{enumerate}
If moreover $\X_0$ is irreducible and reduced and
$(\X,\F_{\X};\X_0)$ is mixed plt, then we call it
\emph{$t$-special}.
\end{definition}

We first record the following Lemma which is an easy consequence of Definition \ref{def:DFwt-corrected}.
\begin{lemma}
\label{lem:anti-adjoint-mixed-DF}
Let $(\pi\colon \X\to \PP^1,\F_{\X},\L)$ be a normal $\F$-compatible test configuration. Assume that $\L\sim_{\QQ,\PP^1}-K^{[t]}_{\X/\PP^1}$. Then
\[
\DF^{[t]}(\X,\F, \L)
=
-\frac{1}{(n+1)V}\L^{n+1}.
\]
\end{lemma}

\begin{proof}
We have $\mu(X,\F,L)=1$, so by Definition \ref{def:DFwt-corrected},
\[
\DF^{[t]}(\X,\F, \L)
=
\frac{1}{V}\left(
\frac{n}{n+1}\L^{n+1}
+
K^{[t]}_{\X/\PP^1}\cdot \L^n
\right).
\]
Using $K^{[t]}_{\X/\PP^1}\sim_{\QQ,\PP^1}-\L$ we get $K^{[t]}_{\X/\PP^1}\cdot \L^n=-\L^{n+1}$, hence
\[
\DF^{[t]}(\X,\F, \L)
=
\frac{1}{V}\left(\frac{n}{n+1}-1\right)\L^{n+1}
=
-\frac{1}{(n+1)V}\L^{n+1}.
\qedhere
\]
\end{proof}

\subsection{K-stability for foliations and adjoint foliated structures}

The following is the key definition of this Section.

\begin{definition}\label{def: k-stability definition}
Let $(X,\F,L)$ be a polarised foliated variety.

\begin{enumerate}
\item $(X,\F,L)$ is \emph{$t$-K-semistable}
if
\[
\DF^{[t]}(\X,\F_{\X},\L)\ge0
\]
for all normal $\F$-compatible test configurations.

\item $(X,\F,L)$ is \emph{$t$-K-stable}
if
\[
\DF^{[t]}(\X,\F_{\X},\L)>0
\]
for all non-trivial $\F$-compatible test configurations.

\item $(X,\F,L)$ is \emph{uniformly $t$-K-stable}
if there exists $\delta>0$ such that
\[
\DF^{[t]}(\X,\F_{\X},\L)
\ge
\delta\,J^{NA}(\X,\L)
\]
for all normal $\F$-compatible test configurations.
\item $(X,\F,L)$ is \emph{$t$-K-polystable} if it is $t$-K-semistable and \[
\DF^{[t]}(\X,\F_{\X},\L)= 0 \text{ if and only if }(\X,\F_{\X}, \L) \text{ is $\F$-compatible of product type.}
\] 
\end{enumerate}
\end{definition}

\begin{remark}\label{rem: K-stab for foliations}
    Notice that if $t=0$ we recover the usual notion of K-stability. When $0<t<1$ we are in the ``mixed'' regime, where the most natural objects to study are adjoint foliated structures. Furthermore, when $t=1$ we obtain a purely foliation-geometric version of K-stability, which we will refer to as ``K-stability for foliations''.
\end{remark}

\section{Valuative invariants and special test configurations}\label{sec: valuative invariants}

In this section we introduce the valuative invariant corresponding to the
mixed Donaldson--Futaki invariant and explain how the two are expected to
coincide for special test configurations.

\subsection{\texorpdfstring{The $\beta$ invariant}{The beta invariant}}
Recall that for a divisorial valuation $v$ over $X$ the mixed log discrepancy is $A^{[t]}_{X,\F}(v)
:=
(1-t)A_X(v)+tA_{X,\F}(v)$.

\begin{definition}\label{def: mixed beta}
Let $v$ be a divisorial valuation over $X$. We define the \emph{mixed $\beta$ invariant}
\[
\beta^{[t]}_{X,\F,L}(v)
:=
A^{[t]}_{X,\F}(v)-S_L(v),
\]
where $S_L(v)$ is the S-invariant defined in Section \ref{sec: lg disc prelim}.
\end{definition}

\subsection{Foliated divisorial valuations}

In this subsection we introduce the class of valuations that naturally
appear in the study of $t$-K-stability for foliated varieties.

\begin{definition}\label{def: F dreamy}
Let $(X,\F,L)$ be a polarised foliated variety. A divisorial valuation $v$ is called
\emph{$\F$-dreamy} if

\begin{enumerate}
\item $v$ is dreamy with respect to $L$, i.e.
\[
\bigoplus_{m,k\ge0}H^0(X,mL-kv)
\]
is finitely generated;

\item the associated Rees degeneration produces a test configuration $(\X,\L)\to \A^1$ that admits a relative foliation $\F_{\X}\subset T_{\X/\A^1}$ extending the product foliation on
$X\times(\A^1\setminus\{0\})$.
\end{enumerate}
\end{definition}

Note that by Lemma \ref{lem:anti-adjoint-mixed-DF}, when $L = -K^{[t]}_{\X/\F}$, we have
\[\DF^{[t]}(\X,\F_{\X}, \L) = -\frac{\L^{n+1}}{V(n+1)}, \quad \mu(X,\F,L)=1,\] where $V = (-K_{X,\F}^{[t]})^n$.

We will first record an auxiliary proposition extending \cite[Proposition 2.10]{Fujita16} and \cite[Proposition 4.11]{BHJ19} in the foliated setting. Let us fix some notation. Let $(\X,\F_{\X})\to \mathbb A^1$ be a normal $\F$-compatible test configuration, and let $\bar{\X}\to \mathbb P^1$ be its natural compactification. Assume that the central fibre $\X_0$ is irreducible and reduced. Let
\[
\xymatrix{
& \bar{\Y} \ar[ld]_{\Pi} \ar[rd]^{\Theta} & \\
X\times \mathbb P^1 \ar@{-->}[rr] && \bar{\X}
}
\]
be the normalization of the graph of the induced birational map, $\Y_0=\sum_{i\in I}m_iE_i+\hat{X}_0+\hat{\X}_0$ be the irreducible decomposition,  where $\hat{X}_0$ is the strict transform of $X\times\{0\}$ and 
$\hat{\X}_0$ is the strict transform of $\X_0$. We define
\[
B^{[t]}
:=
\Theta^*\!\left(-K^{[t]}_{\bar{\X}/\mathbb P^1,
\F_{\bar{\X}}}\right)
-
\Pi^*p_1^*\!\left(-K^{[t]}_{X,\F}\right).
\]

\begin{proposition}\label{prop: auxil ord computation}
We have
\[
-\ord_{\hat{\X}_0}(B^{[t]})
=
A^{[t]}_{X,\F}(v).
\]
\end{proposition}

\begin{proof}
By the above expression, we have $B^{[t]}=(1-t)B_X+tB_{\F}$ where
\[
B_X
:=
\Theta^*(-K_{\bar{\X}/\mathbb P^1})
-
\Pi^*p_1^*(-K_X),
\]
and
\[
B_{\F}
:=
\Theta^*(-K_{\F_{\bar{\X}}})
-
\Pi^*p_1^*(-K_{\F}).
\]
It is therefore enough to prove that $-\ord_{\hat{\X}_0}(B_X)=A_X(v)$ and $-\ord_{\hat{\X}_0}(B_{\F})=A_{X,\F}(v)$. The first inequality follows directly from \cite[Proposition 2.10]{Fujita16} and \cite[Proposition 4.11]{BHJ19}.

It remains to prove the foliated equality. Let $\F_{\bar{\Y}}$ be the
birational transform of $\F_{\bar{\X}}$ on $Y$.
Since the test configuration is $\F$-compatible, we have $\F_{\X}\subset T_{\X/\mathbb A^1}$, therefore its birational transform satisfies $\F_{\bar{\Y}}\subset T_{\bar{\Y}/\mathbb P^1}$  at the generic point of every vertical divisor. In particular, $\hat{\X}_0$ is invariant for the relative foliation on $\Y$, because relative vector fields annihilate the pullback of a local parameter on the base.

We compare the two foliated canonical divisors. By the standard expressions, we have $K_{\F_{\bar{\Y}}} = \Theta^*K_{\F_{\bar{\X}}} + R_\Theta$ and $K_{\F_{\bar{\Y}}} = \Pi^*p_1^*K_{\F} + R_\Pi$. In particular,
\[
B_{\F}
=
\Theta^*(-K_{\F_{\bar{\X}}})
-
\Pi^*p_1^*(-K_{\F})
=
R_\Theta-R_\Pi.
\]
Note that, $\hat{\X}_0$ is not $\Theta$-exceptional, and $\Theta$ is an isomorphism
at the generic point of $\hat{\X}_0$. Hence $\ord_{\hat{\X}_0}(R_\Theta)=0$ which implies that $-\ord_{\hat{\X}_0}(B_{\F})
=
\ord_{\hat{\X}_0}(R_\Pi)$. We now show that $\ord_{\hat{\X}_0}(R_\Pi)=A_{X,\F}(v)$.

This is local at the generic point of $\hat{\X}_0$. Since $\hat{\X}_0$ is
$\mathbb G_m$-invariant and $\ord_{\hat{\X}_0}(t)=1$, the valuation of
$K(X)(t)$ defined by $\hat{\X}_0$ is the extension of $v$:
\[
\ord_{\hat{\X}_0}\!\left(\sum_i f_i t^i\right)
=
\min_i\{v(f_i)+i\}.
\]
We choose a birational model $\mu:Z\to X$ on which the divisor $F$ computing $v$ appears, and work at the generic point of $F$. We may assume that $Z$ is smooth at this point and that $F=\{z=0\}$ is a smooth divisor.
Let $\F_Z$ be the birational transform of $\F$. We have
\[
K_{\F_Z}
=
\mu^*K_{\F}
+
a(F,\F)F+\cdots.
\]
The extension of $v$ is realised, locally, by the exceptional divisor
of the blow-up of $F\times\{0\}\subset Z\times \mathbb A^1$. Let $\sigma:W\to Z\times \mathbb A^1$ be this blow-up and let $G\subset W$ be the exceptional divisor. Since the coefficient of $F\times \mathbb A^1$ in $K_{p_1^{-1}\F_Z}-p_1^*K_{\F}$ is $a(F,\F)$, its pullback contributes $a(F,\F)$ to the coefficient of $G$.

It remains to compute the additional contribution coming from the blow-up of $F\times\{0\}$. We will show that this contribution is exactly $\varepsilon(F)$. Near the generic point of $F$, there are two cases we need to study.

If $F$ is $\F_Z$-invariant, then local generators of
$\F_Z$ preserve the ideal $(z)$. After blowing up $(z,t)$, their
relative lifts remain regular generators of the transformed foliation without an additional vanishing factor along $G$. Hence the blow-up contributes $0$, which is $\varepsilon(F)$ in the invariant case.

If $F$ is transverse to $\F_Z$, then one local generator has a
non-zero normal component to $F$. On the blow-up chart $t=zs$ the rational lift of this generator which
annihilates \(t\) is, up to tangent terms, $\partial_z-\frac{s}{z}\partial_s$. The saturated transform of the foliation is therefore generated, up to tangent terms, by $z\partial_z-s\partial_s$, thus the saturated generator differs from the rational pullback of the
original generator by one factor of \(z\), i.e. by a local equation of the exceptional divisor \(G\). Consequently
\[
\det T_{\mathcal F_W}
=
\sigma^*p_1^*\det T_{\mathcal F_Z}\otimes\cO_W(-G)
\]
along the generic point of \(G\). Equivalently,
\[
K_{\mathcal F_W}
=
\sigma^*p_1^*K_{\mathcal F_Z}+G
\]
along \(G\). Hence the blow-up contributes \(1=\varepsilon(F)\).

Thus in both cases the coefficient of $G$ in the relative foliated canonical
divisor over $X\times\mathbb A^1$ is $a(F,\F)+\varepsilon(F)
=
A_{X,\F}(v)$. Since $G$ and $\hat{\X}_0$ define the same divisorial valuation of $K(X)(t)$, the same coefficient is obtained for $\hat{\X}_0$. Therefore
\[
\ord_{\hat{\X}_0}(R_\Pi)=A_{X,\F}(v).
\]
Consequently,
\[
-\ord_{\hat{\X}_0}(B_{\F})
=
A_{X,\F}(v).
\]

The proof of the proposition is completed by combining the ambient and foliated parts.
\end{proof}
We are now in a position to express the mixed Donaldson--Futaki invariant for $\F$-dreamy divisorial valuations in terms of the $\beta$-invariant.

\begin{theorem}\label{thm: DF description for special t.c.}
Let $v$ be an $\F$-dreamy divisorial valuation, and let $(\mathcal X,\mathcal F_{\mathcal X},\mathcal L)\to \mathbb A^1$ be the associated normal special \(\F\)-compatible Rees test configuration, with $v = v_{\X_0}$. Then the mixed
Donaldson--Futaki invariant satisfies
\[
\DF^{[t]}(\X,\F_{\X},\L)
=
\beta^{[t]}_{X,\F,L}(v_{\X_0}).
\]
\end{theorem}
\begin{proof}
    We argue as in \cite[Proof of Theorem 5.1]{Fujita16}. We pick $r_0 \in \Z_{>0}$ such that $-r_0K_{\X,\F/\A^1}^{[t]}$ is Cartier. Let 
    \[\xymatrix{
& \bar{\Y}  \ar[dl]_\Pi \ar[dr]^\Theta & \\
X\times\PP^1 & & \bar{{\X}}
}\]
be the normalization of the graph as before. Let $\Y_0=\sum_{i\in I}m_iE_i+\hat{X}_0+\hat{\X}_0$ be the irreducible decomposition, 
where $\hat{X}_0$ is the strict transform of $X\times\{0\}$ and 
$\hat{\X}_0$ is the strict transform of $\X_0$. As before, we set 
$B^{[t]}:=\Theta^*(-K_{\bar{\X}, \F_{\bar{\X}}^{[t]}/\pr^1})-\Pi^*p_1^*(-K_{X,\F}^{[t]})$ supported on $\Y_0$. Note that 
$-\ord_{\hat{\X}_0}(B^{[t]})=\ord_{\hat{\X}_0}(K_{\Y/X\times\A^1})$. 
Let $V_\bullet$ be the complete graded linear series of $-r_0K^{[t]}_{X,\F}$ and let us consider 
the filtration $\mathcal{G}:=\mathcal{G}_{(\Y, \Theta^*(-r_0K_{\X,\F_{\X}/\A^1}^{[t]}))}$ of $V_\bullet$. Note, that by Proposition \ref{prop: auxil ord computation} we have $-\ord_{\hat{\X}_0}(B^{[t]})=A^{[t]}_{X,\F}(v_{\X_0})$, and by an identical argument as in \cite[Proof of Theorem 5.1]{Fujita16} $v_{\X_0}$ is dreamy. 

Let 
\begin{eqnarray*}
\lambda_{\max}^{(k)} & := & \sup\{x\in\R\,|\, \mathcal{G}^x_{(\X, r_0\L)}V_k\neq 0\}, \\
\lambda_{\min}^{(k)} & := & \inf\{x\in\R\,|\, \mathcal{G}^x_{(\X, r_0\L)}V_k\neq V_k\}, \\
w(k) & := & \int_{\lambda_{\min}^{(k)}}^{\lambda_{\max}^{(k)}}\dim\mathcal{G}^x_{(\X, r_0\L)}V_k
dx+\lambda_{\min}^{(k)}\cdot\dim V_k.
\end{eqnarray*}
Then, by \cite[Proposition 2.12]{Fujita16} we have that 
\[
w(k) = \int_0^{kr_0\tau(v_{\X_0})}\dim H^0(X, -kr_0K_{X,\F}^{[t]}-xv_{\X_0})dx
-kr_0A_{X,\F}^{[t]}(v_{\X_0})\dim V_k. 
\]

In particular, by \cite[Proposition 2.12.(ii)]{Fujita16} or \cite[\S5, Lemma 7.7, Proposition 3.12]{BHJ19} we have that $((-K_{\bar{\X}/\pr^1})^{\cdot n+1}) = \L^{n+1}$ is equal to
\begin{eqnarray*}
\L^{n+1}&=&\lim_{k\to\infty}\frac{w(k)}
{(kr_0)^{n+1}/(n+1)!}\\
&=&(n+1)\left(\int_0^{\tau(v_{\X_0})}\vol_X(-K_{X,\F}^{[t]}-xv_{\X_0})dx
-A_{X,\F}^{[t]}(v_{\X_0})((-K_{X,\F}^{[t]})^{n})\right)\\
&=&-(n+1)V\beta^{[t]}(v_{\X_0}).
\end{eqnarray*}
Since by Lemma \ref{lem:anti-adjoint-mixed-DF} we have \[\DF^{[t]}(\X,\F_{\X}, \L) = -\frac{\L^{n+1}}{V(n+1)},\] 
by the above we obtain 
\[
\DF^{[t]}(\X,\F_{\X},\L)
=
\beta^{[t]}_{X,\F,L}(v_{\X_0}),
\]
as required.
\end{proof}

\begin{remark}
    Note that although at first glance it seems that the $S$-invariant does not inherit any foliation type geometric data, since the polarisation usually contains a foliated part based on $K_{\F}$, the volume form defining the S-invariant does inherit geometric data from the filtration.
\end{remark}

\begin{lemma}
\label{lem:vertical-divisors-invariant}
Let $\pi \colon (\X,\F_{\X}) \to C$ be a normal $\F$-compatible family over a smooth curve $C$, in the sense that $\F_{\X}\subset T_{\X/C}$ is a saturated integrable subsheaf. Let $E\subset \X$ be a prime divisor
contained in a fibre of $\pi$. Then $E$ is $\F_{\X}$-invariant.
\end{lemma}
\begin{proof}
Let $\eta_E$ be the generic point of $E$, and let $t$ be a local parameter on $C$ at $\pi(\eta_E)$. Since $E$ is contained in the fibre $\pi^{-1}(\pi(\eta_E))$, the ideal sheaf of $E$ at $\eta_E$ is generated by $\pi^*t$ up to a unit.

Let $\partial$ be a local section of $\F_{\X}$ near $\eta_E$. Since $\F_{\X}\subset T_{\X/C}$, we have $\partial(\pi^*t)=0$. Hence $\partial$ preserves the ideal of $E$ at $\eta_E$, i.e.
\[
\partial(\mathcal I_E)\subset \mathcal I_E.
\]
This shows that $\F_{\X}$ is tangent to $E$ at the generic point, and therefore $(\F_{\X})_x \subset T_{E,x}$ for general $x\in E$, i.e. $E$ is invariant under $\F_{\X}$.
\end{proof}

\begin{proposition}
\label{prop:special-tc-gives-F-invariant-valuation}
Let $(\pi\colon \X\to \A^1,\L,\F_{\X})$ be a nontrivial normal $\F$-compatible special test configuration of the polarised
foliated variety $(X,\F,L)$. Let $v_{\X_0}$ be the divisorial valuation over $X$ induced by $\X_0$. Then $v_{\X_0}$ is $\F$-dreamy.
\end{proposition}

\begin{proof}
Since $(\X,\L)$ is a special test configuration of $(X,L)$ on the underlying variety, the irreducible central fibre $\X_0$ defines a divisorial valuation $v_{\X_0}$ over $X$ (c.f. \cite[\S 4]{BHJ19}).

It remains to prove that this valuation is $\F$-dreamy. As before, we choose a normal
projective $\G_m$-equivariant birational model
\[
\xymatrix{
& \mathcal Y \ar[dl]_{\mu} \ar[dr]^{\rho} & \\
X\times \A^1 \ar@{-->}[rr] && \X
}
\]
which extracts a prime divisor $E\subset \mathcal Y$ corresponding to the valuation
$v_{\X_0}$. Let $\F_{\mathcal Y}$ be the birational transform of
$\F_{\X}$ on $\mathcal Y$.

By construction, $E$ lies over the central fibre $\{0\}\subset \A^1$, so $E$ is a
prime divisor contained in the fibre of the morphism $\pi\circ \rho \colon \mathcal Y \to \A^1$. Since $\F_{\X}\subset T_{\X/\A^1}$ by Definition \ref{def: foliated t.c.}, the birational transform $\F_{\mathcal Y}$ is still a saturated integrable subsheaf of
$T_{\mathcal Y/\A^1}$ on a big open set, hence $\F_{\mathcal Y}\subset T_{\mathcal Y/\A^1}$ generically along $E$. Applying Lemma~\ref{lem:vertical-divisors-invariant}, we
conclude that $E$ is $\F_{\mathcal Y}$-invariant.

Now restrict to the generic fibre over $\A^1\setminus \{0\}$. Over that punctured
curve, the test configuration is isomorphic to the product family $(X,\F)\times (\A^1\setminus\{0\})$, so the induced foliation on the generic fibre of $\mathcal Y$ is exactly the
birational transform of $\F$ over $X$. Therefore the divisor $E$, viewed as a
divisor over $X$, is invariant for the induced foliation over $X$. Hence, $v_{\X_0}=\ord_E$ is $\F$-invariant.

Finally, dreaminess of $v_{\X_0}$ with respect to $L$ is classical for special test configurations on the underlying variety \cite[\S 4]{BHJ19}. Hence $v_{\X_0}$ is $\F$-dreamy by Definition \ref{def: F dreamy}.
\end{proof}

\begin{corollary}
\label{cor:special-F-compatible-implies-F-dreamy}
Every nontrivial normal $\F$-compatible special test configuration of
$(X,\F,L)$ determines an $\F$-dreamy divisorial valuation over $X$.
\end{corollary}

\begin{proof}
This is immediate from Proposition~\ref{prop:special-tc-gives-F-invariant-valuation}.
\end{proof}

\section{\texorpdfstring{Construction of $\F$-compatible special test configurations from arbitrary $\F$-compatible test configurations}{Construction of F-compatible special test configurations from arbitrary F-compatible test configurations}}\label{sec: special t.c.}

In this section we adapt the construction developed by Li--Xu  \cite{LX14} to families of polarised adjoint Fano foliated structures
over a curve. This will allow us to show how one can construct $\F$-compatible special test configurations from arbitrary $\F$-compatible test configurations.  

In particular, let $C$ be a smooth curve, let $0\in C$ be a closed point, and let
$C^\circ:=C\setminus\{0\}$. Let $f^\circ\colon
(\X^\circ,\F^\circ,B^\circ,M^\circ,t)\to C^\circ$ be a family of lc algebraically integrable adjoint Fano foliated structures,
where $0<t<1$, such that $\F^\circ$ is compatible with the map to
$C^\circ$. We write
\[
K_{\mathcal A^\circ}
:=
tK_{\F^\circ}+(1-t)K_{\X^\circ}+B^\circ+M_{\X^\circ},
\]
and assume that the family is polarised by a relatively ample $\mathbb Q$-Cartier divisor $\L^\circ$ such that $\L^\circ\sim_{\mathbb Q,C^\circ}-K_{\mathcal A^\circ}$.

We now choose a projective completion over $C$, $f\colon
(\X,\F_{\X},B_{\X},M,t)\to C$ whose restriction over $C^\circ$ recovers the given polarised family, and we further impose that $\F_{\X}$ is an algebraically integrable foliation on $\X$. We assume throughout that $\X$ is potentially klt and that
\[
K_{\mathcal A}
:=
tK_{\F_{\X}}+(1-t)K_{\X}+B_{\X}+M_{\X}
\]
is not pseudo-effective over $C$. These conditions are essential in order to run the MMP techniques developed in \cite{CHLMSSX24, CHLMSX25}. This choice may seem arbitrary, but in future work \cite{Pap26b} we will show that $t$-K-semistable adjoint Fano foliated structures are klt, and their ambient variety is potentially klt, which justifies that choice.

We summarise the steps required to produce an $\F$-compatible special test configuration here for the convenience of the reader. The argument follows the same overall structure as the proof of
\cite[Theorem~6]{LX14}, but with the caveat that we have to run adjoint foliated MMP with scaling instead of the regular MMP.

The first step is to choose a projective compactification over $C$, pass to a $\mathbb Q$-factorial qdlt model (which gives us the analogue of $\X^{\rm lc}$ in \cite{LX14}), and then run the  $K_{\mathcal A}$-MMP with scaling over $C$. At the threshold stage we reach a model $(\X^j,\F^j,B^j,M,t)\to C$ such that $-K_{\mathcal A^j}\sim_{\mathbb Q,C}L^j$, with $L^j$ nef and big over $C$. This allows one to define the relative anti-adjoint model
\[
\X^{ac}:=
\operatorname{Proj}_C R(\X^j/C,-K_{\mathcal A^j}).
\]
which is the foliated analogue of the similar variety defined in \cite{LX14}.

The second step is to start from $f^{ac}\colon (\X^{ac},\F^{ac},B^{ac},M,t)\to C$ where $-K_{\mathcal A^{ac}}$ is relatively ample over $C$, and after a finite base change $\phi\colon C'\to C$ take a semistable log resolution $\pi\colon Y\to \widetilde{\X}^{ac}:=\X^{ac}\times_C C'$ of the pair $(\widetilde{\X}^{ac},\widetilde{\X}^{ac}_0)$. We can then pull back the polarisation
$-K_{\mathcal A^{ac}}$ to $Y$, perturb it slightly by a
$\pi$-ample exceptional divisor, and consider the birational transform of the foliation together with the corresponding adjoint divisor on $Y$.

The third step is to replace this lc adjoint foliated structure by a klt
rational one using the perturbation theorem for algebraically integrable
adjoint foliated structures, namely \cite[Lemma~3.29]{CHLMSX25}. Once this is done, the $K_{\mathcal A}$-MMP over $\widetilde{\X}^{ac}$ with scaling is available by the klt MMP theorems in \cite{CHLMSX25}.

The fourth step is the birational heart of the argument. Ordering carefully the coefficients of the vertical components we run the adoint foliated MMP so that all unwanted vertical divisors are contracted, while one distinguished component survives. The output is a model
\[
(\X',\F',B',M,t)\to C'
\]
such that $K_{\mathcal A'}+\L'\sim_{\mathbb Q,C'}0$ and $\L'$ is nef and big over $C'$. Here the coefficient-ordering and support argument are the same as in \cite[Theorem~6]{LX14}, while the existence and termination of the required MMP come from \cite{CHLMSX25}.

The final step is to define
\[
\X^s:=
\operatorname{Proj}_{C'}R(\X'/C',\L')
=
\operatorname{Proj}_{C'}R(\X'/C',-K_{\mathcal A'}).
\]
and notice that we also obtain an induced foliation $\F^s$. Over the punctured curve, this construction recovers the original polarised family (and foliation), and the surviving vertical divisor is identified with the special fibre of the extension withdiscrepancy over
$\widetilde{\X}^{ac}$ equal to $0$.

\subsection{qdlt model and ambient klt property}\label{sec: qdlt model}

The first step is to replace the given family by a $\mathbb Q$-factorial qdlt
model.

\begin{proposition}
\label{prop:prepared-qdlt-model}
After replacing $(\X,\F_{\X},B_{\X},M,t)\to C$ by a birational model over $C$, there exists a projective
birational morphism
\[
h\colon
(\X^0,\F^0,B^0,M,t)\to
(\X,\F_{\X},B_{\X},M,t)
\]
such that
\begin{enumerate}
    \item $(\X^0,\F^0,B^0,M,t)$ is $\mathbb Q$-factorial and qdlt;
    \item $K_{\mathcal A^0}=h^*K_{\mathcal A}$;
    \item the transformed foliation $\F^0$ remains algebraically
    integrable and compatible with the map to $C$;
    \item the pair \((\X^0,\X^0_0)\) is log canonical.
\end{enumerate}
\end{proposition}

\begin{proof}
This is the qdlt modification theorem for algebraically integrable adjoint
foliated structures, namely
\cite[Theorem~3.6]{CHLMSSX24}, together with
\cite[Definition~3.7]{CHLMSSX24}. Applying
\cite[Theorem~3.6]{CHLMSSX24} over the base curve $C$ yields
the required $\mathbb Q$-factorial qdlt model and the crepant pullback
identity for the adjoint divisor. Algebraic integrability and compatibility
with the map to $C$ are preserved under birational transform, since they are
determined on the common big open subset where the birational map is an
isomorphism. In particular, after forgetting the foliation, the underlying generalised pair $(\X^0,B^0+\red(\X^0_0),M_{X^0})$ is qdlt, hence log canonical. Since \(B^0\geq 0\), decreasing the boundary preserves log canonicity, and therefore \((\X^0,\X^0_0)\) is log canonical.
\end{proof}

\begin{lemma}
\label{lem:prepared-model-klt}
With the notation of Proposition~\ref{prop:prepared-qdlt-model}, the ambient
variety $\X^0$ is klt.
\end{lemma}

\begin{proof}
This is recorded immediately after the proof of
\cite[Theorem~3.6]{CHLMSSX24}: for a $\mathbb Q$-factorial qdlt
adjoint foliated structure, the underlying generalised pair is qdlt, hence the
ambient variety is klt.
\end{proof}

\begin{corollary}
\label{cor:prepared-model-klt-regime}
The model $(\X^0,\F^0,B^0,M,t)/C$ is an lc algebraically integrable adjoint foliated structure on a
$\mathbb Q$-factorial klt variety.
\end{corollary}

\begin{proof}
By Proposition~\ref{prop:prepared-qdlt-model}, the model is
$\mathbb Q$-factorial and qdlt, hence lc. By
Lemma~\ref{lem:prepared-model-klt}, the ambient variety is klt.
\end{proof}

Since the family is polarised over $C^\circ$, we denote by $L^0$ the
birational transform of $\L^\circ$ on $\X^0$. By
construction we have $L^0|_{\X^0_\eta}\sim_{\mathbb Q}-K_{\mathcal A^0_\eta}$.

\begin{remark}
    $(\X^0,\F^0,B^0,M,t)/C$ plays the role of $\X^{lc}$ in \cite{LX14}.
\end{remark}

We will use the following relative lemma over the base curve, following \cite[Lemma~1]{LX14}.

\begin{lemma}
\label{lem:vertical-nef}
Let $\pi\colon X\to C$ be a projective dominant morphism from a normal variety to a smooth curve, and
let $D$ be a $\mathbb Q$-Cartier $\mathbb Q$-divisor on $X$. Assume that
\begin{enumerate}
    \item $D$ is nef over $C$;
    \item $D|_{X_\eta}\sim_{\mathbb Q}0$ where $\eta$ is the generic point of $C$.
\end{enumerate}
Then $D\sim_{\mathbb Q,C}0$.
\end{lemma}

\begin{proof}
Choose $m>0$ sufficiently divisible such that $mD$ is Cartier and $(mD)|_{X_\eta}\sim 0$. Then there exists a rational function $\varphi\in K(X)^\times$ such that $E:=mD-\operatorname{div}(\varphi)$ is supported on finitely many fibres of $\pi$. Since $E\sim mD$, the divisor $E$ is also nef over $C$.

Thus we have $E=\sum_{p\in\Sigma}E_p$ where each $E_p$ is supported on the fibre $X_p$, for a finite set $\Sigma\subset C$. Let $A$ be an ample Cartier divisor on $X$. Distinct
fibres are disjoint, so
\[
E_p\cdot E_q\cdot A^{\dim X-2}=0
\qquad (p\neq q),
\]
hence
\[
E^2\cdot A^{\dim X-2}
=
\sum_{p\in\Sigma}E_p^2\cdot A^{\dim X-2}.
\]

By \cite[Lemma~1]{LX14}, for each $p$, $E_p^2\cdot A^{\dim X-2}\le 0$ with equality only if $E_p$ is proportional to the full fibre $X_p$. Since $E$ is nef over $C$, we have
\[
E\cdot E_p\cdot A^{\dim X-2}\ge 0
\]
for every $p$. As cross-terms vanish, this gives
\[
E_p^2\cdot A^{\dim X-2}\ge 0
\]
for all $p$. Hence equality holds for every $p$, and the equality case of
\cite[Lemma~1]{LX14} implies that $E_p=t_pX_p$ for some $t_p\in\mathbb Q$. Therefore
\[
E=\sum_{p\in\Sigma}t_pX_p=\pi^*\Bigl(\sum_{p\in\Sigma}t_pp\Bigr),
\]
so $E\sim_C0$. Since $E\sim mD$, it follows that $D\sim_{\mathbb Q,C}0$.
\end{proof}

\subsection{MMP with scaling}\label{sec: mmp with scaling}

We choose an ample  $\Q$-divisor $H^0$ over $C$ on $\X^0$ such that $(K_{\mathcal A^0}+H^0)|_{\X^0_\eta}\sim_{\mathbb Q}0$. We now apply the MMP techniques from the adjoint foliated theory
\cite{CHLMSSX24,CHLMSX25} in order to pass from the lc setting to the klt
rational setting while preserving the adjoint numerical class up to adding an
ample divisor. Let
\[
\mathcal A^0/C=
(\X^0,\F^0,B^0,M,t)/C
\]
be the model we have already constructed, and let $P$ be an ample $\mathbb R$-divisor over
$C$ on $\X^0$. We now make the definition of the divisor $H^0$ explicit. We define
\[
H^0 := L^0 - \frac{1}{l+1}(K_{\cA^0}+L^0).
\]
Since $L^0$ is ample over $C$, for $l$ sufficiently large the divisor $H^0$ is ample over $C$. Moreover, using $L^0|_{(\X^0)_\eta} \sim_{\Q} -K_{\cA^0}|_{(\X^0)_\eta}$ we obtain
\[
(K_{\cA^0}+H^0)|_{(\X^0)_\eta}
\sim_{\Q}
\frac{l}{l+1}(K_{\cA^0}+L^0)|_{(\X^0)_\eta}
\sim_{\Q} 0.
\]
Also, $K_{\cA^0}+(l+1)H^0 \sim_{\Q,C} lL^0$, so in particular $K_{
\cA^0}+(l+1)H^0$ is ample over $C$.

\begin{proposition}
\label{prop:klt-perturbation}
There exist
\begin{enumerate}
    \item an ample $\mathbb R$-divisor $H$ over $C$, and
    \item a klt algebraically integrable adjoint foliated structure
    \[
    \mathcal A'/C:=
    (\X^0,\F^0,B',M',t')/C
    \]
\end{enumerate}
such that $B'$ has rational coefficients, $t'\in\mathbb Q\cap[0,1)$, and
\[
K_{\mathcal A^0}+P\sim_{\mathbb R,C}K_{\mathcal A'}+H.
\]
\end{proposition}

\begin{proof}
This is exactly \cite[Lemma~3.29]{CHLMSX25}, applied
over $U=C$, where the potentially klt assumption on the ambient variety is
satisfied by Corollary~\ref{cor:prepared-model-klt-regime}.
\end{proof}

Note that, after applying Proposition~\ref{prop:klt-perturbation}, the resulting perturbed klt rational adjoint foliated structure satisfies the hypotheses of \cite[Theorem~7.5]{CHLMSX25} for the choice $P^0:=\varepsilon H^0$, $C^0:=(1-\varepsilon)H^0$ where $0<\varepsilon\ll1$. We then run the $(K_{\mathcal A^0}+P^0)$-MMP over $C$ with scaling of $C^0$, and let $\phi_i\colon \X^{i-1}\dashrightarrow \X^i$ be the induced birational maps, with scaling numbers $\lambda_i$, where 
\[\lambda_i = \min\{\lambda|K_{\cA^i}+P^i+\lambda C^i \text{ is nef over }C \}\]
where $P^i$, $C^i$, $H^i$ denote the pushforwards of $P^0$, $C^0$, $H^0$ to $\X^i$, and define $\mu_i:=\varepsilon+\lambda_i(1-\varepsilon)$. Here, for every $i$, the divisor $K_{\mathcal A^i}+P^i+\lambda_i C^i$ is nef over $C$.

\begin{lemma}
\label{lem:normalisation-scaling-package}
For every $i$, $K_{\mathcal A^i}+\mu_iH^i$ is nef over $C$. Moreover, $\mu_i=1$ if and only if $\lambda_i=1$ and in this case $K_{\mathcal A^i}+H^i$ is nef over $C$.
\end{lemma}

\begin{proof}
The divisor $K_{\mathcal A^i}+P^i+\lambda_iC^i$ is nef over $C$. Since $P^i=\varepsilon H^i$ and $C^i=(1-\varepsilon)H^i$, we have $K_{\mathcal A^i}+P^i+\lambda_iC^i = K_{\mathcal A^i}+\mu_iH^i$ proving the first claim. The equivalence $\mu_i=1$ if and only if $\lambda_i=1$ is immediate from the definition of $\mu_i$.
\end{proof}

With this explicit choice of $H^0$, the standard argument of \cite[Section 4.1]{LX14} applies verbatim in the present setting. In particular, the scaling sequence reaches the value $1$, and we denote by $j+1$ is the first index such $\lambda_{j+1} = 1$. Note that $\lambda_j >1$. In particular, the birational maps $\phi_i$ are isomorphisms over the generic point $\eta$ of $C$ for all $i \leq j+1$, as, for every $i$ we have $K_{\cA^i}+P^i+\lambda_i C^i = K_{\cA^i}+\mu_i H^i$, and on the generic fibre this restricts to
\[
(K_{\cA^i}+\mu_i H^i)|_{(\X^i)_\eta}
\sim_{\Q}
(\mu_i-1)H^i|_{(\X^i)_\eta}.
\]
Hence, as long as $\mu_i>1$, the restriction of $K_{\cA_i}+P^i+\lambda_i C^i$ to the generic fibre is ample, so the corresponding extremal contractions are vertical over $C$. We now record the formal preservation of the generic-fibre normalisation. 
 
\begin{lemma}
\label{lem:generic-triviality-preserved}
 We have $L^i|_{\X^i_\eta}\sim_{\mathbb Q}-K_{\mathcal A^i_\eta}$.
In particular, $(K_{\mathcal A^i}+H^i)|_{\X^i_\eta}\sim_{\mathbb Q}0$.
\end{lemma}

\begin{proof}
Since $\phi_i$ is an isomorphism over the generic point of $C$, the generic fibres $\X^0_\eta$ and $\X^i_\eta$ are naturally
identified. Under this identification, the restrictions of $L^0$ and
$K_{\mathcal A^0}$ coincide with those of $L^i$ and $K_{\mathcal A^i}$,
respectively. The same argument applies to $H^0$.
\end{proof}

Now, let us denote, for $\lambda >1$,  
\[\L_{\lambda}:= \frac{1}{\lambda-1}(K_{\X^0,C}^{[t]}+\lambda H^0)\]
with corresponding pushforwards $\L^j_\lambda$.

\begin{proposition}
\label{prop:threshold-stage-package}
$-K_{\mathcal A^j}\sim_{\Q,C} \L^j_{\lambda_j}$ is big and nef over $C$.
\end{proposition}

\begin{proof}
    Recall that $K_{\mathcal A^j}+ \L^j_{\lambda_j}$ is relatively nef over $C$ its restriction to the generic fiber is trivial by Lemma \ref{lem:generic-triviality-preserved}. Furthermore, by the definition of $\L^j_{\lambda_j}$ and by applying Lemma~\ref{lem:vertical-nef} we have $K_{\mathcal A^j}+ \L^j_{\lambda_j}\sim_{\Q,C} 0$. Furthermore $\L^j_{\lambda_j}$ is proportional to $K_{\cA^j} + \lambda_jH^j$ which is big because $\lambda_j>1$, and nef by construction. 
\end{proof}

By Proposition~\ref{prop:threshold-stage-package}, the divisor $-K_{\mathcal A^j}$ is nef and big over $C$. By the base-point-freeness theorem for klt algebraically integrable adjoint
foliated structures, \cite[Theorem~2.1.2]{CHLMSX25}, it is also semiample over $C$. This allows us to pass to the relative
anti-canonical model
\[
\X^{ac}:=
\operatorname{Proj}_C R\bigl(\X^j/C,-K_{\mathcal A^j}\bigr).
\]
We refer to $\X^{ac}$ as the relative anti-adjoint model.

\begin{lemma}\label{lem:anti-adjoint-model-lc-central-fibre}
With notation as above, the pair $(\X^{ac},\X^{ac}_0)$ is log canonical.
\end{lemma}

\begin{proof}
By Proposition \ref{prop:prepared-qdlt-model}, the pair $(\X^0,\X^0_0)$ is log canonical. Since $\X^0_0=(f^0)^*(0)$ the \(K_{\cA^0}\)-MMP over \(C\) is also a \((K_{\cA^0}+\X^0_0)\)-MMP. Hence log canonicity of the pair with the central fibre is preserved along the sequence, so $(\X^j,\X^j_0)$ is log canonical. The morphism $\X^j \to \X^{ac}$ is the relative anti-adjoint contraction, and \(\X^j_0\) pushes forward to
\(\X^{ac}_0\). Therefore $(\X^{ac},\X^{ac}_0)$ is log canonical.
\end{proof}

\subsection{Adjoint Fano foliated extension}\label{sec: adjoint fano extension}

We now show a general result of the existence of the log canonical modification for the variety fibered over a curve.

\begin{proposition}
\label{prop:foliated-prop1}
Let $f\colon (Y,\F_Y,\Delta_Y,M,t)\to C$ be a projective klt algebraically integrable adjoint foliated structure over a
smooth curve $C$, and let $L_Y$ be an $f$-ample $\mathbb Q$-divisor.
Assume that $K_{\mathcal A_Y}+L_Y \sim_{\mathbb Q,C} E=E_h+E_v\ge 0$ where:
\begin{enumerate}
    \item $E_h$ is exceptional over a birational contraction $Y\dashrightarrow X$;
    \item the vertical part can be written
    \[
    E_v=\sum_i a_iE_i,
    \]
    where $\Supp(E_v)$ does not contain any whole fibre.
\end{enumerate}
Then the $K_{\mathcal A_Y}$-MMP over $C$ with scaling of $L_Y$ terminates
with a model
\[
f'\colon (\X',\F',\Delta',M,t)\to C
\]
such that
\begin{enumerate}
    \item $K_{\mathcal A'}+\L'\sim_{\mathbb Q,C}0$;
    \item the divisors contracted by $Y\dashrightarrow \X'$ are precisely $\Supp(E)$;
    \item $\L'$ is nef and big over $C$.
\end{enumerate}
\end{proposition}

\begin{proof}
The proof is the same as \cite[Proposition~1]{LX14} with \cite[Theorem~5]{LX14} replaced by \cite[Theorem~8.1]{CHLMSX25}. By assumption $K_{\mathcal A_Y}+L_Y$ is pseudo-effective over $C$, and since $L_Y$ is ample over $C$, the $K_{\mathcal A_Y}$-MMP with scaling of $L_Y$ exists and terminates with a good minimal model by \cite[Theorem~2.1.1]{CHLMSX25}. At the pseudo-effective threshold, one obtains a model $\X'$ with $K_{\mathcal A'}+\L'\sim_{\mathbb Q,C}0$.

The support statement is exactly the same as in \cite[Proposition~1]{LX14}: the horizontal part $E_h$ is exceptional by construction, while the vertical part $E_v$ is of insufficient fibre type, hence is contained in the relative stable base locus and is therefore contracted in the minimal model. Finally, since $K_{\mathcal A'}+\L'\sim_{\mathbb Q,C}0$ and $\L'$ appears as the scaling divisor at the threshold model, $\L'$ is nef over $C$; it is big because it is birationally the transform of an ample divisor. 
\end{proof}
\begin{remark}
    The model $\X'$ above plays the same role as the model $\X'$ in \cite{LX14}.
\end{remark}

We are now in a position to show that we obtain nicely behaved adjoint Fano foliated extensions.

\begin{theorem}
\label{thm:512A-special-extension-after-base-change}
Let $f^{ac}\colon (\X^{ac},\F^{ac},B^{ac},M,t)\to C$ be a normal compactification of the polarised family over $C^\circ$ such that
$-K_{\mathcal A^{ac}}$ is relatively ample over $C$ Then, after a finite
base change $\phi\colon C'\to C$ there exists a family
\[
f^s\colon (\X^s,\F^s,B^s,M,t)\to C'
\]
of adjoint Fano foliated structures such that:
\begin{enumerate}
    \item $f^s$ agrees with the pullback of the original family over
    $\phi^{-1}(C^\circ)$;
    \item every fibre of $f^s$ is an adjoint Fano foliated structure.
\end{enumerate}
\end{theorem}

\begin{proof}
After the finite base change $\phi\colon C'\to C$, we choose a semistable log resolution (see Proposition \ref{prop:prepared-qdlt-model} and Corollary \ref{cor:prepared-model-klt-regime})
\[
\pi\colon Y\to \widetilde{\X}^{ac}:=\X^{ac}\times_C C'
\]
of the pair $(\widetilde{\X}^{ac},\widetilde{\X}^{ac}_0)$, (following the idea of proof of \cite[Theorem~6]{LX14}). Let $A$ be a $\pi$-ample exceptional divisor and write $A=A_1+A_2$, where $A_2$ consists precisely of
the vertical components over the special point. Over the punctured curve, let
$\F_Y$ be the birational transform of $\F^{ac}$, and write
\[
\pi^*(K_{\mathcal A^{ac}})|_{Y^\circ}+F^\circ
=
K_{\mathcal A_{Y^\circ}}+\Gamma^\circ,
\]
where $F^\circ,\Gamma^\circ\ge 0$ have no common components. Since the generic
fibres are klt, the coefficients of $\Gamma^\circ$ are $<1$. Let $\Gamma$
and $F$ denote the closures of $\Gamma^\circ$ and $F^\circ$ in $Y$.

We set $L_Y:=\pi^*(-K_{\mathcal A^{ac}})+\varepsilon A$ for $0<\varepsilon\ll 1$. Since $-K_{\mathcal A^{ac}}$ is relatively ample
over $C$, the divisor $L_Y$ is ample over $C'$. After a small perturbation
and reordering the vertical components $E_1,\dots,E_k$ of $Y_0$, we may write
\[
K_{\mathcal A_Y}+L_Y+\Gamma
\sim_{\Q,C'}
(\varepsilon A_1+F)+(\varepsilon A_2+B),
\]
with $B+\varepsilon A_2=\sum_{j=1}^k a_jE_j$ where $a_j>a_1$ for $j>1$. Let $G$ be the sum of the $\pi$-exceptional divisors whose centres lie over the nondegenerate locus. Choosing $0<\varepsilon\ll\delta\ll1$, we may assume that $\delta G+\varepsilon A_1\ge 0$ and that its support is exactly $G$. Adding a suitable multiple of the special
fibre, we may also assume $a_1=0$. Then
\[
K_{\mathcal A_Y}+L_Y+\Gamma+\delta G
\sim_{\Q,C'}
(\delta G+\varepsilon A_1+F)+(\varepsilon A_2+B-a_1Y_0)\ge 0,
\]
and the support of the horizontal part over $C^\circ$ contains all the
exceptional divisors of the birational morphism $Y^\circ\to \widetilde{\X}^{ac,\circ}$.

We now perturb to a klt rational adjoint foliated structure using
\cite[Lemma~3.29]{CHLMSX25}. Since the ambient variety is potentially klt, this produces a klt algebraically integrable adjoint foliated structure on $Y$ with the same numerical class up to an ample perturbation. Applying Proposition~\ref{prop:foliated-prop1}, the $K_{\mathcal A_Y}$-MMP over $\widetilde{\X}^{ac}$ with scaling of $L_Y$ terminates with a model
\[
f'\colon (\X',\F',B',M,t)\to C'
\]
such that $K_{\mathcal A'}+\L'\sim_{\Q,C'}0$ and $\L'$ is nef and big over $C'$.

By the base-point-freeness theorem for klt algebraically integrable adjoint
foliated structures, \cite[Theorem~2.1.2]{CHLMSX25}, the divisor $\L'$ is
semiample over $C'$. Therefore we can define
\[
\X^s:=\Proj_{C'} R(\mathcal \X'/C',\L').
\]
Since $K_{\mathcal A'}+\L'\sim_{\Q,C'}0$, this is equivalently $\X^s=\Proj_{C'} R(\X'/C',-K_{\mathcal A'})$.

Over the punctured curve, the birational map $Y^\circ\dashrightarrow \mathcal \X'^{\circ}$ contracts exactly the same divisors as $Y^\circ\to \widetilde{\X}^{ac,\circ}$, hence $\mathcal \X'^{\circ}$ is isomorphic to the original family in codimension
one. Taking relative $\Proj$ therefore recovers the original polarised family over $\phi^{-1}(C^\circ)$. This proves~(1).

For the next step, we choose a general divisor in $|L_Y|_{\Q}$. The MMP above is also a $(K_{\mathcal A_Y}+\Gamma+\delta G+Y_0+L_Y)$-MMP so the output satisfies that $(\X',\X'_0+'L)$ is dlt. Since $\X'$ is $\Q$-factorial and $\X'_0$ is
irreducible by construction, it follows that $(\X',\X'_0)$ is plt. Therefore $(\X^s,\X^s_0)$ is plt, and by adjunction the special fibre $\X^s_0$ is klt. On the
other hand, $-K_{\mathcal A^s}$ is relatively ample over $C'$ by construction, hence its restriction to every fibre is ample. For the general fibres this recovers the original adjoint Fano foliated structures, while for the special fibre the above plt statement yields the required klt singularities. Therefore every fibre of $f^s$ is an adjoint Fano foliated structure. This proves~(2).
\end{proof}

\begin{lemma}
\label{lem:512B-identification-surviving-component}
With notation as in the proof of
Theorem~\ref{thm:512A-special-extension-after-base-change}, the special fibre $\X^s_0$ is
the birational image of the distinguished divisor $E_1$.
\end{lemma}

\begin{proof}
By construction, the divisor $K_{\mathcal A_Y}+L_Y+\Gamma+\delta G$ is $\Q$-linearly equivalent over $C'$ to an effective divisor whose vertical
part is
\[
\varepsilon A_2+B-a_1Y_0=\sum_{j=2}^k a_jE_j,
\]
where each coefficient $a_j$ is strictly positive for $j>1$, while the
coefficient of $E_1$ is zero.

The $K_{\mathcal A_Y}$-MMP with scaling of $L_Y$ is run relative to
$\widetilde{\X}^{ac}$, and the induced birational contractions eliminate precisely the vertical divisors appearing in this positive vertical part via the Negativity Lemma. Hence every divisor $E_j$ with $j>1$ is contracted in the passage from $Y$ to
$\X'$. Since $\X^s$ is obtained from $\X'$ by the
semiample contraction associated to $\L'$, none of these contracted divisors
reappears on $\X^s$.

On the other hand, the divisor $E_1$ does not appear in the positive vertical
part, so it is not contracted by this process. It follows that the birational
image of $E_1$ on the final model $\X^s$ is exactly the special fibre
$\X^s_0$. This proves the claim.
\end{proof}

\begin{proposition}
\label{prop:512C-discrepancy-zero-special-fibre}
The special fibre $\X^s_0$, regarded as a divisorial valuation over $\widetilde{\X}^{ac}:=\X^{ac}\times_C C'$ satisfies
\[
a(\X^s_0;\widetilde{\X}^{ac},\F^{ac},B^{ac},M,t)=0.
\]
\end{proposition}

\begin{proof}
We retain the notation from the proof of
Theorem~\ref{thm:512A-special-extension-after-base-change}. By construction we
have
\[
K_{\mathcal A_Y}+L_Y+\Gamma
\sim_{\Q,C'}
(\varepsilon A_1+F)+(\varepsilon A_2+B),
\]
and after adding $\delta G$ and subtracting $a_1Y_0$, we obtain
\[
K_{\mathcal A_Y}+L_Y+\Gamma+\delta G
\sim_{\Q,C'}
(\delta G+\varepsilon A_1+F)+(\varepsilon A_2+B-a_1Y_0).
\]
By our normalisation $a_1=0$, the vertical part is
\[
\varepsilon A_2+B=\sum_{j=1}^k a_jE_j,
\qquad
a_j\ge 0,
\]
with equality only for $j=1$. Since
$(\widetilde{\X}^{ac},\widetilde{\X}^{ac}_0)$ is log canonical,
the coefficients of the exceptional divisors in this pullback expression are
nonnegative, and the coefficient of $E_1$ is exactly zero. Furthermore, $E_1$ is invariant, hence,
\[
a(E_1;\widetilde{\X}^{ac},\F^{ac},B^{ac},M,t)=0.
\]

By Lemma~\ref{lem:512B-identification-surviving-component}, the special fibre
$\X^s_0$ is the birational image of $E_1$. Hence $\X^s_0$
defines the same divisorial valuation over $\widetilde{\X}^{ac}$ as
$E_1$. Therefore
\[
a(\X^s_0;\widetilde{\X}^{ac},\F^{ac},B^{ac},M,t)=0,
\]
as claimed.
\end{proof}

\begin{corollary}
\label{cor:adjoint-fano-foliated-extension}
Starting with a family $f^\colon (\X^,\F,B,M,t)\to C$  of adjoint Fano foliated structures, after a finite base change $\phi\colon C'\to C$, there exists a family
\[
f^s\colon (\X^s,\F^s,B^s,M,t)\to C'
\]
of adjoint Fano foliated structures such that:
\begin{enumerate}
    \item it is isomorphic to the pullback of the original family over
    $\phi^{-1}(C^\circ)$;
    \item every fibre is an adjoint Fano foliated structure;
    \item the special fibre $\X^s_0$ satisfies $a(\X^s_0;\widetilde{\X}^{ac},\F^{ac},B^{ac},M,t)=0$.
\end{enumerate}
\end{corollary}

\begin{proof}
We apply Theorem~\ref{thm:512A-special-extension-after-base-change} and the constructions of sections \ref{sec: qdlt model} and \ref{sec: mmp with scaling} to obtain the family $f^s\colon (\X^s,\F^s,B^s,M,t)\to C'$. By Lemma~\ref{lem:512B-identification-surviving-component}, the special fibre
$\X^s_0$ is the birational image of the distinguished divisor $E_1$,
and by Proposition~\ref{prop:512C-discrepancy-zero-special-fibre} this divisor has
discrepancy zero over $\widetilde{\X}^{ac}$. This gives~(3), while
(1) and (2) are exactly the conclusions of
Theorem~\ref{thm:512A-special-extension-after-base-change}.
\end{proof}

\subsection{\texorpdfstring{Application to $\F$-compatible test configurations}{Application to F-compatible test configurations}}

We begin with a compatibility Lemma showing that $\F$-compatible test configurations satisfy the assumptions of Sections \ref{sec: qdlt model}-\ref{sec: adjoint fano extension}.

Let $(X,\F,t)$ be an adjoint Fano foliated structure with $0<t<1$, and let $(\pi:\X\to \A^1,\mathcal{L},\F_{\X})$ be a normal $\F$-compatible test configuration for $(X,\F,L)$, where $L\sim_{\Q}-K^{[t]}_{X,\F}$. We keep the notation of Remark \ref{rem: alg integrable condition} and assume that $\F_{\X}$ is algebraically integrable. Now, let $(\bar{\X},\bar{\F},\bar{\mathcal{L}})\to \PP^1$ be an $\infty$-trivial compactification.

\begin{lemma}\label{lem:bridge-test-config-section5}
After a finite base change and a $\mathbb{G}_m$-equivariant semistable resolution, we obtain a projective morphism $f:(Y,\F_Y)\to \PP^1$ such that:
\begin{enumerate}
    \item over $\PP^1\setminus \{0\}$, the family is a family of algebraically integrable adjoint Fano foliated structures polarized by $L$;
    \item $\F_Y$ is algebraically integrable and compatible with the morphism to $\PP^1$;
    \item $Y$ is potentially klt;
    \item $K^{[t]}_{Y,\F_Y}$ is not pseudo-effective over $\PP^1$.
\end{enumerate}
\end{lemma}

\begin{proof}
By definition, over $\A^1\setminus\{0\}$ the test configuration is $\mathbb{G}_m$-equivariantly isomorphic to the product family $(X,\F,L)\times (\A^1\setminus\{0\})$. Since $\F_{\X}$ is assumed to be algebraically integrable, it follows that after passing to an $\infty$-trivial compactification, the restriction of $(\bar{\X},\bar{\F},\bar{\mathcal{L}})\to \PP^1$ over $\PP^1\setminus\{0\}$ is a family of algebraically integrable adjoint Fano foliated structures polarized by $L$. This proves~(1).

After a finite base change and a $\mathbb{G}_m$-equivariant semistable resolution, we obtain a projective morphism $f:Y\to \PP^1$ together with the induced foliation $\F_Y$. By construction, $\F_Y$ is algebraically integrable, and it remains compatible with the morphism to the base. This gives~(2).

Now, (3) follows by \cite[Theorem~3.6]{CHLMSSX24}. It remains to prove~(4). Suppose by contradiction that $K^{[t]}_{Y,\F_Y}$ is pseudo-effective over $\PP^1$. Then its restriction to the generic fiber $Y_\eta$ is pseudo-effective. On the other hand, by construction the generic fiber is isomorphic to $X$, and $K^{[t]}_{Y,\F_Y}\big|_{Y_\eta}=K^{[t]}_{X,\F}$. Since $(X,\F,t)$ is adjoint Fano, the divisor $-K^{[t]}_{X,\F}$ is ample. Hence $K^{[t]}_{X,\F}$ is not pseudo-effective, a contradiction. Therefore $K^{[t]}_{Y,\F_Y}$ is not pseudo-effective over $\PP^1$, proving~(4).
\end{proof}

In words, Lemma \ref{lem:bridge-test-config-section5} shows that starting with an $\F$-compatible test configuration, we can obtain naturally the model $f:Y\to \PP^1$ where the birational procedure of Sections \ref{sec: qdlt model}-\ref{sec: adjoint fano extension} applies to $f:Y\to \PP^1$, since all the starting hypotheses of Section \ref{sec: special t.c.} apply.

We now study how $\G_m$-equivariance affects birational transformations.

\begin{lemma}
\label{lem:514A-Gm-equivariant-birational-transport}
Let $(\X,\F_{\X},\L)\to \mathbb P^1$ be an $\F$-compatible compactified test configuration. After equivariant base change and equivariant semistable resolution, the induced foliation on the resolution is $\G_m$-equivariant. Moreover, each divisorial contraction and flip in the adjoint foliated MMP can be performed $\G_m$-equivariantly, and the birational transform foliation remains $\G_m$-equivariant and compatible with the morphism to the base.
\end{lemma}
\begin{proof}
    The proof of this Lemma follows directly by the argument in \cite[Proof of Theorem 7]{LX14}, noting that semistable reduction via base change is taken $\G_m$-equivariantly, that the foliation is the birational transform of a $\G_m$-equivariant saturated subsheaf and that each extremal ray and each contraction/flip is canonical from the graded ring / relative $\Proj$ construction, hence inherits the $\G_m$-action.
\end{proof}

We will now prove the following key theorem.

\begin{theorem}\label{thm: construction of special t.c.}
    Let $(X,\F)$ be an adjoint Fano foliated structure and $(\X,\F_{\X})$ an $\F$-compatible test configuration of $(X,\F, -K_{X,\F}^{[t]})$. Then, after a base change $z\mapsto z^m$, there exists a normal $\F$-compatible special test configuration $(\X^{s},\L^{s},\F_{\X^{s}})$ of $(X,-K^{[t]}_{X,\F},\F)$.
\end{theorem}
\begin{proof}
    By \cite[Lemma~5]{LX14}, there exists a base change $z^m : \A^1 \rightarrow \A^1$ and a semistable family $\Y$ over $\A^1$ with a morphism $\pi : \Y \rightarrow X \times_{\A^1, z^m}\A^1$ which is a log resolution of $(\tilde{\X},\tilde{\X}_{0})$ where $\tilde{\X}$ is the normalisation of $X \times_{\A^1, z^m}\A^1$. We define $\F_{\Y}$ as the induced foliation on $\Y$ from $\F_{\X}$. This implies that resolution of singularities and semistable reduction can be obtained $\G_m$-equivariantly. In particular, starting from an $\infty$-trivial compactfication $(\bar{\X}, \F_{\bar{X}}, \bar{\L})\rightarrow \PP^1$, we can perform a base change $d : \PP^1 \rightarrow \PP^1$ such that $\bar{\X}_d := \bar{\X}\times_{d, \PP^1} \PP^1 \rightarrow \PP^1$ via the second projection admits a $\G_m$-equivariant semi-stable reduction $\pi : \Y \rightarrow \bar{\X}_d $, with induced foliation $\F_{\Y}$. Lemma \ref{lem:bridge-test-config-section5} now allows us to perform the birational steps detailed in Sections \ref{sec: qdlt model}-\ref{sec: adjoint fano extension}, noting that by Lemma \ref{lem:514A-Gm-equivariant-birational-transport} all steps can be performed $\G_m$-equivariantly. In particular, by Corollary \ref{cor:adjoint-fano-foliated-extension} we obtain a special, $\F$-compatible test configuration $(\X^{s},\F_{\X}^{s})$ as required.
\end{proof}

\section{\texorpdfstring{Reduction to $\F$-compatible special test configurations}{Reduction to F-compatible special test configurations}}\label{sec: reduction to special t.c.}

In this section we will show that it suffices to check K-stability for adjoint Fano foliated structures against $\F$-compatible special test configurations.

We will achieve this by showing that by starting with an $\F$-compatible test configuration, the Donaldson-Futaki invariant of Definition \ref{def:DFwt-corrected} decreases along each birational step detailed in Section \ref{sec: special t.c.}.

First we fix some notation. For the family $(\X,\F,\L)\rightarrow C$, we denote by 
\[\DF^{[t]}((\X, \F,\L)/C):= \frac{1}{V}\left(
\frac{n}{n+1}\,\L^{n+1}
+
K^{[t]}_{\X/C}\cdot \L^n \right),\]
as in Definition \ref{def:DFwt-corrected}. When $C$ is explicit, we may drop it from the notation and just write $\DF^{[t]}(\X, \F,\L)$.

\begin{proposition}\label{prop: df decreases in first step}
    If $\X^0$ is not isomorphic to $\X$ then we can choose a polarisation $\L^0$ on $\X^0$ such that 
    $$\DF^{[t]}((\X^0, \F^0, \L^0)/C)\leq \DF^{[t]}((\X, \F,\L)/C).$$
\end{proposition}
\begin{proof}
    We argue as in \cite[proof of Proposition 3]{LX14}. We choose the relatively $\pi^{0}$-ample $\Q$-divisor $E = K_{\X^0,\F^0}^{[t]}+\pi^{0*}\L$ such that for sufficiently small rational $\epsilon$, $\L^0_r: = \pi^{0*}\L+rE$ is ample for all $0<r<\epsilon$. In particular, $(\X^0,\F^0,\L^0)\rightarrow C$ is also an $\F^0$-compatible family of adjoint Fano foliated structures. We now compute the derivative of $\DF^{[t]}$ at $r_0\in (0,\epsilon)$. Using Definition \ref{def:DFwt-corrected} we have
    \[\frac{d}{dr}\DF^{[t]}(\X^0,\F^0,\L^0_r))\bigg|_{r_0} = n(n+1)C_0\cdot\left((\L^0_{r_0})^n\cdot E+K^{[t]}_{\X^0/\PP^1}\cdot(\L^0_{r_0})^{n-1}\cdot E  \right)\]
    \[= C_1\cdot(\L^0_{r_0})^{n-1}\cdot E\cdot ((\L^0_{r_0})^{n-1}\cdot E) = C_1\cdot  (\L^0_{r_0})^{n-1}\cdot E^2\leq 0  \] 
    where, $C_0$ and $C_1$ are constants, and the last inequality follows by \cite[Lemma 1]{LX14}, with equality if and only if $\X^0 \simeq \X$, as required. Thus we can set $\L_0 = \L_{0,r}$ for some sufficiently small rational $r>0$ to satisfy the conditions of the Proposition.
\end{proof}

\begin{theorem}\label{thm: xlc step}
    Let $(\X,\F_{\X},\L)\to C$ be a polarised family of adjoint Fano foliated structures over a proper smooth curve, with $C^\circ\subset C$ parametrizing the non-degenerate fibers. Then we can construct  a finite morphism $\phi:C'\to C$ and a polarised generic family of adjoint Fano foliated structures $(\X^{{0}},\F^0,\L^{{0}})\to {C}'$
 with the following properties:
\begin{enumerate}
\item There is a birational morphism $\X^{{0}}\rightarrow \X\times_CC'$,
 which is isomorphic over $\phi^{-1}(C^\circ)$.
\item For every $t\in C'$,  $(\X^{{0}},\X^{{0}}_t)$ is log canonical.
\item There is an inequality
$$\DF^{[t]}((\X^{{0}}, \F^0,\L^{{0}})/C')\le \deg (\phi) \cdot \DF^{[t]}((\X,\F,\L)/C).$$
Moreover, the equality holds for our construction if and only if $(\X,\X_t)$ is log canonical for every $t\in C$, in which case $\X^{{\rm 0}}$ is isomorphic to $\X\times_{C}C'$.
\end{enumerate}
\end{theorem}
\begin{proof}
    This is a direct analogue of \cite[Theorem 2]{LX14}. For (1) and (2), we only have to verify that after performing the base change we can replicate the procedure of Corollary \ref{cor:prepared-model-klt-regime}. By taking the base change $\X\times_{C}C'$ and the normalisation $\widetilde{\X}$, we can apply Proposition \ref{prop:prepared-qdlt-model} and Lemma \ref{lem:prepared-model-klt} to obtain (1) and (2). We have a natural finite morphism $\phi_{\X}\colon \widetilde{\X}\rightarrow \X$ with $\widetilde{\L}:=\phi^*\L_{\X}$ and induced foliation $\F_{\widetilde{\X}}$. Then we have
    \[K_{\widetilde{X}/C'}^{[t]} = \phi^*_{\X}\left(K^{[t]}_{\X/C} + (\mathrm{red}(\X_0)-\X_0) \right)\]
    which in turn gives 
    $$\deg(C'/C)\cdot\DF((\X,\F_{\X},\L)/C)\ge
\DF((\widetilde{\X},\F_{\widetilde{\X}},\widetilde{\L})/C').$$
Combining with Proposition \ref{prop: df decreases in first step} which gives 
$$\DF^{[t]}((\X^0, \F^0,\L^0)/C)\leq \DF^{[t]}((\widetilde{\X}, \F_{\widetilde{\X}},\widetilde{\mathcal \L})/C)$$
we obtain the first claim of (3). For the isomorphism claim, note that we must require the above inequalities to be equalities, in which case $(\widetilde{\X}, \widetilde{\X}_0)$ must be log canonical. Hence $(\X,\X_0)$ must also be log canonical.
\end{proof}

We will now study how the mixed Donaldson--Futaki changes along the divisorial contractions and flips we perform in section \ref{sec: mmp with scaling} to construct $\X^{ac}$. Recall that as in Remark \ref{rem: equality of df for anti adjoint model}, we denote, for $\lambda >1$, 
    \[\L_{\lambda}:= \frac{1}{\lambda-1}(K_{\X^0,C}^{[t]}+\lambda H^0)\]
    and $\L^i$ denotes the pushforward of $\L^0$ to $\X^i$. Note, that if $\X_0^0 = \sum_j E_j$ is the irreducible decomposition of the central fibre of $\X^0$ we can find $a_j(\lambda)\in \R$ such that 
    \[K_{\X^0,C}^{[t]}+\L^0_{\lambda} \sim_{\Q,C}\sum_j a_j(\lambda)E_j.\]
    These notations give us the following Lemma. 

\begin{lemma}\label{lem: decrease on same model}
    If $\lambda_i\ge a > b\ge \lambda_{i+1}$ and $b>1$, then
$\DF^{[t]}(\X^{i}, \F^i,\L_a^{i})\ge \DF^{[t]}(\X^{i}, \F^i,\L_b^{i})$
\end{lemma}
\begin{proof}
    Note that as in Proposition \ref{prop: df decreases in first step} we have 
    \begin{equation*}
        \begin{split}
            \frac{d}{d\lambda}\DF^{[t]}(\X^i,\F^i,\L^i_\lambda)) &= C_0\cdot\left((\L^i_{\lambda})^{n-1}\cdot (\L^i_{\lambda})'\cdot (\L^i_{\lambda}+K_{\X^i/C}^{[t]}) \right)\\
            &= -C_1 (\L^i_{\lambda})^{n-1}\cdot (\L^i_{\lambda}+K_{\X^i/C}^{[t]})^2\\
            &= -C_1 (\L^i_{\lambda})^{n-1}\cdot (\sum_j a_j(\lambda)E_j)^2.
        \end{split}
    \end{equation*}
    Combined with \cite[Lemma 1]{LX14} we obtain the desired result.
\end{proof}

We now have two Lemmas studying the mixed Donaldson--Futaki invariant along divisorial contractions and flips.

\begin{lemma}
\label{lem:62A-monotonicity-divisorial-contractions}
Suppose $\varphi_i:\X^i\to \X^{i+1}$ is a $\G_m$-equivariant divisorial contraction in the relative
$K^{[t]}$-MMP with $\L^i\sim_{\Q,\PP^1}-K^{[t]}_{\X^i/\PP^1}$ and $\L^{i+1}\sim_{\Q,\PP^1}-K^{[t]}_{\X^{i+1}/\PP^1}$. Then
\[
\DF^{[t]}(\X^i,\F^i,\L^i)\ge
\DF^{[t]}(\X^{i+1},\F^{i+1},\L^{i+1}).
\]
\end{lemma}

\begin{proof}
We have $K^{[t]}_{\X^i/\PP^1}
=
\varphi_i^*K^{[t]}_{\X^{i+1}/\PP^1}+aE$ for some $\varphi_i$-exceptional prime divisor $E$ and some $a>0$.
Equivalently, $\L^i=\varphi_i^*\L^{i+1}-aE$.

Since the birational step is over $\PP^1$ and does not change the generic fibre, the volume $V=L^n$ remains constant throughout the construction. Moreover, $\L^{i+1}$ is relatively nef and $\varphi_i$ is $K^{[t]}$-negative. Hence the standard intersection-theoretic argument from \cite[Section~4.2.1]{LX14} gives $(\L^i)^{n+1}\le (\L^{i+1})^{\,n+1}$. Therefore, by Lemma~\ref{lem:anti-adjoint-mixed-DF},
\[
\DF^{[t]}(\X^i,\F^i,\L^i)
=
-\frac{1}{(n+1)V}(\L^i)^{n+1}
\ge
-\frac{1}{(n+1)V}(\L^{i+1})^{\,n+1}
=
\DF^{[t]}(\X^{i+1},\F^{i+1},\L^{i+1}),
\]
as claimed.
\end{proof}

\begin{lemma}
\label{lem:62B-monotonicity-flips}
Suppose $\X^i\dashrightarrow \X^{i+1}$ is a $\G_m$-equivariant flip in the relative $K^{[t]}$-MMP, with $\L^i\sim_{\Q,\PP^1}-K^{[t]}_{\X^i/\PP^1}$ and $\L^{i+1}\sim_{\Q,\PP^1}-K^{[t]}_{\X^{i+1}/\PP^1}$. Then
\[
\DF^{[t]}(\X^i,\F^i,\L^i)\ge
\DF^{[t]}(\X^{i+1},\F^{i+1},\L^{i+1}).
\]
\end{lemma}

\begin{proof}
Suppose $\X^i \dashrightarrow \X^{i+1}$ is a $\G_m$-equivariant flip in the relative $K^{[t]}$-MMP, again in the
anti-adjoint situation. Let
\[
\xymatrix{
& W \ar[dl]_p \ar[dr]^q & \\
\X^i \ar@{-->}[rr] && \X^{i+1}
}
\]
be the normalisation of the graph.

Let $D:=p^*K^{[t]}_{\X^i/\PP^1}-q^*K^{[t]}_{\X^{i+1}/\PP^1}$. By the $K^{[t]}$-negativity of the flip, the divisor $D$ is effective and exceptional over $\X^{i+1}$. Equivalently, $q^*\L^{i+1}=p^*\L^i+D$ with $D\ge 0$.

Since the birational step is over $\PP^1$ and does not change the generic fibre, the volume $V=L^n$ remains constant throughout the construction. By monotonicity of volume under addition of effective divisors, $\vol(q^*\L^{i+1})\ge \vol(p^*\L^i)$. Using birational invariance of volume, we obtain $\vol(\L^{i+1})\ge \vol(\L^i)$. Since both divisors are relatively nef in the anti-adjoint model, this gives us $(\L^{i+1})^{\,n+1}\ge (\L^i)^{n+1}$. Therefore, by Lemma~\ref{lem:anti-adjoint-mixed-DF},
\[
\DF^{[t]}(\X^i,\F^i,\L^i)
=
-\frac{1}{(n+1)V}(\L^i)^{n+1}
\ge
-\frac{1}{(n+1)V}(\L^{i+1})^{\,n+1}
=
\DF^{[t]}(\X^{i+1},\F^{i+1},\L^{i+1}),
\]
as claimed.
\end{proof}

\begin{remark}\label{rem: equality of df for anti adjoint model}
    Note that Proposition \ref{prop:threshold-stage-package} and the construction of $\X^{ac}$ show that 
    \[\DF^{[t]}(\X^j,\F^j, \L^j_{\lambda_j}) = \DF^{[t]}(\X^j,\F^j, -K_{A^j}) = \DF^{[t]}(\X^{ac},\F^{ac}, -K_{X^{ac}/C}^{[t]}). \] 
\end{remark}

\begin{proposition}\label{prop: df decreases along MMP}
    We have 
    \[\DF^{[t]}(\X^0,\F^0, \L^0) \geq \DF^{[t]}(\X^{ac},\F^{ac}, -K_{X^{ac}/C}^{[t]}). \]
\end{proposition}
\begin{proof}
    The proposition now is an immediate consequence of Lemmas \ref{lem: decrease on same model}, \ref{lem:62A-monotonicity-divisorial-contractions}, \ref{lem:62B-monotonicity-flips} and Remark \ref{rem: equality of df for anti adjoint model} which show successively:
    \begin{eqnarray*}
& &\DF^{[t]}((\X^{0}, \F^{0}, \L^{0}_{\lambda_0})/C)\ge \DF^{[t]}((\X^{0}, \F^{0},\L^{0}_{\lambda_1})/C)\\
&\ge&\DF^{[t]}((\X^{1}, \F^{1}, \L^{1}_{\lambda_1})/C)\ge \DF^{[t]}((\X^{1}, \F^{1},
\L^{1}_{\lambda_2})/C)\\
&&\cdots\quad  \cdots \quad \cdots\\
&\ge&\DF^{[t]}((\X^{i}, \F^{i}, \L^{i}_{\lambda_i})/C)\ge \DF^{[t]}((\X^{i}, \F^{i},
\L^{i}_{\lambda_{i+1}})/C)\\
&\ge&\DF^{[t]}((\X^{i+1}, \F^{i+1}, \L^{i+1}_{\lambda_{i+1}})/C)\ge \DF^{[t]}((\X^{i+1}, \F^{i+1},
\L^{i+1}_{\lambda_{i+2}})/C)\\
&&\cdots\quad \cdots \quad \cdots\\
&\ge&\DF^{[t]}((\X^{j}, \F^{j}, \L^{k}_{\lambda_j})/C)=\DF^{[t]}((\X^{j}, \F^{j},-K_{\X^j/C}^{[t]})/C)\\
&=& \DF^{[t]}(\X^{ac},\F^{ac}, -K_{X^{ac}/C}^{[t]}).
\end{eqnarray*}
\end{proof}

We now study how the mixed Donaldson--Futaki invariant changes along the ajoint Fano foliation structure extension from $\X^{ac}$ to $\X^s$. 

\begin{proposition}\label{prop: decrease of df in prev to last step}
    We have the inequality
$$\DF^{[t]}((\X',\F',-K^{[t]}_{\X'/C'})/C')\ge \DF^{[t]}((\X^{{\rm s}}, \F^{\rm s},-K^{[t]}_{\X^{{\rm s}}/C')/C'}).$$
\end{proposition}
\begin{proof}
    By Lemma~\ref{lem:anti-adjoint-mixed-DF}, we have
\[
\DF^{[t]}((\X',\F',\L')/C)
=
-\frac{1}{(n+1)V}(\L')^{n+1}
\]
and
\[
\DF^{[t]}(\X^s, \F^s,\L^s) = -\frac{1}{(n+1)V}(\L^{s})^{\,n+1}.
\]
We now take a common log resolution 
\[
\xymatrix{
& \hat{\X} \ar[dl]_p \ar[dr]^q & \\
\X' \ar@{-->}[rr] && \X^{s}
}
\]
where 
\[
(\pi'\circ p)^*(K_{\X^{ac}/C'}^{[t]}) = p^*K_{\X'/C'}^{[t]} = q^*K_{\X^{s}/C'}^{[t]}+E
\]
where $E\geq 0$ by the Negativity Lemma. For $0\leq \lambda\leq 1$ we define 
\[f(\lambda):= (-p^*K_{\X'/C'}^{[t]}+\lambda E)^{n+1}.\]
As in the proof of Proposition \ref{prop: df decreases in first step} we take the derivative; for any $0\leq \lambda\leq 1$ we have
\begin{eqnarray}
\frac{df(\lambda)}{d\lambda} &=& (n+1)E \cdot (-p^*K_{\X'/C'}^{[t]}+\lambda E)\nonumber\\
 &=&(n+1)E \cdot (-(1-\lambda)p^*K_{\X'/C'}^{[t]}-\lambda q^*K_{\X^{{\rm s/C}^{[t]}}})^{n} \nonumber\\
&\ge& 0 \nonumber.
\end{eqnarray}
Combined with the above expressions, this gives the required formula.
\end{proof}

Now we are in a position to prove the main theorem of this section.

\begin{theorem}[Reduction to $\F$-compatible special test configurations]
\label{thm:reduction-special-mixed-DF}
Let $(X,\F)$ be an adjoint Fano foliated structure, fix $t\in [0,1)\cap \Q$ and set $L:=-K^{[t]}_{X,\F}$. Let $(\pi\colon \X\to \A^1,\L,\F_{\X})$ be a normal $\F$-compatible test configuration for $(X,\F,L)$.
Then there exist:
\begin{enumerate}
    \item a base change
    \[
    d\colon \A^1\to \A^1,\qquad z\mapsto z^m,
    \]
    for some $m\in \Z_{>0}$; and
    \item a normal $\F$-compatible special test configuration
    \[
    (\pi^{\mathrm{sp}}\colon \X^{\mathrm{sp}}\to \A^1,\L^{\mathrm{sp}},\F_{\X^{\mathrm{sp}}})
    \]
    for $(X,\F,L)$,
\end{enumerate}
such that
\[
m\,\DF^{[t]}(\X,\F_{\X},\L_{\X})\ge
\DF^{[t]}(\X^{\mathrm{sp}},\F_{\X^{\mathrm{sp}}}, \L^{\mathrm{sp}}).
\]
\end{theorem}

\begin{proof}
By Theorem~\ref{thm: construction of special t.c.}, after a finite base change
\[
d\colon \A^1\to \A^1,\qquad z\mapsto z^m,
\]
the pullback of $(\X,\ \F_{\X},\ \L)$ admits a normal $\F$-compatible special test configuration $(\X^{\mathrm{sp}},\F_{\X^{\mathrm{sp}},\L^{\mathrm{sp}}})$ obtained through the birational procedure of Section \ref{sec: special t.c.}.
For the mixed Donaldson--Futaki invariants we then have

\begin{eqnarray*}\label{decseq2}
 \deg(C'/C)\cdot\DF^{[t}((\X, \F, \L)/C)
&\ge&\DF^{[t}((\X^{0},\F^0, \L^{{0}})/C') \hspace{5mm} \mbox{(by Theorem \ref{thm: xlc step})} \\
&\ge&\DF^{[t}((\X^{{\rm ac}}, \F^{\rm ac}, -K^{[t]}_{\X^{{\rm ac}/C}})/C') \hspace{5mm} \mbox{(by Proposition \ref{prop: df decreases along MMP})} \\
&\ge&\frac{1}{\deg(C''/C')}\DF^{[t}((\X^{{\rm sp}}, \F^{\rm sp},-K^{[t]}_{\X^{{\rm s}/C''}})/C'') \hspace{5mm}
\mbox{(by Proposition \ref{prop: decrease of df in prev to last step})} 
\end{eqnarray*}
which gives the complete statement.
\end{proof}

The following corollary, which determines $t$-K-semistability in terms of $\F$-compatible special test configurations, is an immediate consequence of Theorem~\ref{thm:reduction-special-mixed-DF}.

\begin{corollary}
\label{cor:reduction-special-mixed-DF}
Let $(X,\F)$ be an adjoint Fano foliated structure and set $L:=-K^{[t]}_{X,\F}$. Then $(X,\F,L)$ is $t$-K-semistable if and only if
\[
\DF^{[t]}(\mathcal Y,\F_{\mathcal Y},\mathcal M)\ge 0
\]
for every normal $\F$-compatible special test configuration of $(X,\F,L)$.
\end{corollary}

\section{\texorpdfstring{$t$-Ding stability and $t$-K-semistability}{t-Ding stability and t-K-semistability}}\label{sec: ding stability}
In this section we will introduce the notion of $t$-Ding stability which we will show is equivalent to $t$-K-stability. Throughout the section let $(X,\F,t)$ be an adjoint Fano foliated structure, and set $L \sim_{\Q} -K^{[t]}_{X,\F}$.

\subsection{\texorpdfstring{The $t$-Ding invariant and $t$-Ding stability}{The t-Ding invariant and t-Ding stability}}\label{sec: def ding inv}

Let $(\pi\colon \X \to \A^1,\ \F_{\X},\ \L)$ be a normal ample $\F$-compatible test configuration for $(X,\F,L)$, and let $(\bar{\X},\bar{\F},\bar{\L}) \to \PP^1$ be its natural compactification.

\begin{definition}\label{def:mixed-correction-divisor}
The \emph{mixed correction divisor} $D^{[t]}_{(\X,\F_{\X},\L)}$ is the unique vertical $\Q$-divisor on $\bar{\X}$ supported on
$\bar{\X}_0$ such that
\[
D^{[t]}_{(\X,\F_{\X},\L)}
\sim_{\Q}
-K^{[t]}_{\bar{\X}/\PP^1,\bar{\F}}-\bar{\L}.
\]
\end{definition}

Let $B\geq 0$ be an effective vertical $\Q$-divisor on $\bar{\X}$.
Recall that as in Definition \ref{def:mixed-lct} the mixed log canonical threshold of $B$ along $\bar{\X}_0$ is, using Lemma \ref{lem:mixed-discrepancy-subtract-D}, 
\[
\lct^{[t]}(\bar{\X},\bar{\F};B;\bar{\X}_0)
:=
\inf_{w}
\frac{A^{[t]}_{\bar{\X},\bar{\F}}(w)-w(B)}
{w(\bar{\X}_0)},
\]
where the infimum is taken over all divisorial valuations $w$ over
$\bar{\X}$ with $w(\bar{\X}_0)>0$.
Equivalently,
\[
\lct^{[t]}(\bar{\X},\bar{\F};B;\bar{\X}_0)
=
\sup\Bigl\{
c\in \R_{\ge 0}\ \Big|\ 
A^{[t]}_{\bar{\X},\bar{\F}}(w)-w(B)-c\,w(\bar{\X}_0)\ge 0
\text{ for all divisorial }w
\Bigr\}.
\]

We are now in a position to define the mixed Ding invariant.

\begin{definition}\label{def:t-Ding}
Let $(\pi\colon \X \to \A^1,\ \F_{\X},\ \L)$ be a normal ample $\F$-compatible test configuration for $(X,\F,L)$, and set $V:=L^n$. The $t$-\emph{Ding invariant} is
\[
\Ding^{[t]}(\X,\F_{\X},\L)
:=
-\frac{\bar{\L}^{\,n+1}}{(n+1)V}
-(1-t)
+
\lct^{[t]}
\!\left(
\bar{\X},\bar{\F};
D^{[t]}_{(\X,\F_{\X},\L)};
\bar{\X}_0
\right).
\]
\end{definition}

\begin{definition}\label{def:t-Ding-stability}
Let $(X,\F,L)$ be a polarised foliated variety.
\begin{enumerate}
\item We say that $(X,\F,L)$ is \emph{$t$-Ding semistable} if
\[
\Ding^{[t]}(\X,\F_{\X},\L)\ge 0
\]
for every normal ample $\F$-compatible test configuration
$(\X,\F_{\X},\L)$.
\item We say that $(X,\F,L)$ is \emph{uniformly $t$-Ding stable} if there exists
$\delta\in (0,1)$ such that
\[
\Ding^{[t]}(\X,\F_{\X},\L)\ge \delta\,J^{\NA}(\X,\L)
\]
for every normal ample $\F$-compatible test configuration
$(\X,\F_{\X},\L)$.
\end{enumerate}
\end{definition}

Using this definition and Definition \ref{def:t-weakly-special} we obtain the following.

\begin{proposition}\label{prop:mixed-BF}
Let $(\X,\F_{\X},\L)$ be a normal ample $\F$-compatible test configuration for $(X,\F,L)$. Then
\[
\DF^{[t]}(\X,\F_{\X},\L)
\ge
\Ding^{[t]}(\X,\F_{\X},\L).
\]
Moreover, equality holds whenever
$(\X,\F_{\X},\L)$ is $t$-weakly special.
\end{proposition}

\begin{proof}
The proof is essentially identical to \cite[Proof of Proposition 2.5.(4)]{Fujita16} (itself using \cite{Berman16}) by replacing $K_{\bar{\X}/\PP^1}+\Delta_{\bar{\X}}$ by $K^{[t]}_{\bar{\X}/\PP^1,\bar{\F}}$ and replacing the usual log canonical threshold by the mixed threshold
of Definition~\ref{def:mixed-lct}. We thus omit the full proof. The equality statement follows because for a
$t$-weakly special test configuration we have $D^{[t]}_{(\X,\F_{\X},\L)}=0$ and
\[
\lct^{[t]}(\bar{\X},\bar{\F};0;\bar{\X}_0)=1-t,
\]
so the $t$-Ding and mixed Donaldson--Futaki invariants agree.
\end{proof}

The next proposition studies the mixed Donaldson--Futaki and Ding invariants in projective birational morphisms.

\begin{proposition} \label{prop:birational-invariance}
Let $\gamma:(\Y,\F_{\Y},\M) \rightarrow (\X,\F_{\X},\L)$
be a $\G_m$-equivariant projective birational morphism of normal
$\F$-compatible test configurations of $(X,\F,L)$ over $\A^1$. Assume
that $\gamma$ is an isomorphism over $\A^1\setminus\{0\}$, that
$\F_{\Y}$ is the induced foliation on $\Y$, and
that $\M=\gamma^*\L$. Then
\[
\DF^{[t]}(\Y,\F_{\Y},\M)
=
\DF^{[t]}(\X,\F_{\X},\L).
\]
and
\[
\Ding^{[t]}(\Y,\F_{\Y},\M)
=
\Ding^{[t]}(\X,\F_{\X},\L).
\]
\end{proposition}
\begin{proof}
Let $(\bar{\Y},\bar{\F}_{\Y},
\bar{\M}) \rightarrow(\bar{\X},\bar{\F}_{\X}, \bar{\L})$ denote the induced morphism between the $\infty$-trivial compactifications over $\PP^1$. By abuse of notation, we will continue to denote it by $\gamma$. Since $\bar{\M}=\gamma^*\bar{\L}$, the projection formula gives
\[
\bar{\M}^{\,n+1}
=
(\gamma^*\bar{\L})^{n+1}
=
\bar{\L}^{\,n+1}.
\]

It remains to compare the mixed canonical term. Since $\F_{\Y}$ is the induced foliation, we have $K_{\bar{\Y}/\PP^1} = \gamma^*K_{\bar{\X}/\PP^1}+E_X$ and $K_{\bar{\F}_{\Y}} = \gamma^*K_{\bar{\F}_{\X}} + E_{\F}$, where $E_X$ and $E_{\F}$ are $\gamma$-exceptional Weil divisors. Hence
\[
K^{[t]}_{\bar{\Y},
\bar{\F}_{\Y}/\PP^1}
=
\gamma^*
K^{[t]}_{\bar{\X},
\bar{\F}_{\X}/\PP^1}
+
\bigl((1-t)E_X+tE_{\F}\bigr).
\]
Intersecting with $\bar{\M}^n
=(\gamma^*\bar{\L})^n$ gives
\[
K^{[t]}_{\bar{\Y},
\bar{\F}_{\Y}/\PP^1}
\cdot
\bar{\M}^{\,n}
=
\gamma^*
K^{[t]}_{\bar{\X},
\bar{\F}_{\X}/\PP^1}
\cdot
(\gamma^*\bar{\L})^n
+
\bigl((1-t)E_X+tE_{\F}\bigr)
\cdot
(\gamma^*\bar{\L})^n.
\]
The first term equals $K^{[t]}_{\bar{\X},
\bar{\F}_{\X}/\PP^1}
\cdot
\bar{\L}^{\,n}$ by the projection formula. The second term is zero, again by the projection formula, since $(1-t)E_X+tE_{\F}$ is $\gamma$-exceptional. Therefore
\[
K^{[t]}_{\bar{\Y},
\bar{\F}_{\Y}/\PP^1}
\cdot
\bar{\M}^{\,n}
=
K^{[t]}_{\bar{\X},
\bar{\F}_{\X}/\PP^1}
\cdot
\bar{\L}^{\,n}.
\]

The slope term is unchanged, because $\gamma$ is an isomorphism over
$\PP^1\setminus\{0\}$ and both compactified test configurations have the
same general fibre $(X,\F,L)$. Combining the two equalities we obtain
\[
\DF^{[t]}(\Y,\F_{\Y},\M)
=
\DF^{[t]}(\X,\F_{\X},\L).
\]
The same computation above shows that 
\[\lct^{[t]}(\bar{\Y},\bar{\F}_{\bar{\Y}}; D^{[t]}_{\Y,\F_\Y,\M};\Y_0) =  \lct^{[t]}(\bar{\X},\bar{\F}_{\bar{\X}}; D^{[t]}_{\X,\F_\X,\L};\X_0)\]
and hence that 
\[
\Ding^{[t]}(\Y,\F_{\Y},\M)
=
\Ding^{[t]}(\X,\F_{\X},\L).
\]
\end{proof}

\begin{remark}\label{rem:ample-semiample-interchange}
In the definitions of $t$-K-stability and $t$-Ding stability we can equivalently test against normal ample $\F$-compatible test configurations or against normal semiample $\F$-compatible test configurations. To see this, let $\gamma:(\X,\F_{\X},\L) \rightarrow (\X^{\mathrm{amp}}, \F_{\X^{\mathrm{amp}}},  \L^{\mathrm{amp}})$ be the ample model of a normal semiample $\F$-compatible test configuration, so that $\L=\gamma^*\L^{\mathrm{amp}}$ and $\F_{\X}$ is the induced foliation. By Proposition~\ref{prop:birational-invariance}, we have
\[
\DF^{[t]}(\X,\F_{\X},\L)
=
\DF^{[t]}(\X^{\mathrm{amp}},
          \F_{\X^{\mathrm{amp}}},
          \L^{\mathrm{amp}}).
\]
and
\[
\Ding^{[t]}(\X,\F_{\X},\L)
=
\Ding^{[t]}(\X^{\mathrm{amp}},
            \F_{\X^{\mathrm{amp}}},
            \L^{\mathrm{amp}}).
\]
Thus allowing semiample test configurations does not change the
resulting notions of $t$-K-semistability or $t$-Ding semistability.
\end{remark}

\subsection{Reduction of the Ding invariant to special test configurations}

In this section we repeat the same type of argument as in Section \ref{sec: reduction to special t.c.} to show that the twisted Ding invariant decreases along each step of the birational procedure detailed in Section \ref{sec: special t.c.}. We fix an adjoint Fano foliated structure $(X,\F,t,L)$.

\begin{theorem}\label{thm:t-Ding-lc}
Let $(\pi\colon \X\to \A^1,\ \F_{\X},\ \L)$ be a normal ample $\F$-compatible test configuration for $(X,\F,L)$.
Then, after a base change $d\colon \A^1\to \A^1,$ $z\mapsto z^m$ there exist a projective birational $\G_m$-equivariant morphism $\pi^{0}\colon \X^{0}\to \X^{(d)}$ and a normal ample $\F$-compatible test configuration $(\X^{0},\F^{0},\L^{0})$ for $(X,\F,L)$ such that:
\begin{enumerate}
\item $(\X^{0},\F^{0};\X^{0}_0)$ is mixed log canonical;
\item for every $\delta\in [0,1]$,
\[
m\Bigl(
\Ding^{[t]}(\X,\F_{\X},\L)
-\delta\,J^{\NA}(\X,\L)
\Bigr)
\ge
\Ding^{[t]}(\X^{0},\F^{0},\L^{0})
-\delta\,J^{\NA}(\X^{0},\L^{0}).
\]
\end{enumerate}
\end{theorem}

\begin{proof}
The underlying geometric construction is exactly the one already carried out in Theorem \ref{thm: xlc step}. To prove the stated inequality, we will follow \cite[Proof of Theorem 3.1]{Fujita16}.

Let $(\bar{\X}^{0},\bar{\F}^{0},\bar{\L}_0^0)
\to \PP^1$ be the natural compactification. We take the mixed correction divisor $E:=D^{[t]}_{(\X^{0},\F^{0},\L_0^0)}$ with
\[
E\sim_{\Q}
-K^{[t]}_{\bar{\X}^{0}/\PP^1,\bar{\F}^{0}}
-\bar{\L}_0^0.
\]

We write $E=\sum_{i=1}^p e_iE_i$  where $e_1\le \cdots \le e_p$. For $0<\lambda\ll 1$, we also let $\L_{\lambda}^0:=\L_0^0+\lambda E$. Notice that for $\lambda>0$ sufficiently small the divisor
$\L_{\lambda}^0$ is still relatively ample over $\A^1$, so $(\X^{0},\F^{0},\L_{\lambda}^0)$ is again a normal ample $\F$-compatible test configuration.

Since each $E_i$ is vertical, Lemma \ref{lem:vertical-divisors-invariant} implies that every $E_i$ is $\F^{0}$-invariant. Hence, $A^{[t]}_{\X^{0},\F^{0}}(E_i)=1-t$. In particular, using Lemma \ref{lem:mixed-discrepancy-subtract-D}, the mixed log canonical threshold along the central fibre satisfies 
\[
\lct^{[t]}(\X^0,\F^0;E;\bar{\X}_0)=1-t -(1-\lambda)e_1.
\]

Thus, the mixed Ding invariant is
\[
\Ding^{[t]}(\X^{0},\F^{0},\L_{\lambda}^0) = -\frac{(\bar{\L^0})^{\,n+1}}{(n+1)V}+(1+\lambda)e_1.
\]
Hence, by differentiating this expression with respect to $\lambda$, we obtain
\[
\frac{d}{d\lambda}
\Bigl(
\Ding^{[t]}(\X^{0},\F^{0},\L_{\lambda}^0)
-\delta\,J^{\NA}(\X^{0},\L_{\lambda}^0)
\Bigr)
\le 0
\]
for every $\delta\in[0,1]$.
Hence
\[
\Ding^{[t]}(\X^{0},\F^{0},\L_{\lambda}^0)
-\delta\,J^{\NA}(\X^{0},\L_{\lambda}^0)
\le
\Ding^{[t]}(\X^{0},\F^{0},\L_0^0)
-\delta\,J^{\NA}(\X^{0},\L_0^0).
\]

Now, recall that after base change $J^{\NA}(\X^{0},\L_0^0)
=
m\,J^{\NA}(\X,\L)$. Likewise, by the same base-change computation and the definition of the mixed
Ding invariant,
\[
\Ding^{[t]}(\X^{0},\F^{0},\L_0^0)
=
m\,\Ding^{[t]}(\X,\F_{\X},\L).
\]
Combining this with the previous inequality fyields
\[
\Ding^{[t]}(\X^{0},\F^{0},\L_{\lambda}^0)
-\delta\,J^{\NA}(\X^{0},\L_{\lambda}^0)
\le
m\Bigl(
\Ding^{[t]}(\X,\F_{\X},\L)
-\delta\,J^{\NA}(\X,\L)
\Bigr).
\]
which implies
\[
m\Bigl(
\Ding^{[t]}(\X,\F_{\X},\L)
-\delta\,J^{\NA}(\X,\L)
\Bigr)
\ge
\Ding^{[t]}(\X^{0},\F^{0},\L^{0})
-\delta\,J^{\NA}(\X^{0},\L^{0}),
\]
as claimed.
\end{proof}

Now let $(\X^{0},\F^{0},\L^{0})$ be as before, and let $(\X^{\mathrm{ac}},\F^{\mathrm{ac}},\L^{\mathrm{ac}})$ be the anti-adjoint model produced by the birational procedure of Section \ref{sec: special t.c.}.

\begin{theorem}\label{thm:t-Ding-ac}
For every $\delta\in [0,1]$,
\[
\Ding^{[t]}(\X^{0},\F^{0},\L^{0})
-\delta\,J^{\NA}(\X^{0},\L^{0})
\ge
\Ding^{[t]}(\X^{\mathrm{ac}},\F^{\mathrm{ac}},\L^{\mathrm{ac}})
-\delta\,J^{\NA}(\X^{\mathrm{ac}},\L^{\mathrm{ac}}).
\]
\end{theorem}

\begin{proof}
By the construction in Section \ref{sec: reduction to special t.c.}, starting from the lc model $(\X^{0},\F^{0},\L^{0})$ one obtains a finite sequence of birational models
\[
(\X^0,\F^0,\L^0_{\lambda_0}),
\ 
(\X^0,\F^0,\L^0_{\lambda_1}),
\ 
(\X^1,\F^1,\L^1_{\lambda_1}),
\ 
(\X^1,\F^1,\L^1_{\lambda_2}),
\ \ldots,\
(\X^j,\F^j,\L^j_{\lambda_j}),
\]
where $\lambda_0>\lambda_1>\cdots>\lambda_j$ and where the final model satisfies $\L^j_{\lambda_j}\sim_{\Q,\A^1}
-K^{[t]}_{\X^j/\A^1,\F^j}$. By definition, the anti-adjoint model is
\[
\X^{\mathrm{ac}}
:=
\operatorname{Proj}_{\A^1}
R(\X^j/\A^1,-K^{[t]}_{\X^j/\A^1,\F^j}),
\]
and
\[
\DF^{[t]}(\X^j,\F^j,\L^j_{\lambda_j})
=
\DF^{[t]}(\X^{\mathrm{ac}},\F^{\mathrm{ac}},\L^{\mathrm{ac}})
\]
by Remark \ref{rem: equality of df for anti adjoint model} and Proposition \ref{prop: df decreases along MMP}.

We claim that the same chain of inequalities holds with $\Ding^{[t]}-\delta J^{\NA}$ in place of $\DF^{[t]}$, namely
\begin{align*}
&\Ding^{[t]}(\X^0,\F^0,\L^0_{\lambda_0})
-\delta J^{\NA}(\X^0,\L^0_{\lambda_0})
\\
&\ge
\Ding^{[t]}(\X^0,\F^0,\L^0_{\lambda_1})
-\delta J^{\NA}(\X^0,\L^0_{\lambda_1})
\\
&\ge
\Ding^{[t]}(\X^1,\F^1,\L^1_{\lambda_1})
-\delta J^{\NA}(\X^1,\L^1_{\lambda_1})
\\
&\ge
\Ding^{[t]}(\X^1,\F^1,\L^1_{\lambda_2})
-\delta J^{\NA}(\X^1,\L^1_{\lambda_2})
\\
&\hspace{1cm}\vdots
\\
&\ge
\Ding^{[t]}(\X^j,\F^j,\L^j_{\lambda_j})
-\delta J^{\NA}(\X^j,\L^j_{\lambda_j}).
\end{align*}
Once this is proved, the desired inequality follows immediately, since
\[
\Ding^{[t]}(\X^j,\F^j,\L^j_{\lambda_j})
=
\Ding^{[t]}(\X^{\mathrm{ac}},\F^{\mathrm{ac}},\L^{\mathrm{ac}})
\]
and similarly for $J^{\NA}$, by the same observation as in Remark \ref{rem: equality of df for anti adjoint model}. Thus, it remains to justify the above chain.

We now fix $i$. On the fixed model $\X^i$, the passage from
$\L^i_{\lambda_i}$ to $\L^i_{\lambda_{i+1}}$ is obtained by
varying the coefficient of a vertical divisor supported on the central fibre.
Exactly as in the proof of Theorem~\ref{thm:t-Ding-lc}, the derivative of
\[
\Ding^{[t]}(\X^i,\F^i,\L^i_{\lambda})
-\delta J^{\NA}(\X^i,\L^i_{\lambda})
\]
with respect to $\lambda$ is non-positive. We omit the details, which are identical to the argument in \cite[Proof of Theorem 3.2]{Fujita16}, where the only new input is the mixed log canonical threshold term. Since every irreducible component of the central fibre is vertical, it is $\F^i$-invariant by Lemma \ref{lem:vertical-divisors-invariant}, and therefore has mixed discrepancy $1-t$. Arguing as in the proof of Theorem \ref{thm:t-Ding-lc}, we see that the mixed Ding invariant is formally identical to the one in \cite[Proof of Theorem 3.2]{Fujita16}, with the normalising constant $1$ replaced by $1-t$, which cancels in the definition of $\Ding^{[t]}$. Therefore
\begin{equation}\label{eq: birational Ding step 1}
    \Ding^{[t]}(\X^i,\F^i,\L^i_{\lambda_i})
-\delta J^{\NA}(\X^i,\L^i_{\lambda_i})
\ge
\Ding^{[t]}(\X^i,\F^i,\L^i_{\lambda_{i+1}})
-\delta J^{\NA}(\X^i,\L^i_{\lambda_{i+1}}).
\end{equation}

Now consider the birational contraction $\X^i \dashrightarrow \X^{i+1}$ appearing in the $K^{[t]}$-MMP with scaling. On a common resolution $\widehat{\X}
\overset{p}{\longrightarrow} \X^i$, $\widehat{\X}
\overset{q}{\longrightarrow} \X^{i+1}$ we have
\[
p^*K^{[t]}_{\X^i/\A^1,\F^i}
=
q^*K^{[t]}_{\X^{i+1}/\A^1,\F^{i+1}}+E
\]
with $E\ge 0$ by the Negativity Lemma, as in the proof of
Proposition \ref{prop: df decreases along MMP}. The intersection-theoretic part of $\Ding^{[t]}-\delta J^{\NA}$ is therefore controlled exactly as in Proposition \ref{prop: df decreases along MMP}. For the threshold part, again the only input needed is that mixed discrepancy behaves under adding vertical divisors exactly as in the usual log setting by Lemma \ref{lem:mixed-discrepancy-subtract-D}, and that all vertical divisors are $\F$-invariant by Lemma \ref{lem:vertical-divisors-invariant}. Hence the same
negativity argument as in \cite[Proof of Theorem 3.2]{Fujita16} gives
\begin{equation}\label{eq: birational ding 2}
    \Ding^{[t]}(\X^i,\F^i,\L^i_{\lambda_{i+1}})
-\delta J^{\NA}(\X^i,\L^i_{\lambda_{i+1}})
\ge
\Ding^{[t]}(\X^{i+1},\F^{i+1},\L^{i+1}_{\lambda_{i+1}})
-\delta J^{\NA}(\X^{i+1},\L^{i+1}_{\lambda_{i+1}}).
\end{equation}

Iterating these steps along the entire MMP chain gives
\begin{equation}\label{eq: birational ding 3}
\begin{split}
&\Ding^{[t]}(\X^{0},\F^{0},\L^{0})
-\delta J^{\NA}(\X^{0},\L^{0})
\\
&\ge
\Ding^{[t]}(\X^j,\F^j,\L^j_{\lambda_j})
-\delta J^{\NA}(\X^j,\L^j_{\lambda_j}).
\end{split}
\end{equation}

Furthermore, by construction, $(\X^{\mathrm{ac}},\F^{\mathrm{ac}},\L^{\mathrm{ac}})$ is the anti-adjoint model of $(\X^j,\F^j,\L^j_{\lambda_j})$. Since the birational transform is crepant with respect to the mixed adjoint class, the mixed correction divisor on $\X^{\mathrm{ac}}$ vanishes:
\[
D^{[t]}_{(\X^{\mathrm{ac}},\F^{\mathrm{ac}},\L^{\mathrm{ac}})}=0.
\]
Moreover, the central fibre is vertical and hence $\F^{\mathrm{ac}}$-invariant, so its mixed discrepancy is $1-t$. Therefore
\[
\lct^{[t]}(\bar{\X}^{\mathrm{ac}},\bar{\F}^{\mathrm{ac}};0;
\bar{\X}^{\mathrm{ac}}_0)=1-t,
\]
and hence
\[
\Ding^{[t]}(\X^{\mathrm{ac}},\F^{\mathrm{ac}},\L^{\mathrm{ac}})
=
\DF^{[t]}(\X^{\mathrm{ac}},\F^{\mathrm{ac}},\L^{\mathrm{ac}}).
\]
In particular, replacing $(\X^j,\F^j,\L^j_{\lambda_j})$ by its
anti-adjoint model does not change the value of $\Ding^{[t]}-\delta J^{\NA}$.

Combining this with inequalities \eqref{eq: birational Ding step 1}, \eqref{eq: birational ding 2} and $\eqref{eq: birational ding 3}$ gives
\[
\Ding^{[t]}(\X^{0},\F^{0},\L^{0})
-\delta J^{\NA}(\X^{0},\L^{0})
\ge
\Ding^{[t]}(\X^{\mathrm{ac}},\F^{\mathrm{ac}},\L^{\mathrm{ac}})
-\delta J^{\NA}(\X^{\mathrm{ac}},\L^{\mathrm{ac}}),
\]
as required.
\end{proof}

Let as before $(\pi^{\mathrm{sp}}\colon \X^{\mathrm{sp}}\to \A^1,\ \F^{\mathrm{sp}},\ \L^{\mathrm{sp}})$ be the special test configuration constructed in Theorem \ref{thm: construction of special t.c.}.

\begin{theorem}[Reduction to $t$-weakly special test configurations for the $t$-Ding invariant]\label{thm:t-Ding-special-reduction}
For every $\delta\in [0,1]$,
\[
m\Bigl(
\Ding^{[t]}(\X,\F_{\X},\L)-\delta\,J^{\NA}(\X,\L)
\Bigr)
\ge
\Ding^{[t]}(\X^{\mathrm{sp}},\F^{\mathrm{sp}},\L^{\mathrm{sp}})
-\delta\,J^{\NA}(\X^{\mathrm{sp}},\L^{\mathrm{sp}}).
\]
Moreover,
\[
\Ding^{[t]}(\X^{\mathrm{sp}},\F^{\mathrm{sp}},\L^{\mathrm{sp}})
=
\DF^{[t]}(\X^{\mathrm{sp}},\F^{\mathrm{sp}},\L^{\mathrm{sp}}).
\]
\end{theorem}

\begin{proof}
Recall that by Theorem \ref{thm: construction of special t.c.} and the proof of Theorem \ref{thm:reduction-special-mixed-DF}, after the above base change we have a sequence of birational models
\[
(\X,\F_{\X},\L)
\rightsquigarrow
(\X^{\mathrm{lc}},\F^{\mathrm{lc}},\L^{\mathrm{lc}})
\rightsquigarrow
(\X^{\mathrm{ac}},\F^{\mathrm{ac}},\L^{\mathrm{ac}})
\rightsquigarrow
(\X^{\mathrm{sp}},\F^{\mathrm{sp}},\L^{\mathrm{sp}})
\]
where:
\begin{enumerate}
\item the first arrow is the base change plus lc/qdlt model step;
\item the second arrow is the adjoint MMP with scaling leading to the
anti-adjoint model;
\item the third arrow is the final special extraction and relative anti-adjoint model step.
\end{enumerate}

By Theorem~\ref{thm:t-Ding-lc}, for every $\delta\in[0,1]$ we have
\[
m\Bigl(
\Ding^{[t]}(\X,\F_{\X},\L)
-\delta J^{\NA}(\X,\L)
\Bigr)
\ge
\Ding^{[t]}(\X^{\mathrm{lc}},\F^{\mathrm{lc}},\L^{\mathrm{lc}})
-\delta J^{\NA}(\X^{\mathrm{lc}},\L^{\mathrm{lc}}).
\]
By Theorem~\ref{thm:t-Ding-ac}, again for every $\delta\in[0,1]$,
\[
\Ding^{[t]}(\X^{\mathrm{lc}},\F^{\mathrm{lc}},\L^{\mathrm{lc}})
-\delta J^{\NA}(\X^{\mathrm{lc}},\L^{\mathrm{lc}})
\ge
\Ding^{[t]}(\X^{\mathrm{ac}},\F^{\mathrm{ac}},\L^{\mathrm{ac}})
-\delta J^{\NA}(\X^{\mathrm{ac}},\L^{\mathrm{ac}}).
\]

It therefore remains only to treat the final step
\[
(\X^{\mathrm{ac}},\F^{\mathrm{ac}},\L^{\mathrm{ac}})
\rightsquigarrow
(\X^{\mathrm{sp}},\F^{\mathrm{sp}},\L^{\mathrm{sp}}).
\]
We argue exactly as in the proof of Theorem~\ref{thm:t-Ding-ac}: on a common resolution $\widehat{\X}
\overset{p}{\longrightarrow} \X'$, $\widehat{\X}
\overset{q}{\longrightarrow} \X^{\mathrm{sp}},$ we have
\[
p^*K^{[t]}_{\X'/\A^1,\F'}=
q^*K^{[t]}_{\X^{\mathrm{sp}}/\A^1,\F^{\mathrm{sp}}}+E
\]
with $E\ge 0$ by the Negativity Lemma. Interpolating between the two relative anti-adjoint polarisations by the same one-variable family as in Section \ref{sec: reduction to special t.c.} and using Lemmas \ref{lem:mixed-discrepancy-subtract-D} and \ref{lem:vertical-divisors-invariant}, the mixed threshold term varies in the same way as in \cite[Proof of Theorem 3.3]{Fujita16}.
Hence, omitting the full details to avoid repetition, we obtain
\[
\Ding^{[t]}(\X^{\mathrm{ac}},\F^{\mathrm{ac}},\L^{\mathrm{ac}})
-\delta J^{\NA}(\X^{\mathrm{ac}},\L^{\mathrm{ac}})
\ge
\Ding^{[t]}(\X^{\mathrm{sp}},\F^{\mathrm{sp}},\L^{\mathrm{sp}})
-\delta J^{\NA}(\X^{\mathrm{sp}},\L^{\mathrm{sp}}).
\]

Combining the three displayed inequalities gives
\begin{align*}
m\Bigl(
\Ding^{[t]}(\X,\F_{\X},\L)
-\delta J^{\NA}(\X,\L)
\Bigr)
&\ge
\Ding^{[t]}(\X^{\mathrm{sp}},\F^{\mathrm{sp}},\L^{\mathrm{sp}})
-\delta J^{\NA}(\X^{\mathrm{sp}},\L^{\mathrm{sp}}).
\end{align*}

Finally, since $(\X^{\mathrm{sp}},\F^{\mathrm{sp}},\L^{\mathrm{sp}})$ is $t$-weakly special (in particular, $t$-special), Proposition~\ref{prop:mixed-BF}
gives
\[
\Ding^{[t]}(\X^{\mathrm{sp}},\F^{\mathrm{sp}},\L^{\mathrm{sp}})
=
\DF^{[t]}(\X^{\mathrm{sp}},\F^{\mathrm{sp}},\L^{\mathrm{sp}}).
\]
This proves the theorem.
\end{proof}

We are now in a position to directly compare $t$-Ding and $t$-K-stability.

\begin{corollary}\label{cor:t-K-vs-t-Ding}
Let $(X,\F,t)$ be an adjoint Fano foliated structure and set $L\sim_{\Q} -K^{[t]}_{X,\F}$. Then the following are equivalent for every $\delta\in [0,1)$:
\begin{enumerate}
\item
\[
\DF^{[t]}(\X,\F_{\X},\L)\ge \delta\,J^{\NA}(\X,\L)
\]
for every normal ample $\F$-compatible test configuration;
\item
\[
\Ding^{[t]}(\X,\F_{\X},\L)\ge \delta\,J^{\NA}(\X,\L)
\]
for every normal ample $\F$-compatible test configuration.
\end{enumerate}
In particular, $t$-K-semistability is equivalent to $t$-Ding semistability, and
uniform $t$-K-stability is equivalent to uniform $t$-Ding stability.
\end{corollary}

\begin{proof}
The implication $(2)\Rightarrow(1)$ follows from
Proposition~\ref{prop:mixed-BF}. Conversely, assume $(1)$.
If $(2)$ failed, there would exist a normal ample $\F$-compatible test
configuration with
\[
\Ding^{[t]}-\delta J^{\NA}<0.
\]
Applying Theorem~\ref{thm:t-Ding-special-reduction}, after a base change we obtain a $t$-weakly special test configuration with no larger value of $\Ding^{[t]}-\delta J^{\NA}$. On that final model we have by Theorem \ref{thm:t-Ding-special-reduction} $0>\Ding^{[t]}=\DF^{[t]}$ contradicting $(1)$. Hence the equivalence is proven.
\end{proof}

\subsection{\texorpdfstring{$t$-Ding stability in terms of divisorial valuations}{t-Ding stability in terms of divisorial valuations}}

In this section we will show that $t$-K-stability and $t$-Ding stability can be tested for prime divisors. We first have the following auxiliary results. Let us fix some notation.

Let $(X,\F,L)$ be a polarised foliated variety, and set $W:=X\times \A^1$, $\cG:=p_1^*\F\subset T_{W/\A^1}$. Note that $T_{X\times \A^1} = p_1^*T_X\oplus p_2^*T_{\A^1}$, hence by $\cG$ we mean the natural product foliation on $X\times \A^1$. Let $\mathcal I\subset \cO_W$ be a $\G_m$-invariant coherent ideal sheaf
supported on $X\times \{0\}$, and let $\Pi\colon Y\to W$ be the normalization of the blowup of $\mathcal I$. 

\begin{lemma}\label{lem:birational-transform-product-foliation}
There exists a $\G_m$-equivariant saturated integrable subsheaf $\F_Y\subset T_{Y/\A^1}$ such that:
\begin{enumerate}
\item $\F_Y$ has the same generic rank as $\F$;
\item on the locus where $\Pi$ is an isomorphism, $\F_Y$ is the birational
transform of $\cG$;
\item over $X\times (\A^1\setminus\{0\})$, the foliation $\F_Y$ coincides
with the product foliation.
\end{enumerate}
\end{lemma}

\begin{proof}
Let $r:=\operatorname{rk}(\F)$. Since $\F\subset T_X$ is saturated, it is a
torsion-free coherent sheaf of rank $r$. Hence $\cG=p_1^*\F\subset T_{W/\A^1}$ is again a torsion-free saturated coherent subsheaf of rank $r$. It is integrable because $\F$ is integrable.

Since $\Pi\colon Y\to W$ is birational and $Y$ is normal, we have
$K(Y)=K(W)$. Thus the generic fibre of $\cG$ determines an $r$-dimensional
$K(Y)$-subspace
\[
\cG_\eta \subset T_{W/\A^1}\otimes K(W)
      = T_{Y/\A^1}\otimes K(Y),
\]
where the equality is via the birational identification.

We define $\F_Y$ to be the unique saturated subsheaf of $T_{Y/\A^1}$ whose
generic fibre is $\cG_\eta$. Equivalently, if $U\subset Y$ is any open set on
which $\Pi$ is an isomorphism onto its image, then on $U$ we set
\[
\F_Y|_U := (\Pi|_U)^*(\cG),
\]
and then extend to all of $Y$ by saturation inside $T_{Y/\A^1}$.
By construction, $\F_Y$ is a coherent saturated subsheaf of rank $r$, and
it agrees with the birational transform of $\cG$ on every open set where
$\Pi$ is an isomorphism.

We next show integrability. Since $\F_Y$ is saturated in the torsion-free
sheaf $T_{Y/\A^1}$, the quotient $T_{Y/\A^1}/\F_Y$ is torsion-free. The Lie bracket induces an $\cO_Y$-morphism
\[
[\ ,\ ]\colon \wedge^2\F_Y \longrightarrow T_{Y/\A^1}/\F_Y.
\]
On the dense open subset where $\Pi$ is an isomorphism, $\F_Y$ is identified
with $\cG$, and $\cG$ is integrable. Hence the above morphism vanishes on a
dense open subset of $Y$. Since the target is torsion-free, it follows that
the bracket vanishes identically. Therefore $\F_Y$ is integrable.

We now prove $\G_m$-equivariance. Since $\mathcal I$ is $\G_m$-invariant,
the blowup of $\mathcal I$ is $\G_m$-equivariant, and hence so is its
normalization $Y$. The product foliation $\cG=p_1^*\F\subset T_{W/\A^1}$ is preserved by the $\G_m$-action, because the action is trivial on the
$X$-factor and standard on the $\A^1$-factor. For any $g\in \G_m$, the
pullback $g^*\F_Y$ is again a saturated subsheaf of $T_{Y/\A^1}$. On the
dense open set where $\Pi$ is an isomorphism, it agrees with the transform
of $g^*\cG=\cG$. Hence $g^*\F_Y$ and $\F_Y$ have the same generic fibre
inside $T_{Y/\A^1}\otimes K(Y)$. By uniqueness of the saturated subsheaf
with a given generic fibre, we conclude that $g^*\F_Y=\F_Y$ for every $g\in \G_m$. Thus $\F_Y$ is $\G_m$-equivariant.

Finally, since $\mathcal I$ is supported on $X\times \{0\}$, the morphism
$\Pi$ is an isomorphism over $X\times (\A^1\setminus\{0\})$ and so, by construction, $\F_Y$ restricts there to the product foliation.
\end{proof}

Now let $E$ be a prime divisor over $X$. We choose $r_0\in \Z_{>0}$ such that $r_0L$ is Cartier and set $V_m:=H^0(X,mr_0L)$ for $m\in \Z_{>0}$. We define the filtration associated to $E$ by
\[
\F_E^x V_m
:=
H^0(X,mr_0L-\lfloor x\rfloor E)
\subset V_m.
\]

For $m\in \Z_{>0}$ and $p\in \Z_{\ge 0}$, let
\[
I_{(m,p)}
:=
\operatorname{Im}
\Bigl(
\F_E^p V_m\otimes_{\C}\cO_X(-mr_0L)\longrightarrow \cO_X
\Bigr).
\]
We also choose integers $e_+\gg 1$ and $e_-\ll 0$ such that $e_+ > r_0 \cdot \tau(E)$ and  $e_- < 0$, which imply that for $m\gg 1$ the flag
ideal
\[
\mathcal I_m
:=
I_{(m,me_+)}
+
I_{(m,me_+-1)}s
+\cdots+
I_{(m,me_-+1)}s^{m(e_+-e_-)-1}
+(s^{m(e_+-e_-)})
\subset \cO_{X\times \A^1}
\]
is well-defined. For each $r\gg 1$, let
\[
\Pi_r\colon \mathcal Y_r\to X\times \A^1
\]
be the normalized blowup of the flag ideal $\mathcal I_r$, and let $E_r$
be the effective Cartier divisor on $\mathcal Y_r$ defined by
\[
\mathcal I_r\cdot \cO_{\mathcal Y_r}
=
\cO_{\mathcal Y_r}(-E_r).
\]
We also set
\[
\mathcal M_r:=\Pi_r^*p_1^*L-\frac{1}{rr_0}E_r.
\]

\begin{proposition}\label{prop:Fujita-blowup-F-compatible}
There exists a $\G_m$-equivariant saturated integrable subsheaf $\F_{\mathcal Y_r}\subset T_{\mathcal Y_r/\A^1}$ such that:
\begin{enumerate}
\item $\F_{\mathcal Y_r}$ is the birational transform of the product
foliation $p_1^*\F$;
\item $\F_{\mathcal Y_r}$ coincides with $p_1^*\F$ over
$X\times (\A^1\setminus\{0\})$;
\item $\F_{\mathcal Y_r}$ is algebraically integrable.
\end{enumerate}
\end{proposition}

\begin{proof}
Since the flag ideal $\mathcal I_r$ is $\G_m$-invariant and supported on the
central fibre, Lemma~\ref{lem:birational-transform-product-foliation}
applies and yields a $\G_m$-equivariant saturated integrable subsheaf $\F_{\mathcal Y_r}\subset T_{\mathcal Y_r/\A^1}$ which is the birational transform of the product foliation $p_1^*\F$ and
agrees with it over $X\times (\A^1\setminus\{0\})$.

It remains to prove algebraic integrability of $K_{\F_{\mathcal Y_r}}$. Let $y\in \mathcal Y_r$ be a very general point. Since the exceptional locus of $\Pi_r$ and the central fibre are proper closed subsets, we may assume that $y$ lies outside both. Then $\Pi_r$ is an isomorphism near $y$, and
\[
\Pi_r(y)=(x,a)\in X\times (\A^1\setminus\{0\}).
\]
Near $y$, the foliation $\F_{\mathcal Y_r}$ is identified with the product foliation $p_1^*\F$, so the leaf through $y$ corresponds to $L_x\times \{a\}$ where $L_x$ is the leaf of $\F$ through $x$.

Since $\F$ is algebraically integrable, the leaf through a very general point $x\in X$ is Zariski open in an algebraic subvariety $\bar{L_x}\subset X$. Hence $L_x\times \{a\}$ is Zariski open in the algebraic subvariety $\bar{L_x}\times \{a\}\subset X\times \A^1.$ Pulling back through the local isomorphism $\Pi_r$, we see that the leaf of $\F_{\mathcal Y_r}$ through $y$ is Zariski open in the strict transform of $\bar{L_x}\times \{a\}$, which is an algebraic subvariety of $\mathcal Y_r$. Therefore $\F_{\mathcal Y_r}$ is algebraically integrable.
\end{proof}

By the above, once we equip $\Y_r$ with the semiample polarization $\mathcal M_r$, the triple $(\mathcal Y_r,\F_{\mathcal Y_r},\mathcal M_r)/\A^1$ is a normal semiample $\F$-compatible test configuration. This extends Odaka's usual construction \cite[Definition 3.1]{Odaka12} to $\F$-compatible test configurations.

\begin{theorem}\label{thm:t-Ding-prime-divisor}
Assume that there exists $\delta\in [0,1)$ such that
\[
\Ding^{[t]}(\X,\F_{\X},\L)\ge \delta\cdot J^{\NA}(\X,\L)
\]
for every normal semiample $\F$-compatible test configuration $(\X,\F_{\X},\L)/\A^1$ for $(X,\F,L)$. Then
\[
\beta^{[t]}_{X,\F,L}(E)\ge \delta\cdot j_L(E)
\]
for every prime divisor $E$ over $X$.
In particular, if $(X,\F,L)$ is $t$-Ding semistable, then
\[
\beta^{[t]}_{X,\F,L}(E)\ge 0
\]
for every prime divisor $E$ over $X$.
\end{theorem}

\begin{proof}
Let $E$ be a prime divisor over $X$. We shall follow the idea of proof in \cite[Proof of Theorem 4.1]{Fujita16}. As before, we consider the flag ideals $\mathcal I_m$ corresponding to the filtration of $E$ and for each $r\gg 1$, we take the normalisation of the blow-up $\Pi_r\colon \mathcal Y_r\to X\times \A^1$ with divisor $\M_r$. Then, by Proposition~\ref{prop:Fujita-blowup-F-compatible}, $(\mathcal Y_r,\F_{\mathcal Y_r},\mathcal M_r)/\A^1$ is a normal semiample $\F$-compatible test configuration. Hence, by the assumption and Remark \ref{rem:ample-semiample-interchange}, we have
\[
\Ding^{[t]}(\mathcal Y_r,\F_{\mathcal Y_r},\mathcal M_r)
\ge
\delta\,J^{\NA}(\mathcal Y_r,\mathcal M_r).
\]

Let $(\bar{\mathcal Y}_r,\bar{\F}_{\mathcal Y_r},\bar{\mathcal M}_r)
\to \PP^1$ be the natural compactification. By the definition of the mixed correction divisor we have
\[
\lct^{[t]}\!\left(
\bar{\mathcal Y}_r,\bar{\F}_{\mathcal Y_r};
D^{[t]}_{(\mathcal Y_r,\F_{\mathcal Y_r},\mathcal M_r)};
\bar{\mathcal Y}_{r,0}
\right)
=
\lct^{[t]}\!\left(
X\times \A^1,\ p_1^*\F;\
\mathcal I_r^{1/(rr_0)}; (s)
\right).
\]
We now define
\[
d_{r,\delta}
:=
(1-t)
+\delta\,\lambda^{\max}(\mathcal Y_r,\mathcal M_r)
+(1-\delta)\frac{\bar{\mathcal M}_r^{\,n+1}}{(n+1)V}.
\]
Since
\[
J^{\NA}(\mathcal Y_r,\mathcal M_r)
=
\lambda^{\max}(\mathcal Y_r,\mathcal M_r)
-\frac{\bar{\mathcal M}_r^{\,n+1}}{(n+1)V},
\]
the inequality
\[
\Ding^{[t]}(\mathcal Y_r,\F_{\mathcal Y_r},\mathcal M_r)\ge
\delta\,J^{\NA}(\mathcal Y_r,\mathcal M_r)
\]
is equivalent, by Definition~\ref{def:t-Ding}, to
\[
\lct^{[t]}\!\left(
\bar{\mathcal Y}_r,\bar{\F}_{\mathcal Y_r};
D^{[t]}_{(\mathcal Y_r,\F_{\mathcal Y_r},\mathcal M_r)};
\bar{\mathcal Y}_{r,0}
\right)
\ge d_{r,\delta},
\]
hence to
\[
\lct^{[t]}\!\left(
X\times \A^1,\ p_1^*\F;\
\mathcal I_r^{1/(rr_0)};\ (s)
\right)\ge d_{r,\delta}.
\]
By the valuative definition of the mixed threshold, this means that for every
divisorial valuation $w$ over $X\times \A^1$,
\[
A^{[t]}_{X\times \A^1,\ p_1^*\F}(w)
-\frac{1}{rr_0}w(\mathcal I_r)
-d_{r,\delta}\,w(t)
\ge 0.
\tag{$*$}
\]

Now, let $\sigma\colon Y\to X$ be a fixed log resolution on which $E$
appears as a prime divisor, and let $\Pi\colon \widetilde Y\to Y\times \A^1$ be the blowup of $E\times \{0\}$. Denote by $\widetilde E$ the exceptional divisor. We will evaluate the expression $(*)$ at $\ord_{\widetilde E}$.

We first compute the three terms appearing in $(*)$.

\smallskip

\noindent
\emph{Claim 1.}
\[
\ord_{\widetilde E}(t)=1.
\]
\emph{Proof.}
This is immediate from the definition of $\widetilde E$, since
$\widetilde E$ is the exceptional divisor of the blowup of the ideal
$(\cO_Y(-E),t)$.
\qed

\smallskip

\noindent
\emph{Claim 2.}
\[
\ord_{\widetilde E}(\mathcal I_r)\ge re_+.
\]
\emph{Proof.}
By construction of the filtration, we have $I_{(r,j)}\cdot \cO_Y \subset \cO_Y(-jE)$ for all $j\ge 0)$. Hence
\[
\mathcal I_r\cdot \cO_{Y\times \A^1}
\subset
\cO_Y(-re_+E)
+\cO_Y(-(re_+-1)E)t
+\cdots+
\cO_Y(-E)t^{re_+-1}
+(t^{re_+}).
\]
The right-hand side is exactly $(\cO_Y(-E)+(t))^{re_+}$. Since $\widetilde E$ is the exceptional divisor of the blowup of $(\cO_Y(-E),t)$, we have $\ord_{\widetilde E}(\cO_Y(-E)+(t))=1$ and therefore $\ord_{\widetilde E}(\mathcal I_r)\ge re_+$.
\qed

\smallskip

\noindent
\emph{Claim 3.}
\[
A^{[t]}_{X\times \A^1,\ p_1^*\F}(\widetilde E)
=
A^{[t]}_{X,\F}(E)+(1-t).
\]
\emph{Proof.}
We compute the ordinary and foliated parts separately. For the ordinary discrepancy, $E\times \A^1$ is a divisor over
$X\times \A^1$ with
\[
A_{X\times \A^1}(E\times \A^1)=A_X(E).
\]
Blowing up the smooth codimension-two center $E\times\{0\}\subset Y\times \A^1$
adds ordinary discrepancy $1$. Hence
\[
A_{X\times \A^1}(\widetilde E)=A_X(E)+1.
\]

For the foliated discrepancy, let $\F_Y$ be the induced foliation on $Y$.
We claim that
\[
A_{X\times \A^1,\ p_1^*\F}(\widetilde E)=A_{X,\F}(E).
\]
We distinguish two cases.

\smallskip

\noindent
\emph{Case 1: $E$ is $\F_Y$-invariant.}
We choose local coordinates $(x_0,\dots,x_n)$ near a general point of $E$ such that $E=(x_0=0)$ and $\F_Y$ is
generated by vector fields tangent to $E$. On $Y\times \A^1$, with
coordinate $s$ on $\A^1$, the product foliation $p_1^*\F_Y$ is generated by the same vector fields, viewed on the first factor. The blowup of $E\times\{0\}$ is locally the blowup of the ideal $(x_1,\ell)$. In the chart $x_0=u$, $\ell=uv$ the pullbacks of the generators remain regular and tangent to the exceptional divisor $(u=0)$. Thus the induced foliation on $\widetilde Y$ has no additional discrepancy along $\widetilde E$, and $\widetilde E$ is invariant for the induced foliation. Therefore
\[
A_{X\times \A^1,\ p_1^*\F}(\widetilde E)=A_{X,\F}(E).
\]

\smallskip

\noindent
\emph{Case 2: $E$ is transverse to $\F_Y$.}
We again choose local coordinates $(x_0,\dots,x_n)$ near a general point of $E$ such that $E=(x_0=0)$ and $\F_Y$ is locally
generated by
\[
\frac{\partial}{\partial x_0},
\frac{\partial}{\partial x_1},
\dots,
\frac{\partial}{\partial x_r}.
\]
Again let $\ell$ be the coordinate on $\A^1$, so the product foliation
$p_1^*\F_Y$ on $Y\times \A^1$ is generated by the same vector fields.
After blowing up the ideal $(x_0,\ell)$, in the chart $x_0=u$, $\ell=uv$ the pullback of $\partial/\partial x_0$ is
\[
\frac{\partial}{\partial u}-\frac{v}{u}\frac{\partial}{\partial v}.
\]
After saturation, the transformed foliation is locally generated by
\[
u\frac{\partial}{\partial u}-v\frac{\partial}{\partial v},
\frac{\partial}{\partial x_1},
\dots,
\frac{\partial}{\partial x_r}.
\]
Hence $\widetilde E=(u=0)$ is invariant for the transformed foliation, and
exactly one factor of $u$ is introduced in the determinant. Equivalently,
the coefficient of $\widetilde E$ in the foliated canonical divisor increases
by $1$. Since $E$ is transverse, $\epsilon(E)=1$, whereas
$\epsilon(\widetilde E)=0$. Therefore
\[
A_{X\times \A^1,\ p_1^*\F}(\widetilde E)=A_{X,\F}(E).
\]
Thus in either case
\[
A_{X\times \A^1,\ p_1^*\F}(\widetilde E)=A_{X,\F}(E).
\]
Combining this with the ordinary discrepancy computation gives
\[
A^{[t]}_{X\times \A^1,\ p_1^*\F}(\widetilde E)
=
(1-t)(A_X(E)+1)+tA_{X,\F}(E)
=
A^{[t]}_{X,\F}(E)+(1-t).
\]
\qed

Substituting Claims~1--3 into $(*)$ gives
\[
A^{[t]}_{X,\F}(E)+(1-t)-\frac{e_+}{r_0}-d_{r,\delta}\ge 0,
\]
that is,
\[
A^{[t]}_{X,\F}(E)\ge d_{r,\delta}-(1-t)+\frac{e_+}{r_0}.
\tag{$**$}
\]

We now compute the limit of $d_{r,\delta}$. First, exactly as in \cite[Claim~4.3]{Fujita16}, since $\max\{j\in \Z_{\ge 0}\mid I_{(r,j)}\neq 0\} = \lfloor r r_0 T_L(E)\rfloor$ and by the explicit description of the filtration associated to $(\mathcal Y_r,rr_0\mathcal M_r)$ we have
\[
\lim_{r\to\infty}\lambda^{\max}(\mathcal Y_r,\mathcal M_r)
=
-\frac{e_+}{r_0}+T_L(E).
\]
Furthermore, by the standard intersection formula for the blowup test configurations,
\[
\frac{\bar{\mathcal M}_r^{\,n+1}}{(n+1)V}
=
-\frac{e_+}{r_0}
+
\frac{1}{V}\int_0^{e_+/r_0}\vol(L-xE)\,dx
+o(1),
\]
and because $e_+/r_0>T_L(E)$, the integral equals
\[
\int_0^{T_L(E)}\vol(L-xE)\,dx
=
V\,S_L(E).
\]
Hence, we obtain 
\[
\lim_{r\to\infty}
\frac{\bar{\mathcal M}_r^{\,n+1}}{(n+1)V}
=
-\frac{e_+}{r_0}+S_L(E).
\]

Therefore
\[
\lim_{r\to\infty} d_{r,\delta}
=
(1-t)-\frac{e_+}{r_0}
+\delta\,T_L(E)
+(1-\delta)\,S_L(E).
\]
Passing to the limit in the expresion $(**)$, we obtain
\[
A^{[t]}_{X,\F}(E)
\ge
\delta\,T_L(E)+(1-\delta)\,S_L(E).
\]
Rearranging, this is
\[
A^{[t]}_{X,\F}(E)-S_L(E)\ge \delta\,(T_L(E)-S_L(E)),
\]
i.e.
\[
\beta^{[t]}_{X,\F,L}(E)\ge \delta\,j_L(E).
\]
This is the desired inequality, completing the proof.
\end{proof}

\subsection{\texorpdfstring{A Fujita-Li valuative criterion for $t$-Kstability}{A Fujita-Li valuative criterion for t-Kstability}}

\begin{theorem}\label{thm:main-prime-divisor-criterion}
Let $(X,\F,t)$ be an adjoint Fano foliated structure and let $L\sim_{\Q}-K^{[t]}_{X,\F}$. Fix $\delta\in [0,1)$. Then the following are equivalent:
\begin{enumerate}
\item
\[
\DF^{[t]}(\X,\F_{\X},\L)\ge \delta\,J^{\NA}(\X,\L)
\]
for every normal ample $\F$-compatible test configuration $(\X,\F_{\X},\L)/\A^1$ for $(X,\F,L)$;
\item
\[
\beta^{[t]}_{X,\F,L}(E)\ge \delta\,j_L(E)
\]
for every prime divisor $E$ over $X$;
\item
\[
\beta^{[t]}_{X,\F,L}(v)\ge \delta\,j_L(v)
\]
for every $\F$-dreamy divisorial valuation $v$ over $X$.
\end{enumerate}
\end{theorem}

\begin{proof}
\emph{$(1)\Rightarrow (2)$.}
By Corollary~\ref{cor:t-K-vs-t-Ding}, condition~(1) is equivalent to the
corresponding $t$-Ding inequality for all normal (semi-)ample $\F$-compatible test
configurations. By the usual approximation argument, it is enough to test
$t$-Ding stability on normal semiample $\F$-compatible test configurations.
Hence Theorem~\ref{thm:t-Ding-prime-divisor} applies and gives~(2).

\smallskip

\noindent
\emph{$(2)\Rightarrow (3)$.}
This is immediate, since every $\F$-dreamy divisorial valuation is by definition induced by a prime divisor $E$ over $X$.

\smallskip

\noindent
\emph{$(3)\Rightarrow (1)$.}
This follows immediately from Theorem \ref{thm: DF description for special t.c.}, Lemma \ref{lem:vertical-divisors-invariant} and Corollary \ref{cor:special-F-compatible-implies-F-dreamy}, together with the fact that it suffices to test $t$-K-(semi/uniform)stability on $\F$-compatible
special test configurations by Corollary \ref{cor:reduction-special-mixed-DF}.
\end{proof}

\begin{corollary}\label{cor:t-K-semistable-prime-divisor}
Let $(X,\F,t)$ be an adjoint Fano foliated structure and set $L\sim_{\Q}-K^{[t]}_{X,\F}$
Then:
\begin{enumerate}
\item the following are equivalent:
\begin{enumerate}
\item $(X,\F,L)$ is uniformly $t$-K-stable;
\item there exists $\delta\in (0,1)$ such that
\[
\beta^{[t]}_{X,\F,L}(E)\ge \delta\,j_L(E)
\]
for every prime divisor $E$ over $X$;
\item there exists $\delta\in (0,1)$ such that
\[
\beta^{[t]}_{X,\F,L}(v)\ge \delta\,j_L(v)
\]
for every $\F$-dreamy divisorial valuation $v$ over $X$;
\end{enumerate}
\item the following are equivalent:
\begin{enumerate}
\item $(X,\F,L)$ is $t$-K-semistable;
\item
\[
\beta^{[t]}_{X,\F,L}(E)\ge 0
\]
for every prime divisor $E$ over $X$;
\item
\[
\beta^{[t]}_{X,\F,L}(v)\ge 0
\]
for every $\F$-dreamy divisorial valuation $v$ over $X$.
\end{enumerate}
\end{enumerate}
\end{corollary}

\begin{proof}
We apply Theorem~\ref{thm:main-prime-divisor-criterion} with $\delta\in(0,1)$
for part~(1), and with $\delta=0$ for part~(2).
\end{proof}

\section{K-stability via birational invariants}\label{sec: birational inv for t ks}

In this section we will apply the approach in \cite{BJ20} to show that $t$-K-stability for adjoint Fano foliated structures is given via the birational $\alpha^{[t]}$ and $\delta^{[t]}$ invariants.

\subsection{\texorpdfstring{Mixed basis type thresholds and the $\delta^{[t]}$-invariant}{Mixed basis type thresholds and the deltat-invariant}}

In this subsection will show that the mixed delta invariant obtains a natural description in terms of basis type divisors and the mixed log canonical threshold $\lct^{[t]}$.

Throughout this subsection, we let $(X,\F,t)$ be an adjoint Fano foliated structure, with fixed $t\in [0,1)\cap \QQ$, and let $L\sim_{\QQ}-K^{[t]}_{X,\F}$. Let also $V:=L^n$.

For each $m\in \Z_{>0}$ such that $mL$ is Cartier and $H^0(X,mL)\neq 0$, let $N_m:=h^0(X,mL)$. Given a basis $s_1,\dots,s_{N_m}$ of $H^0(X,mL)$, the associated \emph{$m$-basis type divisor} is
\[
D=\frac{1}{mN_m}\sum_{i=1}^{N_m}\{s_i=0\}.
\]

\begin{definition}
\label{def:delta-m-t}
For $m$ sufficiently divisible, we define
\[
\delta^{[t]}_m(X,\F;L)
:=
\inf\bigl\{
\lct^{[t]}(X,\F;D)
\mid
D \text{ is an $m$-basis type divisor of }L
\bigr\}.
\]
We then define the mixed stability threshold by
\[
\delta^{[t]}(X,\F;L)
:=
\limsup_{m\to\infty}\delta^{[t]}_m(X,\F;L).
\]
\end{definition}

We also recall the finite-level expected vanishing invariant. For a valuation
$v\in \Val_X^*$ and $m$ sufficiently divisible, let $S_m(v):=\max_D v(D)$, where the maximum is taken over all $m$-basis type divisors $D$ of $L$.
Equivalently, if $a_{m,1}(v)\le \cdots \le a_{m,N_m}(v)$ are the vanishing orders
of a basis adapted to the filtration induced by $v$, then
\[
S_m(v)=\frac{1}{mN_m}\sum_{j=1}^{N_m} a_{m,j}(v).
\]
We have $S(v) = \lim_{m\to \infty}S_m(v)$ where the limit exists by \cite{BJ20}.

\begin{proposition}
\label{prop:mixed-delta-m-valuation}
For every sufficiently divisible $m$ we have
\[
\delta^{[t]}_m(X,\F;L)
=
\inf_{v\in \Val_X^*}\frac{A^{[t]}_{X,\F}(v)}{S_m(v)}=\inf_{v\in \Val_X^{\mathrm{Div}, *}}\frac{A^{[t]}_{X,\F}(v)}{S_m(v)}.
\]
\end{proposition}

\begin{proof}
Fix an $m$-basis type divisor $D$. By Definition~\ref{def:mixed-lct},
\[
\lct^{[t]}(X,\F;D)
=
\inf_{v\in \Val_X^*}\frac{A^{[t]}_{X,\F}(v)}{v(D)}.
\]
Taking the infimum over all $m$-basis type divisors yields
\[
\delta^{[t]}_m(X,\F;L)
=
\inf_D \inf_{v\in \Val_X^*}\frac{A^{[t]}_{X,\F}(v)}{v(D)}.
\]
We first prove the inequality
\[
\delta^{[t]}_m(X,\F;L)\ge
\inf_{v\in \Val_X^*}\frac{A^{[t]}_{X,\F}(v)}{S_m(v)}.
\]
To see this, note that for every valuation $v$ and every $m$-basis type divisor $D$, we have $v(D)\le S_m(v)$ hence
\[
\frac{A^{[t]}_{X,\F}(v)}{v(D)}
\ge
\frac{A^{[t]}_{X,\F}(v)}{S_m(v)}.
\]
Taking first the infimum over $v$, and then the infimum over $D$, gives the
desired inequality.

For the reverse inequality, fix $v\in \Val_X^*$. By the characterisation of
$S_m(v)$ via basis type divisors, there exists an $m$-basis type divisor $D_v$ such that $v(D_v)=S_m(v)$. Therefore
\[
\lct^{[t]}(X,\F;D_v)
=
\inf_{w\in \Val_X^*}\frac{A^{[t]}_{X,\F}(w)}{w(D_v)}
\le
\frac{A^{[t]}_{X,\F}(v)}{v(D_v)}
=
\frac{A^{[t]}_{X,\F}(v)}{S_m(v)}.
\]
Taking the infimum over all valuations $v$ gives
\[
\delta^{[t]}_m(X,\F;L)
\le
\inf_{v\in \Val_X^*}\frac{A^{[t]}_{X,\F}(v)}{S_m(v)}.
\]
Combining the two inequalities proves the first formula.

The divisorial version follows from the divisorial expression for
$\lct^{[t]}(X,\F;D)$ in Definition~\ref{def:mixed-lct}, exactly as in the
usual argument.
\end{proof}

\begin{theorem}
\label{thm:mixed-theorem-4-4}
We have
\[
\delta^{[t]}(X,\F;L)
=
\inf_{v\in \Val_X^*}\frac{A^{[t]}_{X,\F}(v)}{S(v)}
=
\inf_{v\in \Val^{\mathrm{div}, *}_X}\frac{A^{[t]}_{X,\F}(v)}{S(v)}.
\]
Moreover, the $\limsup$ in Definition~\ref{def:delta-m-t} is in fact a limit:
\[
\delta^{[t]}(X,\F;L)=\lim_{m\to\infty}\delta^{[t]}_m(X,\F;L).
\]
\end{theorem}

\begin{proof}
We first show that
\[
\delta^{[t]}(X,\F;L)\le
\inf_{v\in \Val_X^*}\frac{A^{[t]}_{X,\F}(v)}{S(v)}.
\]
Fix $v\in \Val_X^*$. Since $S_m(v)\to S(v)$ as $m\to\infty$, Proposition
\ref{prop:mixed-delta-m-valuation} gives
\[
\delta^{[t]}_m(X,\F;L)
\le
\frac{A^{[t]}_{X,\F}(v)}{S_m(v)}
\]
for every sufficiently divisible $m$. Passing to the limsup yields
\[
\delta^{[t]}(X,\F;L)
\le
\frac{A^{[t]}_{X,\F}(v)}{S(v)}.
\]
Taking the infimum over all $v$ proves the inequality.

For the reverse inequality, fix $\varepsilon>0$. By \cite[Corollary 2.10]{BJ20} there exists $m_0=m_0(\varepsilon)$ such that for all sufficiently divisible $m\ge m_0$ and all valuations $v\in \Val_X^*$, $S_m(v)\le (1+\varepsilon)S(v)$. Therefore, by  roposition~\ref{prop:mixed-delta-m-valuation},
\[
\delta^{[t]}_m(X,\F;L)
=
\inf_v \frac{A^{[t]}_{X,\F}(v)}{S_m(v)}
\ge
(1+\varepsilon)^{-1}
\inf_v \frac{A^{[t]}_{X,\F}(v)}{S(v)}.
\]
Taking the liminf as $m\to\infty$ gives
\[
\liminf_{m\to\infty}\delta^{[t]}_m(X,\F;L)
\ge
(1+\varepsilon)^{-1}
\inf_v \frac{A^{[t]}_{X,\F}(v)}{S(v)}.
\]
Since $\varepsilon>0$ is arbitrary, we conclude that
\[
\liminf_{m\to\infty}\delta^{[t]}_m(X,\F;L)
\ge
\inf_v \frac{A^{[t]}_{X,\F}(v)}{S(v)}.
\]
Combined with the first inequality, this proves
\[
\delta^{[t]}(X,\F;L)
=
\inf_{v\in \Val_X^*}\frac{A^{[t]}_{X,\F}(v)}{S(v)}.
\]
The divisorial equality is obtained in the same way as in the classical case:
the right-hand side agrees with the limit of the divisorial expressions for
$\delta^{[t]}_m$ from Proposition~\ref{prop:mixed-delta-m-valuation}. Finally,
the previous inequalities show that $\limsup$ and $\liminf$ coincide, hence the
limit exists.
\end{proof}

Combining the main results of Sections \ref{sec: valuative invariants}, \ref{sec: reduction to special t.c.} and \ref{sec: ding stability} we obtain the following.

\begin{theorem}
\label{thm:dreamy-valuative-criterion}
The following are equivalent:
\begin{enumerate}
\item $(X,\F,L)$ is $t$-K-semistable (respectively uniformly $t$-K-stable);
\item
\[
\delta^{[t]}(X,\F)
\ge 1 \text{ (respectively, }>1).
\]
\end{enumerate}
\end{theorem}

\begin{proof}
Recall that by Corollary \ref{cor:t-K-semistable-prime-divisor} we have that  $(X,\F,L)$ is $t$-K-semistable (respectively uniformly $t$-K-stable) if and only if $\beta^{[t]}_{X,\F,L}(v)\ge 0$ (respectively $>\epsilon \cdot j(F))$ for every divisorial valuation $v$.

For $t$-K-semistability, since
\[
\beta^{[t]}_{X,\F,L}(v)=A^{[t]}_{X,\F}(v)-S_L(v),
\]
the inequality $\beta^{[t]}_{X,\F,L}(v)\ge 0$ is equivalent to
\[
\frac{A^{[t]}_{X,\F}(v)}{S_L(v)}\ge 1
\]
for every divisorial valuation $v$, which is equivalent to
$\delta^{[t]}(X,\F)\ge 1$, by Theorem \ref{thm:mixed-theorem-4-4}.

For uniform $t$-K-stability, by Corollary~\ref{cor:t-K-semistable-prime-divisor} it suffices to test $t$-K-(semi)stability against prime divisors $E$ where by Theorem \ref{thm:main-prime-divisor-criterion} for $v= \ord_E$, $A^{[t]}_{X,\F}(E)-S_L(E)\geq \epsilon (T_L(E)-S_L(E))$. Since $L$ is ample by assumption, \cite[Proposition 3.11]{BJ20} implies that $n^{-1}S_L(E)\leq T_L(E)-S_L(E)\leq nS_L(E)$. Rewriting the condition, we get that $(X,\F,t)$ is uniformly $t$-K-stable if and only if $\beta^{[t]}_{X,\F,L}(v)\geq \epsilon' S_L{(E)}$ for all such $E$. But, by Theorem \ref{thm:mixed-theorem-4-4} this is equivalent to the condition $\delta^{[t]}(X,\F)
>1$ as required.
\end{proof}

\begin{remark}\label{rem: K-stab just for foliations}
    The above results can also motivate the definition of a K-stability theory for Fano (algebraically integrable) foliations on a fixed variety $X$ using the invariants 
    \[\beta(E) = A_{X,\F}(E)- S_{-K_{\F}}(E)\text{ and }\delta := \inf_{E/X} \frac{A_{X,\F}(E)}{S_{-K_{\F}}(E)} \]
    which correspond to the endpoint versions of the mixed $\beta$ and $\delta$-invariants defined in Definition \ref{def: mixed beta} and Theorem \ref{thm:dreamy-valuative-criterion}. However, this definition is, unfortunately, vacuous. To see this, consider an algebraically integrable Fano foliation $\F$ on a variety $X$ with $-K_{\F}$ being ample. Consider a non-trivial prime $\F$-invariant divisor $E$ on $X$. For instance, if the foliation is induced by rational map $g:X\dashrightarrow Y$, we can pullback an invariant divisor on $Y$ to construct $E$. Then, since $E$ is $\F$-invariant on $X$, we have $A_{X,\F}(E)=0$. On the other hand, since $-K_{\F}$ is ample and $E$ does not induce a trivial valuation, we have $S_{-K_{\F}}(E)> 0$. Combining these, we obtain $\beta(E)<0$, and hence $\F$ is never K-semistable. 
\end{remark}

 \subsection{Some examples}\label{sec: examples}

In this section we provide some general results and examples of $t$-K-stable, $t$-K-semistable and $t$-K-unstable adjoint Fano foliated structures. We first recall the following (well-known) homogeneity lemma, which we will use throughout.

\begin{lemma}\label{lem: hom of S}
Let $X$ be a normal projective variety of dimension $n$, $L$ a big $\Q$-Cartier divisor, and $E$ a divisorial valuation. If $L \sim a(-K_X)$ with $a>0$, then
\[
S_L(E) = a \, S_{-K_X}(E).
\]
\end{lemma}

\begin{proposition}\label{prop:q-proportional-canonical}
Let $X$ be a K-semistable $\Q$-Fano variety and let $\F$ be an algebraically integrable foliation on $X$ such that $\F$ is log canonical. Assume that $K_{\F}\sim_{\Q} qK_X$ for some $q\in \Q$. Let $\lambda_t := 1+(q-1)t$. Assume moreover $A_{X,\F}(v)\ge q\,A_X(v)$ for every divisorial valuation $v$ over $X$. Then, for every $t\in [0,1]$
such that $\lambda_t>0$, the adjoint foliated structure $(X,\F,t)$ is
$t$-K-semistable.
\end{proposition}

\begin{proof}
Since $K_{X,\F}^{[t]}=tK_{\F}+(1-t)K_X \sim_{\Q} \lambda_t K_X$, we have $-K_{X,\F}^{[t]}\sim_{\Q} \lambda_t(-K_X)$. Hence, whenever $\lambda_t>0$, the adjoint foliated structure $(X,\F,t)$
is adjoint Fano, with polarisation $L_t:= -K_{X,\F}^{[t]}$

Let $v$ be any divisorial valuation over $X$. Since $L_t\sim_{\Q}\lambda_t(-K_X)$, the $S$-invariant satisfies $S_{L_t}(v)=\lambda_t\,S_{-K_X}(v)$ by Lemma \ref{lem: hom of S}. Therefore
\begin{align*}
\beta^{[t]}_{X,\F,L_t}(v)
&=(1-t)A_X(v)+tA_{X,\F}(v)-\lambda_t S_{-K_X}(v)\\
&=\lambda_t\bigl(A_X(v)-S_{-K_X}(v)\bigr)
+t\bigl(A_{X,\F}(v)-qA_X(v)\bigr).
\end{align*}
Since $X$ is K-semistable, the usual valuative criterion gives $A_X(v)-S_{-K_X}(v)\ge 0$ for every divisorial valuation $v$. By assumption, $A_{X,\F}(v)-qA_X(v)\ge 0$. As $\lambda_t>0$, it follows that $\beta^{[t]}_{X,\F,L_t}(v)\ge 0$ for every divisorial valuation $v$. Hence $(X,\F,t)$ is $t$-K-semistable
by Corollary~\ref{cor:t-K-semistable-prime-divisor}.
\end{proof}

\begin{corollary}\label{cor:minus-one-case}
Let $X$ be a K-semistable $\Q$-Fano variety and let $\F$ be an
algebraically integrable foliation on $X$ such that $(X,\F)$ is log
canonical and $K_{\F}\sim_{\Q} -qK_X$ for some $q\in \Q_{>0}$. Then, for every $0\le t<\frac{1}{|q-1|}$, the adjoint foliated structure $(X,\F,t)$ is $t$-K-semistable.
\end{corollary}

\begin{proof}
This follows directly from Proposition~\ref{prop:q-proportional-canonical}. Since $\lambda_t=1+(q-1)t$ we have $\lambda_t>0$ exactly when $t<\frac{1}{|q-1|}$. Moreover, since $\F$ is log canonical, we have $A_{X,\F}(v)\ge 0\ge qA_X(v)$ for every divisorial valuation $v$.
\end{proof}

The following is one possible example obtained from the above.

\begin{example}\label{ex:lef-cubic-pencil}
Let $\mathcal P\subset |\cO_{\PP^2}(3)|$ be a general pencil of plane cubics, i.e.
a pencil generated by two smooth cubic curves meeting transversely in $9$ distinct
base points, and such that all singular members of the pencil are nodal. Let
$\F_{\mathcal P}$ be the codimension one foliation on $\PP^2$ induced by the rational map $\PP^2 \dashrightarrow \PP^1$ defined by $\mathcal P$. Then $\F_{\mathcal P}$ is algebraically integrable and
\[
K_{\F_{\mathcal P}}\sim_{\Q}\cO_{\PP^2}(3)\sim_{\Q}-K_{\PP^2}.
\]
Moreover, $(\PP^2,\F_{\mathcal P})$ is log canonical. In particular, by
Corollary~\ref{cor:minus-one-case}, the adjoint Fano foliated structure $(\PP^2,\F_{\mathcal P},t)$ is $t$-K-semistable for every $0\le t<\frac12$.
\end{example}

\begin{corollary}
\label{cor:CY-foliations-semistable}
Let $X$ be a K-semistable Fano variety, and let $\F$ be a log canonical foliation on $X$ such that $K_{\F}\equiv 0$. Fix $t\in(0,1)\cap\Q$, and set $L:=-K^{[t]}_{X,\F}$. Then the adjoint foliated structure $(X,\F,t)$ is $t$-K-semistable.
\end{corollary}
\begin{proof}
    This is a special case of Proposition \ref{prop:q-proportional-canonical}, with $q=0$.
\end{proof}

We can also investigate what happens when $t$ is very small in specific situations.

\begin{proposition}
\label{prop:small-t-uniform-stability}
Let $X$ be a uniformly K-stable $\Q$-Fano variety, and let
$\F$ be a log canonical algebraically integrable foliation on $X$
such that $K_{\F}$ is $\Q$-Cartier. Then there exists
$\varepsilon>0$ such that for every $0<t<\varepsilon$ the adjoint foliated structure $(X,\F,t)$ is uniformly $t$-K-stable.
\end{proposition}

\begin{proof}
As before we set $L_t:=-K^{[t]}_{X,\F}=-(1-t)K_X-tK_{\F}$. Since $-K_X$ is ample and ampleness is an open condition in $N^1(X)_{\R}$, after replacing $\varepsilon>0$ by a smaller number if necessary, $L_t$ is ample for every $0<t<\varepsilon$. Hence $(X,\F,t)$ is an adjoint Fano foliated structure for all such $t$.

Let $\delta(X;L_t)$ denote the ordinary delta invariant of the polarised
variety $(X,L_t)$. Since $L_t\longrightarrow -K_X$ in $N^1(X)_{\R}$ as $t\to 0$, and since the delta invariant is continuous on the big cone
by \cite[Theorem 1.7]{ZhangContinuityDelta}, we have $\delta(X;L_t)\longrightarrow \delta(X;-K_X)$. Because $X$ is uniformly K-stable, we have $\delta(X;-K_X)>1$, therefore, after shrinking $\varepsilon$ again, we may assume that $(1-t)\delta(X;L_t)>1$ for every $0<t<\varepsilon$.

Now let $v$ be any divisorial valuation over $X$. Since $\F$ is
log canonical, we have $A_{X,\F}(v)\geq 0$ and thus $A^{[t]}_{X,\F}(v) \geq (1-t)A_X(v)$. Dividing by $S_{L_t}(v)$ gives
\[
\frac{A^{[t]}_{X,\F}(v)}{S_{L_t}(v)}
\geq
(1-t)\frac{A_X(v)}{S_{L_t}(v)}.
\]
Taking the infimum over all divisorial valuations $v$ yields
\[
\delta^{[t]}(X,\F;L_t)
\geq
(1-t)\delta(X;L_t)
>1.
\]
Hence, the adjoint foliated structure $(X,\F,t)$ is uniformly $t$-K-stable for every $0<t<\varepsilon$.
\end{proof}

We now present an example using the above.

\begin{example}
\label{ex:cubic-fourfold-pencil-foliation}
Let $X\subset \PP^5$ be a smooth cubic fourfold, and let
$H:=\cO_X(1)$. Let $\Lambda\subset |H|$ be a Lefschetz pencil of hyperplane sections, and let $f_\Lambda\colon X\dashrightarrow \PP^1$ be the associated rational map. Denote by $\F_\Lambda\subset T_X$ the saturated codimension one foliation tangent to the fibres of $f_\Lambda$.

Then $\F_\Lambda$ is algebraically integrable and $-K_X=3H$. Furthermore, on the open locus where $f_\Lambda$ is a morphism and the foliation is regular, we have an exact sequence
\[
0\longrightarrow T_{\F_\Lambda}
\longrightarrow T_X
\longrightarrow f_\Lambda^*T_{\PP1}
\longrightarrow 0.
\]
Since $f_\Lambda^*\cO_{\PP1}(1)\simeq \cO_X(H)$,
the normal bundle of the foliation is $\cO_X(2H)$. Equivalently, $K_{\F_\Lambda}= K_X+2H= -3H+2H = -H$, thus $-K_{\F_\Lambda}=H$ is ample, so $\F_\Lambda$ is a Fano foliation.

Assume moreover that the Lefschetz pencil is chosen generally, so that
the induced foliation has log canonical singularities. Then, since every
smooth cubic fourfold is K-stable by \cite[Corollary 1.2]{Liu22}, Proposition \ref{prop:small-t-uniform-stability} implies that there exists $\varepsilon>0$ such that $(X,\F_\Lambda,t)$ is uniformly $t$-K-stable for every $0<t<\varepsilon$.

On the other hand, this example is not $t$-K-semistable for all $t$. We will verify this statement using the mixed $\beta$ invariant. Let $D\in \Lambda$ be a smooth member of the pencil. Then $D$ is
$\F_\Lambda$-invariant, hence $A_{X,\F_\Lambda}(D)=0$. Since $X$ is smooth and $D$ is a prime divisor on $X$, we also have $A_X(D)=1$. For the adjoint polarisation we have $L_t = (3-2t)H$. As $D\sim H$ and $\dim X=4$, homogeneity of the $S$-invariant gives
\[
S_{L_t}(D)
=
S_{(3-2t)H}(D)
=
\frac{3-2t}{5}.
\]
Therefore
\[
\begin{aligned}
\beta^{[t]}_{X,\F_\Lambda,L_t}(D)
&=
A^{[t]}_{X,\F_\Lambda}(D)-S_{L_t}(D)\\
&=
(1-t)A_X(D)+tA_{X,\F_\Lambda}(D)
-\frac{3-2t}{5}\\
&=
1-t-\frac{3-2t}{5}\\
&=
\frac{2-3t}{5}.
\end{aligned}
\]
Hence $\beta^{[t]}_{X,\F_\Lambda,L_t}(D)<0$ for every $t>\frac23$, i.e., $(X,\F_\Lambda,t)$ is not $t$-K-semistable for every $t>2/3$. We expect that this is K-stable for all $t<2/3$.
\end{example}

The above Example prompts us to show the following.

\begin{proposition}\label{prop:instability-near-one}
Let $(X,\F)$ be an algebraically integrable Fano foliation on a normal
projective variety $X$, so that $-K_{\F}$ is ample. Then there exists $t_0 \in (0,1)$ such that for every $t \in [t_0,1]$ the
adjoint Fano foliated structure $(X,\F,t)$ is not $t$-K-semistable.
\end{proposition}

\begin{proof}
For $t \in [0,1]$, we take as before $L_t := -K^{[t]}_{X,\F}$. Since $L_1 = -K_{\F}$ is ample, ampleness is open, so there exists $t_0' \in
(0,1)$ such that $L_t$ is ample for all $t \in [t_0',1]$, hence $(X,\F,t)$ is an adjoint Fano foliated structure for all such $t$.

Since $\F$ is algebraically integrable, we can find an $\F$-invariant prime divisor $E$ on $X$, where, for every $t \in [0,1]$, $A^{[t]}_{X,\F}(E) = 1-t$. Furthermore, since $E$ defines a non-trivial valuation $S_{L_1}(E)=S_{-K_{\F}}(E)>0$.

Now the map $t \longmapsto S_{L_t}(E)$ is continuous at $t=1$, since the volume function varies continuously with the numerical class. Hence there exist $\eta>0$ and $t_0'' \in (0,1)$ such that $S_{L_t}(E)\geq \eta$ for all $t \in [t_0'',1]$.

Thus, by defining $t_0:=\max\{t_0',t_0'',1-\eta/2\}$, for every $t \in [t_0,1]$ we have
\[
\beta^{[t]}_{X,\F,L_t}(E)
=
A^{[t]}_{X,\F}(E)-S_{L_t}(E)
=
(1-t)-S_{L_t}(E)
\leq
(1-t)-\eta
<
0
\]
and $(X,\F,t)$ is not $t$-K-semistable by Corollary \ref{cor:t-K-semistable-prime-divisor}.
\end{proof}
We also provide a more general result extending Proposition \ref{prop:instability-near-one}. 

\begin{proposition}\label{prop:non-lc-foliation-obstructs-near-one}
Let $(X,\F,t)$ be an adjoint Fano foliated structure for some $t\in (0,1]$. Assume that the foliation $\F$ is not log canonical. Then there exists $t_0<1$ such that for every $t\in (t_0,1]$, the adjoint Fano foliated structure $(X,\F,t)$ is not $t$-K-semistable.
\end{proposition}

\begin{proof}
Since $\F$ is not log canonical, there exists a divisorial valuation $E$ over $X$ such that $A_{X,\F}(E)<0$. Notice that $A^{[t]}_{X,\F}(E)\to A_{X,\F}(E)<0$ as $t\to 1$, therefore, there exists $t_0<1$ such that $A^{[t]}_{X,\F}(E)<0$ for every $t\in (t_0,1]$.

We now set as before $L_t:=-K^{[t]}_{X,\F}$. Since $(X,\F,t)$ is adjoint Fano, the divisor $L_t$ is ample. In particular,
the $S$-invariant satisfies $S_{L_t}(E)>0$ It follows that for every $t\in (t_0,1]$,
\[
\delta^{[t]}(X,\F;L_t)
\leq
\frac{A^{[t]}_{X,\F}(E)}{S_{L_t}(E)}
<0<1.
\]
Therefore, by Theorem \ref{thm:dreamy-valuative-criterion}, $(X,\F,t)$ is not $t$-K-semistable.
\end{proof}

\begin{corollary}\label{cor:semistable-near-one-implies-lc-foliation}
Assume that there exists a sequence $t_i\to 1$ such that $(X,\F,t_i)$ is
$t_i$-K-semistable for every $i$. Then the foliation $\F$ is log canonical.
\end{corollary}

\begin{proof}
This is the contrapositive of Proposition~\ref{prop:non-lc-foliation-obstructs-near-one}.
\end{proof}

We also obtain a ``wall-crossing'' result in specific Fano situations.

\begin{proposition}
\label{prop:proportional-interval-wall-crossing}
Let $X$ be a klt $\Q$-Fano variety, and let $\F$ be an algebraically integrable foliation on $X$ such that $K_{\F}$ is $\Q$-Cartier. Assume that $K_{\F}\sim_{\Q} qK_X$ for some $q\in \Q_{>0}$. Then, for every $t\in[0,1]$ the set
\[
I_{X,\F}:=
\left\{
t\in[0,1]\mid (X,\F,t)\text{ is }t\text{-K-semistable}
\right\}
\]
is a closed interval, possibly empty.

In particular, the $t$-K-semistable locus has at most two boundary points in $[0,1]$.
\end{proposition}

\begin{proof}
Since $K_{\F}\sim_{\Q}qK_X$, we have $K^{[t]}_{X,\F}\sim_{\Q}
\bigl(1+(q-1)t\bigr)K_X$. For ease of notation we set $\lambda_t:=1+(q-1)t>0$ and $L_t:=-K^{[t]}_{X,\F}$ for every $t\in[0,1]$, which is ample. 

Let $v$ be a divisorial valuation over $X$. By Lemma \ref{lem: hom of S}, $S_{L_t}(v)=\lambda_t S_{-K_X}(v)$. Therefore
\[
\begin{aligned}
\beta^{[t]}_{X,\F,L_t}(v)
&=
A^{[t]}_{X,\F}(v)-S_{L_t}(v)\\
&=
(1-t)A_X(v)+tA_{X,\F}(v)
-\lambda_t S_{-K_X}(v)\\
&=
A_X(v)-S_{-K_X}(v)
+t\Bigl(A_{X,\F}(v)-A_X(v)-(q-1)S_{-K_X}(v)\Bigr).
\end{aligned}
\]
Thus, for each fixed divisorial valuation $v$, the function
\[
t\longmapsto \beta^{[t]}_{X,\F,L_t}(v)
\]
is affine-linear.

We now define
\[
I_v:=
\left\{
t\in[0,1]\big|
\beta^{[t]}_{X,\F,L_t}(v)\geq 0
\right\}.
\]
Since $\beta^{[t]}_{X,\F,L_t}(v)$ is affine-linear in $t$, the set $I_v$ is a closed interval, a closed half-interval, all of $[0,1]$, or the empty set. In particular, $I_v$ is closed and convex.

By Corollary \ref{cor:t-K-semistable-prime-divisor}, we have
\[
I_{X,\F}
=
\bigcap_{v\in \operatorname{Val}^{\mathrm{div},*}_X} I_v.
\]
An arbitrary intersection of closed convex subsets of $[0,1]$ is again
closed and convex. Hence $I_{X,\F}$ is a closed convex subset of
$[0,1]$, and therefore a closed interval, possibly empty.
\end{proof}

We now present an example of a $t$-K-unstable adjoint foliated structure.

\begin{example}
\label{ex:radial-foliation-beta}
Let $X=\PP^2$ with hyperplane class $H$, and let $\F$ be the radial foliation, i.e. the algebraically integrable foliation whose leaves are the lines through a fixed point $p\in \PP^2$. Then we have $K_X=-3H$, $K_{\F}=-H$. We fix $t\in [0,1)$. Then the adjoint divisor is
\[
K^{[t]}_{X,\F}
=
tK_{\F}+(1-t)K_X
=
-\bigl(3-2t\bigr)H,
\]
and hence $L:=-K^{[t]}_{X,\F}=(3-2t)H$, which is ample for all $t<1$, with $L^2=(3-2t)^2$.

Let $D\subset \PP^2$ be a line. Then $D\sim H$, hence $\tau_L(D)=3-2t$ and for $0\le x\le 3-2t$, $\vol(L-xD)=(3-2t-x)^2$. Therefore
\[
S_L(D)
=
\frac{1}{L^2}\int_0^{3-2t}(3-2t-x)^2\,dx
=
\frac{3-2t}{3}.
\]

Suppose that $D$ is invariant, i.e. $D$ passes through the center $p$. Then $D$ is $\F$-invariant, so $A_{X,\F}(D)=0$ and $A_X(D)=1$, hence $A^{[t]}_{X,\F}(D) = 1-t$.
Then
\[
\beta^{[t]}(D) =
(1-t)-\frac{3-2t}{3}= -\frac{t}{3}.
\]
In particular, $\beta^{[t]}(D)<0$ for all $t>0$, and hence 

for every $t>0$ the adjoint Fano foliated structure $(\PP^2,\F,t)$ is not $t$-K-semistable.
\end{example}

Example \ref{ex:radial-foliation-beta} can be also be generalised to arbitrary $\PP^n$ to give a number of $t$-K-unstable adjoint Fano foliated structures. Unfortunately, obtaining more geometrically interesting examples still remains a challenge.

\subsection{\texorpdfstring{The mixed $\alpha^{[t]}$-invariant and a sufficient criterion for $t$-K-semistability}{The mixed alphat-invariant and a sufficient criterion for t-K-semistability}}\label{sec: mixed alpha K-ss}

In this subsection we introduce the mixed analogue of the $\alpha$-invariant
and prove the analogue of \cite[Theorem~A]{BJ20}. As in the usual
theory, the argument is formal once we have the valuation-theoretic
descriptions of the $S$- and $T$-invariants.

\begin{definition}
\label{def:mixed-alpha}
Let $D\ge 0$ be an effective $\QQ$-divisor on $X$. We define the mixed $\alpha$-invariant of $(X,\F;L)$ by
\[
\alpha^{[t]}(X,\F;L)
:=
\inf\bigl\{
\lct^{[t]}(X,\F;D)
\mid
D\sim_{\QQ} L,\ D\ge 0
\bigr\}.
\]
\end{definition}

\begin{remark}\label{rem: val formula alpha}
Exactly as in the usual theory, we also have the valuation-theoretic
expression
\[
\alpha^{[t]}(X,\F;L)
=
\inf_{v\in \Val_X^*}\frac{A^{[t]}_{X,\F}(v)}{T_L(v)}.
\]
This is because, for every valuation $v$ and every effective $\QQ$-divisor
$D\sim_{\QQ}L$, we have $v(D)\le T_L(v)$ while conversely one can approximate $T_L(v)$ by divisors linearly
equivalent to $L$. The proof is identical to the usual one, with the
ordinary discrepancy replaced by $A^{[t]}_{X,\F}$ (see \cite[Corollary 4.2]{BJ20}).
\end{remark}

\begin{theorem}
\label{thm:mixed-alpha-delta}
We have
\[
\frac{n+1}{n}\,\alpha^{[t]}(X,\F;L)
\le
\delta^{[t]}(X,\F;L)
\le
(n+1)\alpha^{[t]}(X,\F;L).
\]
\end{theorem}

\begin{proof}
For every valuation $v\in \Val_X^*$, the usual inequalities between the expected vanishing and pseudo-effective thresholds give
\[
\frac{1}{n+1}T_L(v)\le S_L(v)\le T_L(v)
\]
(c.f. \cite[Equation (3.1)]{BJ20}).
Since $A^{[t]}_{X,\F}(v)\ge 0$, it follows that
\[
\frac{A^{[t]}_{X,\F}(v)}{T_L(v)}
\le
\frac{A^{[t]}_{X,\F}(v)}{S_L(v)}
\le
(n+1)\frac{A^{[t]}_{X,\F}(v)}{T_L(v)}.
\]
Taking the infimum over all $v\in \Val_X^*$ yields by Theorem~\ref{thm:mixed-theorem-4-4} and the above discussion
\[
\alpha^{[t]}(X,\F;L)
\le
\delta^{[t]}(X,\F;L)
\le
(n+1)\alpha^{[t]}(X,\F;L).
\]
In particular, since $L$ is ample (since we assume that $L\sim_{\Q}-K_{X,\F}^{[t]}$), \cite[Proposition 3.11]{BJ20} shows that 
\[
\frac{1}{n+1}\le \frac{S_L(v)}{T_L(v)}\le \frac{n}{n+1},
\]
which immediately gives $\delta^{[t]}\geq \frac{n}{n+1}\alpha^{[t]}$, as claimed.
\end{proof}
Now we also give a sufficient criterion for $t$-K-semistability.
\begin{corollary}
\label{cor:mixed-alpha-semistable}
If
\[
\alpha^{[t]}(X,\F;L)\ge \frac{n}{n+1} \text{ (respectively }>\frac{n}{n+1}),
\]
then
$(X,\F,L)$ is $t$-K-semistable (respectively uniformly $t$-K-stable).
\end{corollary}

\begin{proof}
By Theorem~\ref{thm:mixed-alpha-delta},
\[
\delta^{[t]}(X,\F;L)\ge \frac{n+1}{n}\alpha^{[t]}(X,\F;L).
\]
Hence, if $\alpha^{[t]}(X,\F;L)\ge \frac{n}{n+1}$
we obtain
\[
\delta^{[t]}(X,\F;L)\ge \frac{n+1}{n}\cdot \frac{n}{n+1}=1,
\]
and hence $(X,\F)$ is $t$-K-semistable by Theorem \ref{thm:dreamy-valuative-criterion}. The statement for $t$-uniform K-stability follows via the same argument by the same Theorem.
\end{proof}

\begin{corollary}\label{cor: k-ss to alpha}
    If $(X,\F,L)$ is $t$-K-semistable then $\alpha^{[t]}(X,\F;L)\ge \frac{1}{n+1}$.
\end{corollary}
\begin{proof}
    Since $(X,\F,L)$ is $t$-K-semistable, by Theorem \ref{thm:dreamy-valuative-criterion} we have $\delta^{[t]}(X,\F;L)\ge 1$. This implies that $1\leq (n+1)\alpha^{[t]}(X,\F;L)$, as required.
\end{proof}

\subsection{Mixed normalised volumes}

The definitions we have made so far imply the natural extension of the normalised volume in the mixed setting.

\begin{definition}
Let $v$ be a valuation centered on $X$. We define
\[
\widehat{\mathrm{vol}}^{[t]}_{X,\F}(v)
:=
A^{[t]}_{X,\F}(v)^n\,\mathrm{vol}(v),
\]
where $\mathrm{vol}(v)$ denotes the volume of the valuation.
\end{definition}

We will not analyse the properties of the mixed normalised volume here, but, in future work, we are planning to show how we can characterise $t$-K-semistability of an adjoint Fano foliated structure using mixed normalised volume on its affine cone, mimicking the results of \cite{Li2017EquivariantVolumeMinimization}.

\section{\texorpdfstring{Boundedness of $t$-Ksemistable adjoint Fano foliation structures}{Boundedness of t-Ksemistable adjoint Fano foliation structures}}\label{sec: boundedness}

In this section we will use the description of $t$-K-semistability via alpha invariants in Section \ref{sec: mixed alpha K-ss} along with the approach which first appeared in \cite{Jiang2020} and later refined in \cite[\S 7.2]{Xu2025} along with the boundedness results in \cite[Theorem B]{CHLMSX25} to show that $t$-K-semistable adjoint Fano foliated structures form a bounded set.

We will first need the following Lemmas, which will be key in the proof of the main theorem of this section. We first introduce some notation. Let $Y$ be a smooth $d$-fold, and let $\mathcal F_Y$ be a smooth
foliation of rank $r$. Set $q:=d-r$. Let $E\subset Y$ be a smooth
prime divisor, and let $x\in E$ be a very general point at which
$\mathcal F_Y$ and $E$ are smooth. We choose local coordinates $x_1,\ldots,x_d$ centered at $x$, with
$E=\{x_1=0\}$. Let $g:Z\to Y$ be the weighted blow-up at $x$ with weights
\[
(w_1,\ldots,w_d)=\left(\frac{k}{b},k,\ldots,k\right),
\]
where $b>0$ is rational and $k$ is sufficiently divisible. Let $F$
be the exceptional divisor, and let $\mathcal F_Z$ be the induced
foliation on $Z$.

\begin{lemma}
\label{lem:weighted-blowup-smooth-foliation-rank-r}
We have:
\[
A^{[t]}_{Y,\mathcal F_Y}(F)
=
\begin{cases}
(1-t)\left(\frac{k}{b}+(d-1)k\right)+trk,
& \text{if }E\text{ is }\mathcal F_Y\text{-invariant},\\[6pt]
(1-t)\left(\frac{k}{b}+(d-1)k\right)
+t\left(\frac{k}{b}+(r-1)k\right),
& \text{if }E\text{ is transverse to }\mathcal F_Y.
\end{cases}
\]
\end{lemma}

\begin{proof}
The ordinary log discrepancy of the weighted blow-up valuation is
\[
A_Y(F)=\sum_{i=1}^d w_i=\frac{k}{b}+(d-1)k.
\]

It remains to compute the foliated log discrepancy. Since
$\mathcal F_Y$ is smooth at $x$, the Frobenius theorem gives local
coordinates adapted to $\mathcal F_Y$. Moreover, since $E=\{x_1=0\}$, if $E$ is $\mathcal F_Y$-invariant, then $T_{\mathcal F_Y}
=
\langle
\partial_{x_2},\ldots,\partial_{x_{r+1}}
\rangle$ and If $E$ is transverse to $\mathcal F_Y$, then $T_{\mathcal F_Y}
=
\langle
\partial_{x_1},\ldots,\partial_{x_r}
\rangle$.

We first record the following local computation. Suppose that, in local
coordinates, $T_{\mathcal F_Y}$ is generated by $\partial_{x_i}$ for $i\in I$ where $|I|=r$. Let $v_w$ be the monomial divisorial valuation with
$v_w(x_i)=w_i$. Then 
\[
K_{\mathcal F_Z}
=
g^*K_{\mathcal F_Y}
+
\left(\sum_{i\in I}w_i\right)F
\] and hence $A_{Y,\mathcal F_Y}(v_w)=\sum_{i\in I}w_i$.

We now apply this formula in the two cases above. If $E$ is invariant, then $I=\{2,\ldots,r+1\}$, thus
\[
A_{Y,\mathcal F_Y}(F)=\sum_{i=2}^{r+1}w_i=rk.
\]
Therefore
\[
A^{[t]}_{Y,\mathcal F_Y}(F)
=
(1-t)A_Y(F)+tA_{Y,\mathcal F_Y}(F)
=
(1-t)\left(\frac{k}{b}+(d-1)k\right)+trk.
\]

Similarly, if $E$ is transverse, then $I=\{1,\ldots,r\}$, thus
\[
A_{Y,\mathcal F_Y}(F)
=
w_1+\sum_{i=2}^r w_i
=
\frac{k}{b}+(r-1)k.
\]
Therefore
\[
A^{[t]}_{Y,\mathcal F_Y}(F)
=
(1-t)\left(\frac{k}{b}+(d-1)k\right)
+
t\left(\frac{k}{b}+(r-1)k\right).
\]
This proves the lemma.
\end{proof}

Let now $\mu:Y\to X$ be a birational model on which a prime divisor
$E$ over $X$ appears. Assume that $Y$ is smooth at a very general
point $x\in E$, that $E$ is smooth at $x$, and that the induced
foliation $\mathcal F_Y$ is smooth at $x$. Let $a:=A^{[t]}_{X,\mathcal F}(E)$ and choose a rational number $b>a$, and let $g:Z\to Y$ be the weighted blow-up of $x$ with weights
\[
\left(\frac{k}{b},k,\ldots,k\right),
\]
where $k$ is sufficiently divisible. Let $F\subset Z$ be the exceptional
divisor.

\begin{lemma}
\label{lem:weighted-blowup-cancellation}
We have 
\[
A^{[t]}_{X,\mathcal F}(F)\le kd.
\]
\end{lemma}

\begin{proof}
Writing
\[
K^{[t]}_{Y,\mathcal F_Y}
=
\mu^*K^{[t]}_{X,\mathcal F}
+
c_EE+\cdots,
\]
the coefficient $c_E$ is
\[
c_E
=
A^{[t]}_{X,\mathcal F}(E)-\bigl((1-t)+t\varepsilon(E)\bigr)
=
a-\bigl((1-t)+t\varepsilon(E)\bigr).
\]
Since the weighted blow-up has $\ord_F(E)=\frac{k}{b}$, we get
\[
A^{[t]}_{X,\mathcal F}(F)
=
A^{[t]}_{Y,\mathcal F_Y}(F)
+
\left(a-\bigl((1-t)+t\varepsilon(E)\bigr)\right)\frac{k}{b}.
\]

We now use Lemma~\ref{lem:weighted-blowup-smooth-foliation-rank-r}. First suppose that $E$ is $\mathcal F_Y$-invariant. Then
$\varepsilon(E)=0$, and
\[
A^{[t]}_{Y,\mathcal F_Y}(F)
=
(1-t)\left(\frac{k}{b}+(d-1)k\right)+trk.
\]
Therefore
\[
\begin{aligned}
A^{[t]}_{X,\mathcal F}(F)
&=
(1-t)\left(\frac{k}{b}+(d-1)k\right)+trk
+\left(a-(1-t)\right)\frac{k}{b}  \\
&=
\frac{a}{b}k+\bigl((1-t)(d-1)+tr\bigr)k \\
&<
k+\bigl((1-t)(d-1)+tr\bigr)k\\
&=k\bigl(d-t(q-1)\bigr)< kd
\end{aligned}
\]

Now suppose that $E$ is transverse to $\mathcal F_Y$. Then
$\varepsilon(E)=1$, and
\[
A^{[t]}_{Y,\mathcal F_Y}(F)
=
(1-t)\left(\frac{k}{b}+(d-1)k\right)
+
t\left(\frac{k}{b}+(r-1)k\right).
\]
Thus
\[
\begin{aligned}
A^{[t]}_{X,\mathcal F}(F)
&=
\frac{k}{b}+\bigl((1-t)(d-1)+t(r-1)\bigr)k
+(a-1)\frac{k}{b} \\
&=
\frac{a}{b}k+\bigl((1-t)(d-1)+t(r-1)\bigr)k\\
&< k+\bigl((1-t)(d-1)+t(r-1)\bigr)k \\
&= k(d-tq) < kd.
\end{aligned}
\]
Hence in all cases $A^{[t]}_{X,\mathcal F}(F)<kd$. The equality statements for $b=a$ follow from the same calculation.
\end{proof}

\begin{theorem}
\label{thm:conditional-xu-type-boundedness}
Fix a positive integer $d$, a real number $\delta>0$, and a real number $t\in(0,1)\cap \Q$.
Let $(X,\F)$ be a $d$-dimensional algebraically integrable foliated pair such that:
\begin{enumerate}
    \item $X$ is potentially klt;
    \item $\F$ is log canonical;
    \item $L:=-K^{[t]}_{X,\F}:=-((1-t)tK_X+tK_{\F})$ is ample;
    \item $L^d\geq V$;
    \item $\alpha^{[t]}(X,\F;L)\geq\delta$.
\end{enumerate}
Then, the set of such pairs $(X,\F)$ is bounded.
\end{theorem}

\begin{proof}
 We fix a prime divisor $E$ over $X$, and write $a:=A^{[t]}_{X,\F}(E)$, which we assume to be at most $1$. We claim that
\[
a\ge \frac{L^d\cdot \alpha^{[t]}(X,\F;L)^d}{d^d}.
\]

Let $t_0>0$ be such that $t_0^d>\frac{d^d\,a}{L^d}$. We choose a foliated log resolution $\mu\colon Y\to X$ on which $E\subset Y$, and a very general smooth point $x\in E$ away from $\Sing(\F_Y)$ and away from the other exceptional
divisors. Let $g\colon Z\to Y$ be the weighted blowup at $x$ with weights
\[
k\left(\frac{1}{a},1,\dots,1\right),
\]
where $k$ is sufficiently divisible, and let $F$ be its exceptional divisor.

By \cite[Lemma~7.24]{Xu2025}, applied exactly as in the proof of
\cite[Theorem~7.25]{Xu2025}, there exists an effective $\Q$-divisor $D\sim_{\Q} t_0L$ such that $\ord_F(D)>kd$.

By Lemma \ref{lem:weighted-blowup-cancellation}, we have $A^{[t]}_{X,\F}(F)\leq kd$, hence $\ord_F(D)>kd \geq  A^{[t]}_{X,\F}(F)$. By the valuation-theoretic formula for the mixed $\alpha$-invariant (see Remark \ref{rem: val formula alpha}), we have
\[
\alpha^{[t]}(X,\F;L)
=
\inf_{v\in \Val_X^*}\frac{A^{[t]}_{X,\F}(v)}{T_L(v)},
\]
which implies that $t_0>\alpha^{[t]}(X,\F;L)$. Since $D\sim_{\Q}t_0L$, then $\ord_F(D)\le t_0\,T_L(F)$, while $\ord_F(D)>A^{[t]}_{X,\F}(F)$, which in turn implies that 
\[
t_0\,T_L(F)\ge \ord_F(D)>A^{[t]}_{X,\F}(F),
\]
hence
\[
t_0>\frac{A^{[t]}_{X,\F}(F)}{T_L(F)}\ge \alpha^{[t]}(X,\F;L).
\]

Since this holds for every $t_0$ satisfying $t_0^d>\frac{d^d\,a}{L^d}$, we conclude that
\[
\alpha^{[t]}(X,\F;L)^d\le \frac{d^d\,a}{L^d}.
\]
Equivalently,
\[
a\ge \frac{L^d\cdot \alpha^{[t]}(X,\F;L)^d}{d^d}.
\]

Using $L^d\geq V$ and $\alpha^{[t]}(X,\F;L)\geq\delta$, we obtain $a\geq\frac{\delta^{d}V}{d^d}$. Since $E$ was arbitrary, it follows that every divisorial valuation $v$ over $X$ satisfies
\[
A^{[t]}_{X,\F}(v)\ge \varepsilon_0,
\qquad
\varepsilon_0:=\min\left\{\frac{\delta^{d}V}{d^d},\,1\right\},
\]
which implies that the adjoint Fano foliated structure is $\varepsilon_0$-lc. Note that by Remark \ref{rem:comparison-Cas25-singularities}, this is $\varepsilon_0$-lc also in the notation of \cite{CHLMSX25}. Note by assumption (2) $\F$ is algebraically integrable and lc. Since $t<1$, after passing to a small $\Q$-factorialization of $X$, by \cite[Proposition~9.1]{CHLMSX25} we obtain that the ambient variety $X$ is $\epsilon_0$-lc. Setting $\epsilon:=\min\{\epsilon_0,t,1-t\}$, we have $\epsilon\le t\le 1-\epsilon$, and therefore all the hypotheses of \cite[Theorem~B]{CHLMSX25} are satisfied. It follows that such adjoint Fano foliated structures form a bounded family, as required.
\end{proof}

\begin{corollary}\label{cor: boundedness for K-semi}
    Fix a positive integer $d$ and a real number $\delta > 0$, and $t\in (0,1)\cap \Q$. Then the set of $d$-dimensional $t$-K-semistable adjoint Fano foliated structures $(X,\F,t)$ with $(-K_{X,\F}^{[t]})^d > \delta$, $X$ potentially klt and $\F$ lc forms a bounded family.
\end{corollary}
\begin{proof}
    Without loss of generality, we may assume $\delta < 1/(d+ 1)$. For a $t$-K-semistable adjoint Fano foliated structure $(X,\F)$ of dimension $d$, by Corollary \ref{cor: k-ss to alpha} we have $\alpha^{[t]}(X,\F) \geq 1/(d +1) > \delta$. Hence the boundedness of such structures follows immediately from Theorem \ref{thm:conditional-xu-type-boundedness}.
\end{proof}

\printbibliography
\end{document}